\documentclass{article}[utf8, 12pt]%
\usepackage{amsmath}
\usepackage{amsfonts}
\usepackage{amssymb}
\usepackage{graphicx}
\usepackage{hyperref}
\usepackage{stmaryrd}
\usepackage{mathrsfs}
\usepackage{authblk}
\usepackage{enumerate}
\usepackage{bbm}
\usepackage{color}%
\usepackage{esint}
\usepackage{mathtools}
\usepackage[title,titletoc]{appendix}
\usepackage{xpatch,amsthm}
\makeatletter
\xpatchcmd{\@thm}{\fontseries\mddefault\upshape}{}{}{} 
\makeatother
\providecommand{\U}[1]{\protect\rule{.1in}{.1in}}
\newtheorem{theorem}{Theorem}

\newtheorem{definition}[theorem]{Definition}

\newtheorem{lemma}[theorem]{Lemma}
\newtheorem{notation}[theorem]{Notation}

\newtheorem*{question*}{Question}
\newtheorem{proposition}[theorem]{Proposition}
\newtheorem{remark}[theorem]{Remark}

\begin{document}
	\author{Swarnendu Sil\thanks{swarnendusil@iisc.ac.in}}
	\affil{Department of Mathematics \\ Indian Institute of Science\\ Bangalore, India}
	\title{Topology and approximation of weak $G$-bundles in the supercritical dimensions}
	\maketitle 
	\begin{abstract}
		For analyzing stationary Yang-Mills connections in higher dimensions, one has to work with Morrey-Sobolev bundles and connections. The transition maps for a Morrey-Sobolev principal $G$-bundles are not continuous and thus the usual notion of topology does not make sense. In this work, we develop the notion of a topological isomorphism class for a bundle-connection pair $\left( P, A\right)$ and use these notions to derive several approximability results for bundles and connections in the Morrey-Sobolev setting. Our proofs follow a connection-oriented approach and also highlight the fact that in the low regularity regime, the regularity of the bundle and connection are intertwined. Our results parallels the theory of the topological degree and approximation results for manifold-valued $\mathrm{VMO}$ maps. 
	\end{abstract}
	
	\textbf{Keywords:} Morrey-Sobolev bundles, Connections, Vanishing Morrey spaces, Coulomb gauges, Topology, Approximation, Yang-Mills. \\
	
	\textbf{MSC codes:} 58E15, 53C07, 46T10. 
	\tableofcontents 
	\section{Introduction and main results}
	Throughout this article, we shall assume that $n \geq 3, N \geq 1$ are integers and 
	\begin{itemize}
		\item $1 < m \leq n/2$ is a real number 
		\item $G \subset \mathbb{O}(N)$ is a compact finite dimensional non-Abelian real Lie group with Lie algebra $\mathfrak{Lie}(G)=\mathfrak{g}$ ,
		\item $M^{n}$ is a connected, closed $n$-dimensional smooth orientable Riemannian manifold. 
	\end{itemize} 
	Here closed means a compact manifold without boundary. 
	
	\subsection{Bundles and connections with bounds on $m$-Yang-Mills energy} The main concern of mathematical analysis of classical gauge fields over Riemannian manifolds  is a principal $G$-bundle $P$ over $M^{n}$ equipped with a connection $A$ on $P$ with finite  $m$-Yang-Mills energy, i.e. $$\mathcal{YM}_{m}(A):= \int_{M^{n}} \left\lvert F_{A} \right\rvert^{m} < \infty, $$ where $F_{A}$ is the curvature corresponding to $A$ and the norm is the usual norm induced by the Killing inner product on $\mathfrak{g}$-valued forms. When $m=2,$ the $m$-Yang-Mills energy is simply the Yang-Mills energy.  
	
	\par As is often the case, the crux of the matter boils down to studying a sequence of  pairs $\left\lbrace \left( P^{s}, A^{s} \right) \right\rbrace_{s \in \mathbb{N}},$ where for each $s \in \mathbb{N},$ $P^{s}$ is a principal $G$-bundle over $M^{n}$ equipped with connection $A_{s}$ on $P^{s},$ such that the $m$-Yang-Mills energy is uniformly bounded, i.e. 
	$$ \limsup\limits_{s \rightarrow \infty} \mathcal{YM}_{m}(A^{s}) < \infty. $$ The main questions are the following.  
	\begin{question*}
		Can we extract a subsequence which converges to a `limit bundle' equipped with a `limit connection'?
	\end{question*} 
	\begin{question*}\label{question of topology in the limit}
		Assuming all the bundles $P^{s}$ are topologically isomorphic to a fixed topological bundle $P,$ is the `limit bundle' ( assuming one such bundle indeed exists ) is topologically the `same' or not? 
	\end{question*}

	\par When $m > n/2,$ which is usually called the \emph{subcritical dimensions}, both questions have affirmative answers, as shown by Uhlenbeck in her pioneering work \cite{UhlenbeckGaugefixing} for $m=2,$ but her techniques generalize to any $m > n/2.$ In the same work, she has shown that for $m=2$ and $n=4,$ i.e., when $m=n/2,$ called the \emph{critical dimension}, it is possible to talk about convergence to `limit bundle' and a `limit connection', but only away from finitely many `singular points' ( see also Sedlacek \cite{Sedlacek_YangMills} ), even when the sequence  $\left\lbrace \left( P^{s}, A^{s} \right) \right\rbrace_{s \in \mathbb{N}}$ are all smooth Yang-Mills connections on smooth bundles. Later in \cite{Uhlenbeck_removable_singularity}, Uhlenbeck has shown that these `singularities are removable' in the Sobolev sense, meaning it is possible to obtain a Sobolev bundle of class $W^{2,2}$ as a `limit bundle.' However, since $W^{2,2}$ maps are not necessarily continuous in four dimensions, these `bundles' are not necessarily topological bundles. Same conclusion holds for any $1 < m = n/2.$ The situation is even worse when $m < n/2,$ usually called the \emph{supercritical dimensions} as the $m$-Yang-Mills energy is too weak to even ensure the existence of a `limit Sobolev bundle' in this regime. 
	
	\par However, if the sequence $\left\lbrace \left( P^{s}, A^{s} \right) \right\rbrace_{s \in \mathbb{N}}$ are all smooth Yang-Mills connections on smooth bundles, then the connections are also \emph{stationary Yang-Mills connections}. Using a fundamental monotonicity formula due to Price \cite{Price_monotonicity}, Nakajima \cite{Nakajima_YangMills_compactness} showed that when $m=2,$ weak convergence of curvatures can be achieved, but only away from a `singular set' of codimension $4$. Tian \cite{Tian_GaugetheoryCalibratedgeometry} once again used the monotonicity to show that convergence to `limit bundle' and a `limit connection' can be achieved away from a `singular set' of codimension $4$ and proved improved results concerning the singular set. The main point here is via the monotonicity formula, uniform bound on the Yang-Mills energy implies a uniform bound on a Morrey norm of the curvatures, namely the $\mathrm{L}^{2,n-4}$ norm. A `removable singularity' result, which concludes the removability of only the `codimension four strata'  of the singular set and is considerably harder to prove in this case, has also been established in Tao-Tian \cite{TaoTian_YangMills} and Meyer-Rivi\`{e}re \cite{MeyerRiviere_YangMills}.  Energy identity and bubbling have been studied in Rivi\`{e}re \cite{Riviere_QuantizationYangMills} and Naber-Valtorta \cite{NaberValtorta_YangMillsenergyidenty}. 
	
	\subsection{Topology without continuity}
	
	The upshot of the discussion above is that in the critical and supercritical dimensions, i.e. the cases when $1 < m \leq n/2,$ the `limit bundles' are Sobolev or Morrey-Sobolev bundles, possibly with singularities, \textbf{which are not in general, topological bundles}. Hence one might be tempted to believe that Question \ref{question of topology in the limit} does not even make sense. 
	
	\par However, that is not necessarily the case. An analogous situation exists for manifold-valued Sobolev maps. As is well-known, the homotopy classes of continuous maps $u:\mathbb{S}^{n} \rightarrow \mathbb{S}^{n}$ are classified by their \emph{degree}. In a number of geometric variational problems, one is forced to consider Sobolev maps $u \in W^{1,p}\left( \mathbb{S}^{n}; \mathbb{S}^{n}\right).$
	When $p >n,$ i.e. the \emph{subcritical dimension} case, these maps are continuous by Sobolev embedding and the usual topological result holds. When $p \leq n,$ i.e. the \emph{critical and supercritical dimension} case, a priori, these maps can be discontinuous. However, motivated by the seminal work of Schoen-Uhlenbeck \cite{SchoenUhlenbeck_boundaryregularityharmonicmaps}, Brezis-Nirenberg showed in \cite{BrezisNirenberg_degree1} that maps in $W^{1,n}\left( \mathbb{S}^{n}; \mathbb{S}^{n}\right)$  indeed have a well-defined notion of a \emph{degree}. In fact, their results are actually stronger. They showed that 
	\begin{enumerate}[(i)]
		\item maps in $\mathrm{VMO}\left( \mathbb{S}^{n}; \mathbb{S}^{n}\right)$ has a well-defined notion of topological degree and 
		\item for any $u \in \mathrm{VMO}\left( \mathbb{S}^{n}; \mathbb{S}^{n}\right),$ there exists a $\delta >0,$ possibly depending on $u,$ such that any map $v \in \mathrm{VMO}\left( \mathbb{S}^{n}; \mathbb{S}^{n}\right)$ in the $\delta$-neighborhood of $u$ in the $\mathrm{BMO}$ has the same degree as $u.$
	\end{enumerate}
	Their basic argument is that the degree is continuous with respect to convergence in the $\mathrm{BMO}$ norm and thus, by approximating $\mathrm{VMO}$ maps by smooth maps in the $\mathrm{BMO}$ norm, the topological information survives in the limit, although the limiting maps are no longer necessarily continuous. The last result follows from the stability of $\mathrm{VMO}$ degree under small perturbations in the $\mathrm{BMO}$ distance.  
	
	\subsection{Main results} 
	The main contribution of the present work are analogues of the above-mentioned results to the setting of Morrey-Sobolev bundles and connections. In \cite{Sil_YangMillscriticaltoappear}, the critical dimension case was analyzed ( see Isobe \cite{Isobe_Sobolevbundlecriticaldimension} and Shevchisin \cite{Shevchishin_limitholonomyYangMills} for earlier attempts ) and it was shown that for a pair $\left( P, A\right),$ where $P$ is a Sobolev bundle of class $W^{2,n/2}$ and $A$ is an $W^{1,n/2}$ connection on $P$,  
	\begin{enumerate}[(i)]
		\item we can assign a well-defined topological isomorphism class and 
		\item can be  approximated in the corresponding Sobolev norm topologies by smooth connections on smooth bundles in an appropriate sense. 
	\end{enumerate}
	This result can be viewed as the analogue of `$W^{1,n}$-degree' result in the present setting. However, there is an important difference from the case of the degree. In the case of bundles and connections, the topological isomorphism class is \textbf{not} defined by the approximation. In fact, in some sense, the scenario is reverse here, where one first shows the existence of a topological isomorphism class to deduce the approximability result. 
	
	The natural space to work with in the supercritical dimensions should be Morrey-Sobolev spaces, as these are precisely the norms which are scale invariant and a bound for these norms are implied by the monotonicity formula for stationary connections. The trouble is that even in the case of scalar valued functions, smooth functions are \textbf{not dense} in Morrey-Sobolev spaces. This situation is again similar to case of the degree, as smooth functions are not dense in $\mathrm{BMO}$ and $\mathrm{VMO}$ is precisely the proper subspace of $\mathrm{BMO}$ where smooth functions are dense. So one might expect an analogue in our situation. Roughly speaking, our first main result is precisely such an analogue.  
	\begin{theorem}\label{Vanishing Morrey case approx and topo}
		Let $P$ be a principal $G$-bundle over a $n$-dimensional closed manifold $M^{n},$ such that the bundle transition functions are in the vanishing Morrey-Sobolev class $\mathsf{W}^{2}\mathrm{VL}^{2m,n-2m}$ and let $A$ be a connection on $P$ such that the local connection forms are in vanishing Morrey-Sobolev class $\mathsf{W}^{1}\mathrm{VL}^{m, n-2m}.$ Then the following holds. 
		
		\begin{enumerate}[(i)]
			\item $\left(P,A\right)$ is gauge equivalent to $\left(P_{C}, A_{C}\right)$ via $\mathsf{W}^{2}\mathrm{VL}^{m, n-2m}$ gauge transformations, where $P_{C}$ is a $C^{0}$ bundle and $A_{C}$ is a $\mathsf{W}^{1}\mathrm{VL}^{m, n-2m}$ connection which is Coulomb. Moreover, the $C^{0}$ isomorphism class of $P_{C}$ invariant under $\mathsf{W}^{1}\mathrm{VL}^{m, n-2m}$ gauge transformation of the \emph{pair} $\left( P, A \right).$  
			\item There exists a sequence of smooth bundle-connections pairs $\left\lbrace \left( P^{s}, A^{s}\right)\right\rbrace_{s \in \mathbb{N}}$ on $M^{n}$ such that 
			\begin{enumerate}[(a)]
				\item $P^{s}$ is gauge-equivalent to $P$ via $\mathsf{W}^{2}\mathrm{VL}^{m, n-2m}$ gauges $\sigma^{s}$ for each $s \in \mathbb{N},$
				
				\item  the smooth bundle transition functions $g_{ij}^{s}$ of $P^{s}$ converges to the bundle transition functions $g_{ij}$ of $P$ in the  $\mathsf{W}^{2}\mathrm{L}^{m, n-2m}$ norm, locally on $M^{n},$ i.e.  
				$$ g_{ij}^{s} \rightarrow g_{ij} \qquad \text{ locally in } \mathsf{W}^{2}\mathrm{L}^{m, n-2m} \text{ for every } i,j, $$  and 
				the the pulled back connections $\left( \sigma_{s}^{-1}\right)^{\ast}A^{s}$ converges to the connection forms for $A$ locally in $\mathsf{W}^{1}\mathrm{L}^{m, n-2m}$ norm, i.e. 
				$$ \left( \left( \sigma^{s}\right)^{-1}\right)^{\ast}A^{s}_{i} \rightarrow A_{i} \qquad \text{ locally in } \mathsf{W}^{1}\mathrm{L}^{m, n-2m} \text{ for every } i. $$ 
			\end{enumerate} 
		\end{enumerate}
	\end{theorem}
	The vanishing Morrey-Sobolev spaces are a subspace of Morrey-Sobolev spaces where smooth functions are dense and serves as the correct analogue of $\mathrm{VMO}$ in our setting. However, density of $G$-valued smooth maps in the class of $G$-valued vanishing Morrey-Sobolev maps does not immediately yield the approximation result, since here one needs to approximate cocycles by smooth cocycles, i.e. preserve the pointwise cocycle condition as well.

	\par Another striking feature of our case is that the isomorphism class is assigned to the \emph{bundle-connection pair} and not to the bundles alone. This is in stark contrast to the usual topological theory of principal bundles, where the topology of the bundle is independent of the connection and one can use \emph{any smooth connection on the bundle} to compute the Chern classes of the underlying topological bundle via the Chern-Weil theory.  The reason for this discrepancy is the fact that a smooth ( or regular enough ) connection  can only `live' on one and only one ( up to $C^0$ isomorphism ) topological bundle. We prove here that if we specialize to the case of a smooth connection on a topological bundles, our topological isomorphism class for the \emph{bundle-connection pair} is the same as the usual topological class of the underlying bundle. However, this above fact is no longer true for connections with low regularity. A  connection with critical Morrey-Sobolev regularity can `live' on \textbf{more than one distinct} topological bundles, whose Chern classes differ from each other. This does not contradict the Chern-Weil theory, as the regularity of the connection is too low for the theory to apply. 
	
	When the bundle and the connection is merely Morrey-Sobolev and not vanishing Morrey-Sobolev, clearly approximation in the Morrey-Sobolev norms are impossible. However, one might hope to approximate by smooth bundle-connection pairs in the Sobolev norms. Our next result shows this can be done if the norm of the connection is sufficiently small, with the smallness parameter depending on the cover.   
	\begin{theorem}\label{small in Morrey}
		Let $P = \left( \left\lbrace U_{i}\right\rbrace_{i \in I}, \left\lbrace g_{ij}\right\rbrace_{i,j \in I}\right)$ be a principal $G$-bundle over a $n$-dimensional closed manifold $M^{n},$ such that the bundle transition functions $\left\lbrace g_{ij}\right\rbrace$s are in the Morrey-Sobolev class $\mathsf{W}^{2}\mathrm{L}^{m, n-2m}$ and let $A$ be a connection on $P$ such that the local connection forms $\left\lbrace A_{i}\right\rbrace $s are in Morrey-Sobolev class $\mathsf{W}^{1}\mathrm{L}^{m, n-2m}.$ Then there exists a number $\delta >0,$ depending on the cover  $\left\lbrace U_{i}\right\rbrace_{i \in I}$, $n, m, M^{n}$ and $G$ such that if   
		\begin{align*}
			\sup\limits_{i} \left\lVert A_{i} \right\rVert_{\mathsf{W}^{1}\mathrm{L}^{m, n-2m}} &\leq \delta, 
		\end{align*}
		Then the following holds. 
		
		\begin{enumerate}[(i)]
			\item $\left(P,A\right)$ is gauge equivalent to $\left(P_{C}, A_{C}\right)$ via $\mathsf{W}^{2}\mathrm{L}^{m, n-2m}$ gauge transformations, where $P_{C}$ is a $C^{0}$ bundle and $A_{C}$ is a $\mathsf{W}^{1}\mathrm{L}^{m, n-2m}$ connection which is Coulomb.    
			
			\item There exists a sequence of smooth bundle-connections pairs $\left\lbrace \left( P^{s}, A^{s}\right)\right\rbrace_{s \in \mathbb{N}}$ on $M^{n}$ such that 
			\begin{enumerate}[(a)]
				\item $P^{s}$ is gauge-equivalent to $P$ via $\mathsf{W}^{2}\mathrm{L}^{m, n-2m}$ gauges $\rho_{s}$ for each $s \in \mathbb{N},$ 
				
				\item  the smooth bundle transition functions $g_{ij}^{s}$ of $P^{s}$ is uniformly bounded in $\mathsf{W}^{2}\mathrm{L}^{m, n-2m}$ and converges to the bundle transition functions $g_{ij}$ of $P$ in the  $W^{2,m}$ and $W^{1,2m}$ norm, locally on $M^{n},$ i.e.  
				\begin{align*}
					\limsup\limits_{s \rightarrow \infty }\left\lVert g_{ij}^{s}\right\rVert_{\mathsf{W}^{2}\mathrm{L}^{m, n-2m}} \leq C \left\lVert g_{ij}\right\rVert_{\mathsf{W}^{2}\mathrm{L}^{m, n-2m}}, \\
					g_{ij}^{\varepsilon} \rightarrow g_{ij} \qquad \text{ locally in } W^{2,m} \text{ and } W^{1,2m} \text{ for every } i,j,
				\end{align*}
				and the local connection forms for the pulled back connections $\left( \rho_{\varepsilon}^{-1}\right)^{\ast}A^{\varepsilon}$ is uniformly bounded in $\mathsf{W}^{1}\mathrm{L}^{m, n-2m}$ and  converges to the local connection forms for $A$ in $W^{1, m}$ and $L^{2m}$ norm, i.e. 
				\begin{align*}
					\limsup\limits_{s \rightarrow \infty }\left\lVert \left( \rho_{s}^{-1}\right)^{\ast}A^{s}_{i}\right\rVert_{\mathsf{W}^{1}\mathrm{L}^{m, n-2m}} \leq C \left\lVert A_{i}\right\rVert_{\mathsf{W}^{1}\mathrm{L}^{m, n-2m}}, \\
					\left( \rho_{s}^{-1}\right)^{\ast}A^{s}_{i} \rightarrow A_{i} \qquad \text{ locally in } W^{1, m} \text{ and  } L^{2m} \text{ for every } i.
				\end{align*}
				
			\end{enumerate} 
	\end{enumerate}  \end{theorem}	
	
	Since we would show in Theorem \ref{existence of connection} that every $\mathsf{W}^{2}\mathrm{L}^{m, n-2m},$ respectively $\mathsf{W}^{2}\mathrm{VL}^{m, n-2m},$ bundle admits a $\mathsf{W}^{1}\mathrm{L}^{m, n-2m},$ respectively $\mathsf{W}^{1}\mathrm{VL}^{m, n-2m},$ connection, Theorem \ref{Vanishing Morrey case approx and topo} and Theorem \ref{small in Morrey} immediately implies the following approximability result for cocycles. As far as we are aware, this is the first cocycle approximation results in the Morrey-Sobolev setting. 
	
	\begin{theorem}\label{approximation of cocycles}
		Let $P = \left( \left\lbrace U_{\alpha}\right\rbrace_{\alpha \in I}, \left\lbrace g_{\alpha\beta}\right\rbrace_{\alpha, \beta \in I}\right)$ be a $\mathsf{W}^{2}\mathrm{L}^{m, n-2m}$ principal $G$-bundle over a $n$-dimensional closed manifold $M^{n}.$ \begin{enumerate}[(i)]
			\item If $P$ is a $\mathsf{W}^{2}\mathrm{VL}^{m, n-2m}$ bundle, then there exists \textbf{at least one}  sequence of smooth bundles $\left\lbrace  P^{s} \right\rbrace_{s \in \mathbb{N}}$ on $M^{n}$ such that $P^{s}$ is gauge-equivalent to $P$ via $\mathsf{W}^{2}\mathrm{VL}^{m, n-2m}$ gauges for each $s \in \mathbb{N}$ and the smooth bundle transition functions of $P^{s}$ converges to the bundle transition functions  of $P$ in the  $\mathsf{W}^{2}\mathrm{L}^{m, n-2m}$ norm, locally on $M^{n}.$ 
			\item There exists a $\delta >0,$ depending on the cover  $\left\lbrace U_{\alpha}\right\rbrace_{\alpha \in I}$, $n, m, M^{n}$ and $G$ such that if \begin{align*}
				\sup\limits_{\substack{\alpha, \beta \in I,\\U_{\alpha}\cap U_{\beta} \neq \emptyset}} \left\lVert g_{\alpha\beta} \right\rVert_{\mathsf{W}^{2}\mathrm{L}^{m, n-2m}} \leq \delta, 
			\end{align*}
			then there exists \textbf{at least one}  sequence of smooth bundles $\left\lbrace  P^{s} \right\rbrace_{s \in \mathbb{N}}$ on $M^{n}$ such that $P^{s}$ is gauge-equivalent to $P$ via $\mathsf{W}^{2}\mathrm{L}^{m, n-2m}$ gauges for each $s \in \mathbb{N}$ and the smooth bundle transition functions of $P^{s}$ converges to the bundle transition functions  of $P$ in the  $W^{2,m}$ and $W^{1,2m}$ norm, locally on $M^{n}.$
		\end{enumerate}
	\end{theorem}
	
	The analogue of the stability of degree in our quite nonlinear setting is however a different matter. In this work, we establish such results around a flat bundle with a flat connection \text{ when $\sqrt{n}/2 < m \leq n/2$}. To do this, we first establish a result which is probably of an interest in itself.  This asserts that the for a sequence of bundle-connection pairs in the vanishing Morrey-Sobolev class, the topological isomorphism class must stabilize if the curvatures do not concentrate in Morrey-Sobolev norms. 	
	\begin{theorem}\label{equivanishing Morrey}
		Suppose $\frac{\sqrt{n}}{2} < m \leq \frac{n}{2}.$ Let $\left\lbrace P^{s} \right\rbrace_{s \in \mathbb{N}}$ be a sequence of $\mathsf{W}^{2}\mathrm{VL}^{m, n-2m}$ bundles such that there exists a common refinement $\left\lbrace U_{\alpha}\right\rbrace_{\alpha \in I}$ for the associated covers. Let	$ A^{s}$ be a $\mathsf{W}^{1}\mathrm{VL}^{m, n-2m}$ connection on $P^{s}$ for all $s \in \mathbb{N}$ such that the local curvature forms are uniformly bounded in $\mathrm{L}^{m, n-2m}$ and equivanishing in $\mathrm{VL}^{m, n-2m}.$ More precisely, we have 
		\begin{align*}
			\sup\limits_{\alpha \in I} \left\lVert F_{A^{s}_{\alpha}}\right\rVert_{\mathrm{L}^{m, n-2m}\left( U_{\alpha}; \Lambda^{2}\mathbb{R}^{n}\otimes \mathfrak{g} \right)} &\leq \Lambda \qquad \text{ for every } s \in \mathbb{N}, \\
			\sup\limits_{\alpha \in I} \sup\limits_{\substack{0 < \rho < r, \\ B_{\rho}(x) \subset \subset U_{\alpha}, }} 
			\frac{1}{\rho^{n-2m}} \int_{B_{\rho}(x)} \lvert F_{A^{s}_{\alpha}} \rvert^{m} &\leq 	\Theta\left( r \right) \qquad \text{ for every } s \in \mathbb{N}, 
		\end{align*}
		for some constant $\Lambda >0$ and some function $\Theta \left(r\right)$ such that 
		$$ \Theta \left(r\right) \rightarrow 0 \qquad \text{ as } r \rightarrow 0.$$
		Then there exists a subsequence $\left\lbrace \left( P^{s_{\mu}}, A^{s_{\mu}} \right) \right\rbrace_{\mu \in \mathbb{N} },$ an integer $\mu_{0} \in \mathbb{N},$ a sequence $\mathsf{W}^{2}\mathrm{VL}^{m, n-2m}$ gauges $\left\lbrace \sigma_{s_{\mu}}\right\rbrace_{\mu \in \mathbb{N} },$ a refinement $\left\lbrace V_{i}\right\rbrace_{i \in J}$ of $\left\lbrace U_{\alpha}\right\rbrace_{\alpha \in I}$, a $C^{0}$ bundle $P^{\infty}$ trivialized over $\left\lbrace V_{i}\right\rbrace_{i \in J}$ and a $\mathsf{W}^{1}\mathrm{VL}^{m, n-2m}$ connection $A^{\infty}$ on $P^{\infty}$ such that 
		\begin{align*}
			P^{s_{\mu}} \simeq_{\sigma_{s_{\mu}}}  P^{\infty}  \quad \text{ and } \quad \left[ P^{s_{\mu}}_{C}\right]_{C^{0}} = \left[ P^{\infty}\right]_{C^{0}} \qquad \text{ for every } \mu \geq \mu_{0},
		\end{align*}
		where $P^{s_{\mu}}_{C}$ denotes the Coulomb bundle associated to $\left( P^{s_{\mu}}, A^{s_{\mu}} \right) $ and we have 
		\begin{align*}
			\left\lVert F_{A^{\infty}_{i}}\right\rVert_{\mathrm{L}^{m, n-2m}} &\leq  \liminf\limits_{\mu \rightarrow \infty }\left\lVert F_{A^{s_{\mu}}_{i}}\right\rVert_{\mathrm{L}^{m, n-2m}} \intertext{ and }
			\left( \sigma_{s_{\mu}}^{-1}\right)^{\ast}A^{s_{\mu}}_{i} &\rightharpoonup A^{\infty}_{i} \qquad \text{ weakly in } W^{1, m}\text{ and  } L^{2m} \text{ for every } i \in J.
		\end{align*}
	\end{theorem}	
	This immediately implies a flatness result, which is an analogue of the stability result around the flat connection on a flat bundle in our setting. 
	\begin{theorem}\label{flatness in vanishing}
		Suppose $\frac{\sqrt{n}}{2} < m \leq \frac{n}{2}.$. For any cover $\left\lbrace U_{\alpha}\right\rbrace_{\alpha \in I}$ of $M^{n},$ there exists a constant $\delta >0,$ depending only on $\left\lbrace U_{\alpha}\right\rbrace_{\alpha \in I}$, $M^{n}$, $G$, $m, n$ such that if $P$ is a $\mathsf{W}^{2}\mathrm{VL}^{m, n-2m}$ bundle trivialized over $\left\lbrace U_{\alpha}\right\rbrace_{\alpha \in I}$ and $A$ is a $\mathsf{W}^{1}\mathrm{VL}^{m, n-2m}$ connection on $P$  such that if 
		\begin{align*}
			\sup\limits_{\alpha \in I}\left\lVert F_{A_{\alpha}}\right\rVert_{\mathrm{L}^{m, n-2m}\left( U_{\alpha}; \Lambda^{2}\mathbb{R}^{n}\otimes \mathfrak{g} \right)} < \delta, 
		\end{align*} then the Coulomb bundle associated to $\left(P, A\right)$ is flat and thus $P$ is $\mathsf{W}^{2}\mathrm{VL}^{m, n-2m}$ gauge equivalent to a flat bundle. 
	\end{theorem} 
	
	For Morrey-Sobolev bundles, we have an analogous result assuming smallness for the norm of the connection, not just the curvature. 
	\begin{theorem}\label{small distance to flat}
		Suppose $\frac{\sqrt{n}}{2} < m \leq \frac{n}{2}.$ For any cover $\left\lbrace U_{\alpha}\right\rbrace_{\alpha \in I}$ of $M^{n},$ there exists a constant $\delta >0,$ depending only on $\left\lbrace U_{\alpha}\right\rbrace_{\alpha \in I}$, $M^{n}$, $G$, $m, n$ such that if $P$ is a $\mathsf{W}^{2}\mathrm{L}^{m, n-2m}$ bundle trivialized over $\left\lbrace U_{\alpha}\right\rbrace_{\alpha \in I}$ and $A$ is a $\mathsf{W}^{1}\mathrm{VL}^{m, n-2m}$ connection on $P$  such that if 
		\begin{align*}
			\sup\limits_{\alpha \in I}\left\lVert A_{\alpha}\right\rVert_{\mathsf{W}^{1}\mathrm{L}^{m, n-2m}\left( U_{\alpha}; \Lambda^{1}\mathbb{R}^{n}\otimes \mathfrak{g} \right)} < \delta, 
		\end{align*} then the Coulomb bundle associated to $\left(P, A\right)$ is flat and thus $P$ is $\mathsf{W}^{2}\mathrm{L}^{m, n-2m}$ gauge equivalent to a flat bundle.
	\end{theorem}  
	
	As a consequence of the last two results, we deduce a general result about factorizing a Morrey-Sobolev cocycle. 
	\begin{theorem}\label{cocycle factorization MorreySobolev}
		Suppose $\frac{\sqrt{n}}{2} < m \leq \frac{n}{2}.$ Let $M^{n}$ be a \textbf{simply connected} closed manifold.  For any cover $\left\lbrace U_{\alpha}\right\rbrace_{\alpha \in I}$ of $M^{n},$ there exists a constant $\delta >0,$ depending only on $\left\lbrace U_{\alpha}\right\rbrace_{\alpha \in I}$, $M^{n}$, $G$, $m, n$ such that if $P = \left( \left\lbrace U_{\alpha}\right\rbrace_{\alpha \in I}, \left\lbrace g_{\alpha\beta}\right\rbrace_{\alpha, \beta \in I}\right)$ is a $\mathsf{W}^{2}\mathrm{L}^{m, n-2m}$ bundle trivialized over $\left\lbrace U_{\alpha}\right\rbrace_{\alpha \in I}$ and 
		\begin{align*}
			\sup\limits_{\substack{\alpha, \beta \in I,\\ U_{\alpha}\cap U_{\beta} \neq \emptyset}}  	\left\lVert  dg_{\alpha\beta}\right\rVert_{\mathsf{W}^{1}\mathrm{L}^{m, n-2m}\left( U_{\alpha}\cap U_{\beta}; \Lambda^{1}\mathbb{R}^{n}\otimes\mathfrak{g}\right)}  < \delta, 
		\end{align*}
		then there exists a refined cover $\left\lbrace V_{i}\right\rbrace_{i \in J}$ with the refinement map $\phi: J \rightarrow I$ such that for each $i \in J,$ there exists a map $\sigma_{i} \in \mathsf{W}^{2}\mathrm{L}^{m, n-2m}\left( V_{i}; G \right)$ satisfying 
		\begin{align*}
			\mathbbm{1}_{G} = \sigma_{i}^{-1}g_{\phi\left(i\right)\phi\left(i\right)}\sigma_{j} \qquad \text{ in } V_{i}\cap V_{j},
		\end{align*}
		whenever $V_{i}\cap V_{j} \neq \emptyset.$ Furthermore, if $P$ is a $\mathsf{W}^{2}\mathrm{VL}^{m, n-2m}$ bundle, then we can choose $\sigma_{i} \in \mathsf{W}^{2}\mathrm{VL}^{m, n-2m}\left( V_{i}; G \right)$ as well. 
	\end{theorem} 
	The result basically says that for any cover of a simply connected manifold, there exists a $\delta >0$ such that any $\mathsf{W}^{2}\mathrm{L}^{2m,n-2m}$ cocycle   in the `$\delta$-neighborhood' of the identically identity cocycle is a $\mathsf{W}^{2}\mathrm{L}^{2m,n-2m}$ coboundary. This result holds for any cocycle, rather than just the identity cocycle, in the subcritical situation ( see Theorem \ref{subcritical cocycle factorization} ) and is essentially established by Uhlenbeck ( Proposition 3.2 and Corollary 3.3. in \cite{UhlenbeckGaugefixing} ). To the best of our knowledge, our result is new even in the critical case and is a first result of this kind beyond the subcritical setting.

	\subsection{Coulomb gauges and continuity of the Coulomb bundles} The basic point underlying all of our results is a regularity result for Coulomb bundles, i.e. bundles in which the local connection forms $A_{i}$ satisfy the Coulomb condition $$d^{\ast}A_{i} = 0 \qquad \text{ in } U_{i}.$$ That $W^{1,n}$ Coulomb bundles are continuous is first established for $n=4$ in Taubes \cite{Taubes_YangMillsModulispaces} and for general $n$ in Rivi\`{e}re \cite{Riviere_QuantizationYangMills}. In fact, $W^{1,n}$ Coulomb bundles are not just continuous, they are $C^{0, \alpha}$ bundles for any H\"{o}lder exponent $\alpha < 1$ is shown in Shevchishin \cite{Shevchishin_limitholonomyYangMills} and Sil \cite{Sil_YangMillscriticaltoappear}.  
	
	\par Here we extend the techniques of \cite{Sil_YangMillscriticaltoappear} to prove that the H\"{o}lder continuity result still holds even in the supercritical dimensions as long as the bundles have the critical Morrey-Sobolev regularity. As far as we are aware, this is new and is the first result about continuity of the Coulomb bundles in the supercritical situation and should be of independent interest as well.  This also shows that for regularity of Coulomb bundles, the only issue that matters is the scale invariance of the norm.

	\par The next issue is the existence of Coulomb bundles. The existence of local Coulomb gauges in the Morrey-Sobolev setting is a subtle issue due to the fact that smooth functions are not dense in Morrey spaces. In \cite{MeyerRiviere_YangMills}, existence of local Coulomb gauges were proved for `approximable'  $\mathsf{W}^{1}\mathrm{L}^{2, n-4}$ connections, using an improved Gagliardo-Nirenberg type interpolation inequality for the non-borderline case. Theorem \ref{existence of Coulomb gauges in better spaces} generalizes their result to any $1 < m \leq n/2$ with a new proof, which do not use the interpolation inequality. Our Theorem \ref{existence of Coulomb gauges in critical Morrey-Sobolev case} essentially establishes that the vanishing Morrey condition is a natural sufficient condition for `approximability' in their sense. On the other hand, Theorem \ref{existence_Coulombgauges_small_connection} shows `approximability' is not necessary for local Coulomb gauges to exists. The price to pay for this generality is that the smallness parameters are no longer scale-invariant and the norm of the curvature alone does not control the norm of the connection even in the Coulomb gauge. This of course, severely restricts the efficacy of these gauges for problems involving only a bound on the Yang-Mills energy. However, by keeping careful track of the scale-invariance ( or the lack thereof ) of our constants, we show that even these weaker estimates still lead to new results like the cocycle factorization in Theorem \ref{cocycle factorization MorreySobolev}.    
	
	\subsection{Organization of the article}
	
	The rest of the article is organized as follows. In Section \ref{notations}, we collect the preliminary materials about Sobolev and Morrey-Sobolev principal $G$ bundles and the function spaces that we shall use. Section \ref{Prep estimates} derives the main estimates. Section \ref{topology} gives the definition of the topological isomorphism class and discusses some issues to put matters into context. Section \ref{Proof Main results} proves our main results.

	\section{Notations and preliminaries}\label{notations}
	\subsection{Function spaces}
	\subsubsection{Morrey spaces and Vanishing Morrey spaces}
	Let $\Omega \subset \mathbb{R}^{n}$ be an open set such that for any $x \in \Omega$ and any $0 < r < \operatorname{diam}\left(\Omega\right)$, there exists a constant $A,$ independent of $x$ and $r$, such that $$ \left\lvert B_{r}(x) \cap \Omega\right\rvert \geq A r^{n}.$$ For $1\leqslant p <  \infty$ and $ 0 \leq \lambda \leq n,$ we use the notation   $\mathrm{L}^{p,\lambda}\left(\Omega \right) $ for the Morrey space of all scalar function $u \in L^{p}\left(\Omega \right)$ such that 
	$$ \lVert u \rVert_{\mathrm{L}^{p,\lambda}\left(\Omega\right)}^{p} := \sup_{\substack{ x \in \Omega,\\ 0 < r < \operatorname{diam}\left(\Omega\right) }} 
	\frac{1}{r^{\lambda}} \int_{B_{r}(x) \cap \Omega} \lvert u \rvert^{p} < \infty, $$ endowed with the norm 
	$ \lVert u \rVert_{\mathrm{L}^{p,\lambda}\left(\Omega\right)}.$ For $\lambda = 0,$ these spaces are just the usual $L^{p}$ spaces with equivalent norms and for $\lambda = n,$ we have $\mathrm{L}^{p,n} \simeq L^{\infty}.$ However, we shall be concerned with $0 \leq \lambda < n.$\bigskip 
	
	\noindent For $u \in \mathrm{L}^{p,\lambda}\left(\Omega \right),$  we define the $\mathrm{L}^{p,\lambda}$ modulus of $u,$ denoted $\Theta^{u}_{p, \lambda, \Omega}$ as the function $\Theta^{u}_{p, \lambda, \Omega}: (0, \operatorname{diam}\left(\Omega\right) ] \rightarrow \mathbb{R}$ given by 
	\begin{align*}
		\Theta^{u}_{p, \lambda, \Omega}\left( r \right): = \left( \sup_{\substack{ x \in \Omega,\\ 0 < \rho < r }} 
		\frac{1}{\rho^{\lambda}} \int_{B_{\rho}(x) \cap \Omega} \lvert u \rvert^{p}\right)^{\frac{1}{p}} \qquad \text{ for }  0 < r < \operatorname{diam}\left(\Omega\right). 
	\end{align*}
	We define the \emph{Vanishing Morrey space} 
	$\mathrm{VL}^{p, \lambda}\left(\Omega\right)$ as 
	\begin{align*}
		\mathrm{VL}^{p, \lambda}\left(\Omega\right) : = \left\lbrace u \in \mathrm{L}^{p,\lambda}\left(\Omega \right): \lim\limits_{r \rightarrow 0} \Theta^{u}_{p, \lambda, \Omega}\left( r \right) = 0 \right\rbrace.  
	\end{align*}
	For $\lambda = 0,$ $\mathrm{L}^{p, \lambda} =\mathrm{VL}^{p, \lambda},$ but $\mathrm{VL}^{p, \lambda} \subsetneq \mathrm{L}^{p, \lambda}$ as soon as $\lambda > 0.$
	For any finite dimensional vector space $X,$ these definitions extend componentwise in the obvious manner for $X$-valued functions and the corresponding spaces will be denoted by $\mathrm{L}^{p,\lambda}\left(\Omega; X \right)$ and $\mathrm{VL}^{p, \lambda}\left(\Omega; X \right),$ respectively. 
	
	For convenience, we also record the H\"{o}lder inequality in Morrey spaces, which can be proved by simply applying the usual H\"{o}lder inequality. 
	\begin{proposition}[H\"{o}lder inequality in Morrey spaces]\label{holderMorrey}
		Let $0 < \lambda, \mu < n$ and $1 < p, q < \infty$ be real numbers such that $\frac{1}{p} + \frac{1}{q} \leq 1.$ Then for any $f \in \mathrm{L}^{p,\lambda}\left(\Omega\right)$ and any $g \in \mathrm{L}^{q,\mu}\left(\Omega\right),$ we have $fg \in \mathrm{L}^{r,\nu}$ and the estimate 
		$$\lVert fg \rVert_{\mathrm{L}^{r,\nu}\left(\Omega\right)}  \leq \lVert f \rVert_{\mathrm{L}^{p,\lambda}\left(\Omega\right)}\lVert g \rVert_{\mathrm{L}^{q,\mu}\left(\Omega\right)},$$
		where $\frac{1}{r} = \frac{1}{p} + \frac{1}{q}$ and $\frac{\nu}{r} = \frac{\lambda}{p} + \frac{\mu}{q}. $
	\end{proposition}\bigskip 
	
	\noindent Vanishing Morrey spaces have been introduced in the context of elliptic PDEs in \cite{Vitanza_VanishingMorrey1} ( see also \cite{ChiarenzaFranciosi_vanishingMorrey} and \cite{Vitanza_VanishingMorrey2} ). The reason why the spaces $\mathrm{VL}^{p, \lambda}$ are better behaved in comparison to $\mathrm{L}^{p,\lambda}$ spaces is that as soon as $ 0 < \lambda < n,$ a general $f \in \mathrm{L}^{p,\lambda}\left(\Omega\right)$ can not be approximated, not even locally, by smooth functions in the $\mathrm{L}^{p,\lambda} $ norm. Indeed, we have the following simple counter example due to Zorko. For any $x_{0} \in \Omega,$ define the function $$ f\left( x \right) = \frac{1}{\left\lvert x - x_{0}\right\rvert^{\frac{n-\lambda}{p}}} \qquad \text{ for a.e. } x \in \Omega. $$ Clearly, $f \in \mathrm{L}^{p,\lambda}\left(\Omega\right),$ but Zorko showed in \cite{Zorko_Morreyspaces} that this $f$ can not be approximated by continuous functions in the $\mathrm{L}^{p,\lambda}\left(\Omega\right)$ norm in any neighborhood of $x_{0} \in \Omega.$ However, we do have the following results.  
	\begin{proposition}\label{approximation in Morrey spaces}
		Let $0 < \lambda < n$ and $1 < p < \infty$ be real numbers. 
		\begin{enumerate}[(i)]
			\item For any $u \in L^{p}_{\text{loc}}\left(\mathbb{R}^{n}\right)$ with $\lVert u \rVert_{\mathrm{L}^{p,\lambda}\left(\mathbb{R}^{n}\right)} < +\infty,$ there exists a sequence $\left\lbrace u_{s}\right\rbrace_{s \in \mathbb{N}} \subset C^{\infty}\left(\mathbb{R}^{n}\right)$ such that 
			\begin{align*}
				u_{s} \rightarrow u \quad \text{ in }  L^{p}_{\text{loc}}\left(\mathbb{R}^{n}\right) \qquad \text{ and }\qquad 
				\limsup\limits_{s\rightarrow \infty} \lVert u_{s} \rVert_{\mathrm{L}^{p,\lambda}\left(\mathbb{R}^{n}\right)} \leq \lVert u \rVert_{\mathrm{L}^{p,\lambda}\left(\mathbb{R}^{n}\right)}. 
			\end{align*} 
			\item For any $u \in \mathrm{VL}^{p, \lambda}\left( \mathbb{R}^{n}\right),$ there exists a sequence $\left\lbrace u_{s} \right\rbrace_{s \in \mathbb{N}} \subset C^{\infty}\left( \mathbb{R}^{n}\right)$ such that 
			\begin{align*}
				u_{s} &\rightarrow u \qquad \text{ as } s \rightarrow \infty &&\text{ strongly in } \mathrm{L}^{p, \lambda} \left( \mathbb{R}^{n}\right).
			\end{align*}
		\end{enumerate}
	\end{proposition}
	The first conclusion is easily proved by mollification and a Jensen inequality argument ( see Lemma \ref{approximabilityconditionremovallemma} ). For the second conclusion, one can combine Lemma 1.2 in \cite{ChiarenzaFranciosi_vanishingMorrey} with Zorko's results in \cite{Zorko_Morreyspaces}.
	Now we need a result about sequences which are uniformly bounded in Morrey norms, which can be easily proved. 
	\begin{proposition}\label{Morrey uniformly bounded}
		Let $\Omega \subset \mathbb{R}^{n}$ be open, bounded with at least Lipschitz boundary. Let $1 \leq p < \infty$ and $0 \leq \lambda < n$ be real numbers. Then for any sequence $\left\lbrace u_{s}\right\rbrace_{s \in \mathbb{N}}$ which is uniformly bounded in $\mathrm{L}^{p, \lambda}\left( \Omega\right),$ there exists a subsequence, which we do not relabel, and $u \in \mathrm{L}^{p, \lambda}\left( \Omega\right)$  such that 
		\begin{align*}
			u_{s} \rightharpoonup u \quad \text{ weakly in } \mathrm{L}^{p}\left( \Omega\right) \qquad \text{ and } \qquad 
			\lVert u \rVert_{\mathrm{L}^{p,\lambda}\left(\mathbb{R}^{n}\right)} \leq 	\liminf\limits_{s\rightarrow \infty} \lVert u_{s} \rVert_{\mathrm{L}^{p,\lambda}\left(\mathbb{R}^{n}\right)} .
		\end{align*}
	\end{proposition}
	\subsubsection{Morrey-Sobolev and Vanishing Morrey-Sobolev spaces} 
	We shall be working with the Sobolev scale of spaces modelled on Morrey and Vanishing Morrey spaces. Let  $1\leqslant p <  \infty$, $ 0 \leq \lambda < n$ be real numbers. For any $l \in \mathbb{N},$ we define the $l$-th order Morrey-Sobolev space  as  \begin{align*}
		\mathsf{W}^{l}\mathrm{L}^{p,\lambda}\left(\Omega\right): = \left\lbrace u \in W^{l,p}\left(\Omega \right) : D^{\alpha}u \in \mathrm{L}^{p,\lambda}\left(\Omega\right) \text{ for any multiindex } \left\lvert \alpha \right\rvert \leq l    \right\rbrace  
	\end{align*}
	with the norm 
	$$ \left\lVert u \right\rVert_{\mathsf{W}^{1}\mathrm{L}^{p,\lambda}\left(\Omega\right)} := \sum\limits_{\left\lvert \alpha \right\rvert \leq l}\left\lVert D^{\alpha}u \right\rVert_{\mathrm{L}^{p,\lambda}\left(\Omega\right)} . $$ 
	Similarly, we define the $l$-th order Vanishing Morrey-Sobolev space as 
	\begin{align*}
		\mathsf{W}^{l}\mathrm{VL}^{p,\lambda}\left(\Omega\right): = \left\lbrace u \in W^{l,p}\left(\Omega \right) :  D^{\alpha}u \in \mathrm{VL}^{p,\lambda}\left(\Omega\right), \text{ for any } \left\lvert \alpha \right\rvert \leq l   \right\rbrace .
	\end{align*}
	As usual, the definitions extends to vector-valued maps. We now record an extension result. This can be done in several ways ( See  \cite{Fanciullo_Lamberti_MorreySobolev_extension_Burenkov}, \cite{LambertiViolo_extensionMorreySobolev} ). 
	\begin{proposition}\label{Morrey-Sobolev extension}
		Let $\Omega \subset \mathbb{R}^{n}$ be open, bounded and smooth. Let $1<  p <  \infty$, $ 0 \leq \lambda < n$ be real numbers. For any $l \in \mathbb{N},$ there exists constants $C = C \left( n, l, p, \lambda, \Omega \right) >0$ such that for every $u \in \mathsf{W}^{l}\mathrm{L}^{p,\lambda}\left(\Omega\right),$ there exists an extension $ \tilde{u} \in \mathsf{W}^{l}\mathrm{L}^{p,\lambda}\left(\mathbb{R}^{n}\right)$ such that $\tilde{u} = u$ in $\Omega$ and we have the estimates 
		\begin{align*}
			\left\lVert \tilde{u} \right\rVert_{\mathsf{W}^{l}\mathrm{L}^{p,\lambda}\left(\mathbb{R}^{n}\right)} &\leq C \left\lVert u \right\rVert_{\mathsf{W}^{l}\mathrm{L}^{p,\lambda}\left(\Omega\right)}  \\
			\Theta^{D^{\alpha}\tilde{u}}_{p, \lambda, \mathbb{R}^{n}}\left( r \right) &\leq C \sum\limits_{0 \leq \left\lvert \beta \right\rvert \leq \left\lvert \alpha \right\rvert } \Theta^{D^{\beta}u}_{p, \lambda, \Omega}\left( r \right) \qquad \text{ for all } 0 < r \leq  \operatorname{diam}\left(\Omega\right), 
		\end{align*}
		for all multiindex $0 \leq \left\lvert \alpha \right\rvert \leq l.$ The last estimate implies, in particular, if we have, in addition, that $u \in \mathsf{W}^{l}\mathrm{VL}^{p,\lambda}\left(\Omega\right),$ then $\tilde{u} \in \mathsf{W}^{l}\mathrm{VL}^{p,\lambda}\left(\mathbb{R}^{n}\right).$
	\end{proposition}
	Using the extension result, combined with Proposition \ref{approximation in Morrey spaces}, we have the following approximation result in Morrey-Sobolev spaces. 	
	\begin{proposition}\label{approximation in Morrey-Sobolev spaces}
		Let $\Omega \subset \mathbb{R}^{n}$ be open, bounded and smooth. Let $0 < \lambda < n$ and $1 < p < \infty$ be real numbers and let $l \in \mathbb{N}.$
		\begin{enumerate}[(i)]
			\item For any $u \in \mathsf{W}^{l}\mathrm{L}^{p,\lambda}\left(\Omega\right),$ there exists a constant $C = C\left( \Omega, l, p, \lambda \right) >0 $ and a sequence $\left\lbrace u_{s}\right\rbrace_{s \in \mathbb{N}} \subset C^{\infty}\left(\overline{\Omega}\right)$ such that 
			\begin{align*}
				u_{s} \rightarrow u \quad \text{ in }  W^{l,p}\left(\Omega\right) \qquad \text{ and }\qquad 
				\limsup\limits_{s\rightarrow \infty} \lVert u_{s} \rVert_{\mathsf{W}^{l}\mathrm{L}^{p,\lambda}\left(\Omega\right)} \leq C\lVert u \rVert_{\mathsf{W}^{l}\mathrm{L}^{p,\lambda}\left(\Omega\right)}. 
			\end{align*} 
			\item For any $u \in \mathsf{W}^{l}\mathrm{VL}^{p,\lambda}\left(\Omega\right),$ there exists a sequence $\left\lbrace u_{s} \right\rbrace_{s \in \mathbb{N}} \subset C^{\infty}\left(\overline{\Omega}\right)$ such that 
			\begin{align*}
				u_{s} &\rightarrow u \qquad \text{ as } s \rightarrow \infty &&\text{ strongly in } \mathsf{W}^{l}\mathrm{L}^{p,\lambda}\left(\Omega\right).
			\end{align*}
		\end{enumerate}
	\end{proposition}
	\subsubsection{Sobolev embedding for Morrey-Sobolev spaces}
	We begin with embedding theorems for Morrey-Sobolev spaces in $\mathbb{R}^{n}.$ 
	\begin{proposition}[Adams-Peetre embedding]\label{Adams embedding}
		Let $1 < p < n$ and $0 \leq \lambda < n - p $ be real numbers. Then we have the following Sobolev type embeddings. 
		\begin{itemize}
			\item \textbf{(Adams embedding)} If $u \in \mathsf{W}^{1}\mathrm{L}^{p,\lambda}\left(\mathbb{R}^{n}\right),$ then $u \in \mathrm{L}^{\frac{(n-\lambda)p}{n-\lambda - p}, \lambda}\left(\mathbb{R}^{n}\right).$ 
			\item \textbf{(Peetre embedding)} If $u \in \mathsf{W}^{1}\mathrm{L}^{p,\lambda}\left(\mathbb{R}^{n}\right),$ then $u \in \mathrm{L}^{p, p + \lambda}\left(\mathbb{R}^{n}\right).$ 
		\end{itemize}
	\end{proposition} The first conclusion, usually called the Adams embedding, is due to Adams \cite{Adams_RieszPotentials} and the second one is due to Peetre \cite{Peetre_Morreytypespaces}. The result requires $p + \lambda < n.$ 
	We record the embedding in the case $p+\lambda = n,$ which is a consequence of Poincar\'{e} inequality and H\"{o}lder inequality.  
	\begin{proposition}[Embedding into BMO and VMO]
		Let $n \geq 2$ be an integer and $1 < p \leq n$ be a real number. Then $u \in \mathsf{W}^{1}\mathrm{L}^{p,n-p}\left(\mathbb{R}^{n}\right)$ implies $u \in \mathrm{BMO}\left(\mathbb{R}^{n} \right).$ If $u \in \mathsf{W}^{1}\mathrm{VL}^{p,n-p}\left(\mathbb{R}^{n}\right),$ then $u \in \mathrm{VMO}\left(\mathbb{R}^{n} \right).$
	\end{proposition}
	Using the extension result in Proposition \ref{Morrey-Sobolev extension}, these embeddings can be proved for bounded smooth domains. 
	\begin{proposition}[Morrey-Sobolev embeddings]\label{Adams embedding bounded domain}
		Let $\Omega \subset \mathbb{R}^{n}$ be a bounded open smooth set and $n \geq 2.$ Let $1 < p < n$ and $0 \leq \lambda < n $ be real numbers. Then we have the following Sobolev type embeddings. 
		\begin{itemize}
			\item If $0 \leq \lambda < n-p,$ then 
			\begin{align*}
				\mathsf{W}^{1}\mathrm{L}^{p,\lambda}\left(\Omega \right) &\hookrightarrow \mathrm{L}^{\frac{(n-\lambda)p}{n-\lambda - p}, \lambda}\left(\Omega \right) , \\  \mathsf{W}^{1}\mathrm{VL}^{p,\lambda}\left(\Omega \right) &\hookrightarrow \mathrm{VL}^{\frac{(n-\lambda)p}{n-\lambda - p}, \lambda}\left(\Omega \right). 
			\end{align*} 
			\item If $\lambda = n-p,$ then  
			\begin{align*}
				\mathsf{W}^{1}\mathrm{L}^{p,n-p}\left(\Omega \right) &\hookrightarrow \mathrm{BMO}\left(\Omega \right) , \\  \mathsf{W}^{1}\mathrm{VL}^{p,n-p}\left(\Omega \right) &\hookrightarrow \mathrm{VMO}\left(\Omega \right). 
			\end{align*}
			\item If $\lambda > n-p,$ then 
			\begin{align*}
				\mathsf{W}^{1}\mathrm{L}^{p,\lambda}\left(\Omega \right) &\hookrightarrow C^{0,\frac{\lambda -n + p}{p} }\left(\overline{\Omega} \right). 
			\end{align*}
		\end{itemize} 
	\end{proposition}
	The last conclusion is Morrey's celebrated theorem about the growth of the Dirichlet integral. The theorem was first obtained in a special case by Morrey in \cite{Morrey_dirichletgrowthdim2} ( see also \cite{Morrey_dirichletgrowth2} and Campanato \cite{Campanto_campanatotheoremholdercontinuity} for a general version ).

	\subsection{Smooth principal $G$ bundles with connections}
	A smooth principal $G$-bundle ( or simply a $G$-bundle ) $P$ over $M^{n}$ is usually denoted by the notation $P \stackrel{\pi}{\rightarrow} M^{n},$ where $\pi: P \rightarrow M^{n}$ is a smooth map, called the \emph{projection map}, $P$ is called the \emph{total space} of the bundle, $M^{n}$ is the \emph{base space}. One way to define a smooth principal $G$ bundle $P \stackrel{\pi}{\rightarrow} M^{n},$ is to specify an open cover $\mathcal{U} = \left\lbrace U_{\alpha} \right\rbrace_{\alpha \in I} $ of $M^{n},$ i.e. $M^{n}= \bigcup\limits_{\alpha \in I} U_{\alpha}$ and a collection of bundle trivialization maps $\left\lbrace \phi_{\alpha}\right\rbrace_{\alpha \in I}$ such that $\phi_{\alpha}: U_{\alpha}\times G \rightarrow \pi^{-1}\left( U_{\alpha} \right) $ is a smooth diffeomorphism for every $\alpha \in I$ and each of them preserves the fiber, i.e. $\pi \left(\phi_{\alpha}\left(x, g\right)\right) = x$ for every $g \in G$ for every $x \in U_{\alpha}$ and they are $G$-equivariant, i.e. whenever $U_{\alpha}\cap U_{\beta}$ is nonempty, there exist smooth maps, called transition functions $g_{\alpha\beta}: U_{\alpha}\cap U_{\beta} \rightarrow G$ such that for every $x \in U_{\alpha}\cap U_{\beta},$ we have 
	\begin{align}\label{Gequivariance}
		\left( \phi^{-1}_{\alpha} \circ \phi_{\beta}\right) (x,h) = (x,g_{\alpha\beta}(x)h ) \qquad \text{ for every } h \in G.
	\end{align}  
	From \eqref{Gequivariance}, it is clear that $g_{\alpha\alpha}=\mathbf{1}_{G},$ the identity element of $G,$ for all $\alpha \in I$ and if $U_{\alpha}\cap U_{\beta}\cap U_{\gamma} \neq \emptyset,$ the transition functions satisfy the cocycle identity 
	\begin{align}\label{cocycle condition def}
		g_{\alpha\beta}(x)g_{\beta\gamma}(x) = g_{\alpha\gamma}(x) \qquad \text{ for every } x \in U_{\alpha}\cap U_{\beta}\cap U_{\gamma}.
	\end{align}
	\subsubsection{Bundles as cocycles} We shall be using an equivalent way ( see e.g. \cite{Steenrod_fibrebundles} ) of defining the bundle structure -- by specifying the open cover $\mathcal{U}$ along with the cocycles $\left\lbrace g_{\alpha\beta} \right\rbrace_{\alpha, \beta \in I}.$ $P = \left( \left\lbrace U_{\alpha}\right\rbrace_{\alpha \in I}, \left\lbrace g_{\alpha\beta} \right\rbrace_{\alpha, \beta \in I} \right)$ shall denote a smooth or $C^{0}$ principal $G$ bundle, if $g_{\alpha\beta}$ are smooth or continuous, respectively. 
	
	\subsubsection{Good covers and refinements} A \emph{refinement} of a cover $\left\lbrace U_{\alpha} \right\rbrace_{\alpha \in I}$ is another cover $\left\lbrace V_{j} \right\rbrace_{j \in J}$ with a \emph{refinement map} $\phi: J \rightarrow I $ such that for every $j \in J,$ we have 
	$ V_{j} \subset \subset U_{\phi(j)}.$ If $\left\lbrace U_{\alpha}\right\rbrace_{\alpha \in I}$ and $ \left\lbrace V_{\tilde{\alpha}}\right\rbrace_{\tilde{\alpha} \in \tilde{I}}$ are two covers of the same base space, then a \emph{common refinement} is another cover $\left\lbrace W_{j}\right\rbrace_{j \in J}$ with refinement maps $\phi: J  \rightarrow I$ and $\tilde{\phi}: J \rightarrow \tilde{I}$ such that for every $j \in J,$ we have 
	$ W_{j} \subset \subset U_{\phi(j)}\cap V_{\tilde{\phi}(j)}.$ 
	
	A finite cover is a \emph{good cover in the \u{C}ech sense}, or simply a \emph{good cover} if every nonempty finite intersection of the open sets in the elements of the cover are diffeomorphic to the open unit Euclidean ball. A good cover is called \emph{shrinkable} if, roughly speaking, we can slightly shrink the elements of the cover and the resulting sets still constitute a cover. More formally, if the cover admits a refinement with identity refinement map, i.e. the the refinement map $\phi:I \rightarrow I$ is identity. 
	
	\begin{notation}
		We shall always assume the covers ( including refinements and common refinements ) involved are \emph{finite}, \emph{good cover in the \u{C}ech sense} and \emph{shrinkable}. Since $M^{n}$ are closed manifolds, such covers always exist. In fact, we shall assume that the elements in the cover are small enough convex geodesic balls such that their volume is comparable to Euclidean balls.  
	\end{notation}

	\subsubsection{Connection, gauges and curvature}
	A \emph{connection}, or more precisely, a \emph{connection form} $A$ on $P$ is a collection $\left\lbrace A_{\alpha} \right\rbrace_{\alpha \in I},$ where for each $\alpha \in I,$ $A_{\alpha}: U_{\alpha} \rightarrow \Lambda^{1}\mathbb{R}^{n}\otimes \mathfrak{g}$ is a Lie algebra-valued $1$-form and the collection satisfies the \textbf{gluing relations} 
	\begin{align}\label{gluing relation def}
		A_{\beta} = g_{\alpha\beta}^{-1} dg_{\alpha\beta} +  g_{\alpha\beta}^{-1} A_{\alpha}g_{\alpha\beta} \qquad \text{ in } U_{\alpha}\cap U_{\beta}, \text{  for every } U_{\alpha}\cap U_{\beta} \neq \emptyset.
	\end{align} This collection defines a global $\operatorname*{ad} P$-valued $1$-form   $A: M^{n} \rightarrow \Lambda^{1}T^{\ast}M^{n}\otimes \operatorname*{ad} P,$ which is smooth if $A_{\alpha}$s are. Here $\operatorname*{ad} P= P \times^{\operatorname*{Ad}} \mathfrak{g}$ denotes the \emph{adjoint bundle} associated to $P$ and $\operatorname*{Ad}: G \rightarrow \operatorname*{Aut}\left( \mathfrak{g}\right) \subset \mathbb{GL}\left(\mathfrak{g}\right)$ is the adjoint representation of $G$. We denote the space of smooth connections on a $P$ by the notation $\mathcal{A}^{\infty}\left(P\right).$ A \emph{gauge} $\rho =\left\lbrace \rho_{\alpha} \right\rbrace_{\alpha \in I}$ is a collection of maps $\rho_{\alpha}: U_{\alpha} \rightarrow G$. which represents a change of trivialization for the bundle, given by 
	$$\phi_{\alpha}^{\rho_{\alpha}} (x, h) = \phi_{\alpha}(x,\rho_{\alpha}(x)h ) \qquad \text{ for all } x \in U_{\alpha} \text{ and for all } h \in G . $$ Then the new transition functions are given by $ h_{\alpha\beta} = \rho_{\alpha}^{-1}g_{\alpha\beta}\rho_{\beta}$ in $ U_{\alpha}\cap U_{\beta}$ for all $\alpha, \beta \in I.$ The local representatives of the connections form with respect to the new trivialization $\left\lbrace A_{\alpha}^{\rho_{\alpha}}\right\rbrace_{\alpha \in I}$ satisfy the gauge change identity
	\begin{align}
		\label{gauge change relation}
		A_{\alpha}^{\rho_{\alpha}} = \rho_{\alpha}^{-1}d\rho_{\alpha} + \rho_{\alpha}^{-1} A_{\alpha}\rho_{\alpha} \qquad \text{ in } U_{\alpha}, \quad \text{ for all }\alpha \in I.
	\end{align} 
	The \emph{curvature} or the \emph{curvature form} associated to a connection form $A$ is a global $\operatorname*{ad} P$-valued $2$-form on $M^{n},$ denoted $F_{A}:M^{n}\rightarrow \Lambda^{2}T^{\ast}M^{n}\otimes \operatorname*{ad} P.$ Its local expressions, $\left( F_{A}\right)_{\alpha \in I},$ denoted $F_{A_{\alpha}}$ by a slight abuse of notations, are $\mathfrak{g}$-valued $2$-forms given by 
	\begin{align}
		\label{curvature of a conn def}
		F_{A_{\alpha}} = dA_{\alpha} + A_{\alpha}\wedge A_{\alpha} = dA_{\alpha} + \frac{1}{2} \left[ A_{\alpha}, A_{\alpha}\right] \qquad \text{ in } U_{\alpha} , \quad \text{ for all }\alpha \in I,
	\end{align}
	where the wedge product denotes the wedge product of $\mathfrak{g}$-valued forms and the bracket $\left[ \cdot, \cdot \right]$ is the Lie bracket of $\mathfrak{g},$ extended to $\mathfrak{g}$-valued forms the usual way. The gauge change identity \eqref{gauge change relation} implies 
	\begin{align}
		\label{connection gauge change def}
		F_{A_{\alpha}^{\rho_{\alpha}}} = \rho_{\alpha}^{-1} F_{A_{\alpha}}\rho_{\alpha} \qquad \text{ in } U_{\alpha} , \quad \text{ for all }\alpha \in I.
	\end{align}
	Similarly, the gluing relation \eqref{gluing relation def} implies that we have 
	$F_{A_{\beta}} = g_{\alpha\beta}^{-1} F_{A_{\alpha}}g_{\alpha\beta}$ in $ U_{\alpha}\cap U_{\beta},$ whenever $U_{\alpha}\cap U_{\beta}\neq \emptyset.$ This as usual implies  that the collection $\left( F_{A}\right)_{\alpha \in I}$ defines a global $\operatorname*{ad} P$-valued $2$-form on $M^{n}.$   
	
	\subsubsection{Yang-Mills energy} For any $1 \leq q < \infty,$ the $q$-Yang-Mills energy of a connection $A,$ denoted $YM_{q}\left( A\right),$ is defined as 
	$$ YM_{q}\left( A\right): = \int_{M^{n}} \left\lvert F_{A}\right\rvert^{q},$$
	where the norm $\left\lvert \cdot \right\rvert$ denotes the norm for \emph{$\mathfrak{g}$-valued differential forms}, induced by the Killing scalar product which is invariant under the adjoint action. Invariance under the adjoint action and \eqref{connection gauge change def} implies ( see e.g. \cite{Tu_connections} ) that the norm $\left\lvert F_{A} \right\rvert$ is gauge invariant and thus the integrand in $YM_{q}$ is gauge invariant for any $1 \leq q < \infty.$
	
	\subsection{$G$-valued Morrey-Sobolev maps}\label{Gvaluedmaps}  Without loss of generality, we can always assume that the compact finite dimensional Lie group $G$ is endowed with a bi-invariant metric and is smoothly embedded isometrically in $\mathbb{R}^{N_{0}}$ for some, possibly quite large, integer $N_{0} \geq 1.$ By compactness of $G,$ there exists a constant $C_{G} > 0$ such that $G \subset \subset B_{C_{G}}(0) \subset \mathbb{R}^{N_{0}}.$ 
	
	\begin{definition}
		Let $U \subset \mathbb{R}^{n}$ be open. Let $1 < p < \infty$ and $0 < \lambda < n$ be real numbers. For $l=1$ or $2$, the  space $\mathsf{W}^{l}\mathrm{L}^{p, \lambda}\left( U ; G \right)$ and $\mathsf{W}^{l}\mathrm{VL}^{p, \lambda}\left( U ; G \right)$ is defined as 
		\begin{align*}
			\mathsf{W}^{l}\mathrm{L}^{p, \lambda}\left( U ; G \right)&:=  \left\lbrace  f \in 	\mathsf{W}^{l}\mathrm{L}^{p, \lambda}\in\left( U ; \mathbb{R}^{N_{0}}\right) : f(x) \in G \text{ for a.e. } x \in U  \right\rbrace, \\
			\mathsf{W}^{l}\mathrm{VL}^{p, \lambda}\left( U ; G \right)&:=  \left\lbrace  f \in 	\mathsf{W}^{l}\mathrm{VL}^{p, \lambda}\in\left( U ; \mathbb{R}^{N_{0}}\right) : f(x) \in G \text{ for a.e. } x \in U  \right\rbrace.
		\end{align*}
		
	\end{definition}
	The compactness of $G$ implies that $\mathsf{W}^{l}\mathrm{L}^{p, \lambda}\left( U ; G \right)\left( U ; G \right) \subset L^{\infty}\left( U ; \mathbb{R}^{N_{0}} \right)$  with the bounds $\left\lVert f \right\rVert_{L^{\infty}\left( U ; \mathbb{R}^{N_{0}} \right)} \leq C_{G}$ for any $f \in W^{k,p}\left( U ; G \right).$ By the Gagliardo-Nirenberg inequality, it follows that $W^{k,p}\left( U ; G \right)$ is an infinite dimensional topological group with respect to the topology in inherits as a topological subspace of the  Banach space  $W^{k,p}\left( U ; \mathbb{R}^{N_{0}}\right)$. Note that $\mathsf{W}^{l}\mathrm{L}^{p, \lambda}\left( U ; G \right)$ is not a linear space. As a subset of the Banach space $\mathsf{W}^{l}\mathrm{L}^{p, \lambda}\left( U ; \mathbb{R}^{N_{0}} \right),$ it inherits only a topology from the norm topology of $\mathsf{W}^{l}\mathrm{L}^{p, \lambda}\left( U ; \mathbb{R}^{N_{0}} \right).$ 
	
	On the other hand, since the Lie algebra of $G$, i.e. $\mathfrak{g}$ is a linear space and consequently so is $\Lambda^{k}\mathbb{R}^{n}\otimes\mathfrak{g}$ for any $0 \leq k \leq n,$ the space of $\mathfrak{g}$-valued $k$-forms of class $\mathsf{W}^{l}\mathrm{L}^{p, \lambda}$ or $\mathsf{W}^{l}\mathrm{VL}^{p, \lambda}$ is defined by requiring each scalar component of the maps to be $\mathsf{W}^{l}\mathrm{L}^{p, \lambda}$ or respectively, $\mathsf{W}^{l}\mathrm{VL}^{p, \lambda}$ functions in the usual sense. The standard properties of Sobolev functions, including smooth approximation by mollification, carry over immediately to this setting by arguing componentwise. The stark contrast between the two settings is due to the fact that in general a map $ g \in \mathsf{W}^{l}\mathrm{VL}^{p, \lambda}\left( U; G\right)$ need not have a $\mathsf{W}^{l}\mathrm{VL}^{p, \lambda}$ `lift' to the Lie algebra. More precisely, there need not exist a map $ u \in \mathsf{W}^{l}\mathrm{VL}^{p, \lambda}\left( U; \mathfrak{g}\right)$ with the property that $g = \operatorname{exp}\left( u \right),$ where $\operatorname{exp}: \mathfrak{g} \rightarrow G$ is the exponential map of $G.$ However, $\mathfrak{g} = T_{\mathbbm{1}_{G}}G$ and there exists a small enough $C^{0}$-neighborhood of the identity element $\mathbbm{1}_{G} \in G$ in $G$ such the exponential map is a local smooth diffeomorphism onto that neighborhood. Moreover, since $G$ is a smooth manifold which is also compact, there exists a small number $\delta_{G}>0$ such that `nearest point projection map onto $G$' is well defined and smooth. See Proposition \ref{tubular nbd lemma}.  
	
	\subsection{Sobolev and Morrey-Sobolev bundles and connections}
	\paragraph{Sobolev principal $G$ bundles} Now we define bundles where the transition functions are Sobolev or Morrey-Sobolev maps, not necessarily smooth or continuous ( see \cite{Riviere_YM_Tianbook} ). In analogy with the case of smooth bundles, we define 
	\begin{definition}[$\mathsf{W}^{l}\mathrm{L}^{p,\lambda}$ and $\mathsf{W}^{l}\mathrm{VL}^{p,\lambda}$ principal $G$-bundles]
		Let $1 < p < \infty$ and $0 < \lambda < n$ be real numbers. For $l=1$ or $2,$ we call $P$ a $\mathsf{W}^{l}\mathrm{L}^{p,\lambda}$ principal $G$-bundle over $M^{n}$  if $P = \left( \left\lbrace U_{\alpha}\right\rbrace_{\alpha \in I}, \left\lbrace g_{\alpha\beta} \right\rbrace_{\alpha, \beta \in I} \right),$ where 
		\begin{enumerate}[(i)]
			\item $g_{\alpha\beta} \in \mathsf{W}^{l}\mathrm{L}^{p,\lambda} \left( U_{\alpha}\cap U_{\beta}; G\right) $ for every $\alpha, \beta \in I$ with $U_{\alpha}\cap U_{\beta} \neq \emptyset, $
			\item $g_{\alpha\alpha}= \mathbbm{1}_{G}$ for every $\alpha \in I,$ 
			\item for every $\alpha, \beta, \gamma \in I$ such that $U_{\alpha}\cap U_{\beta}\cap U_{\gamma} \neq \emptyset, $ the transition maps satisfy the cocycle conditions 
			\begin{align}
				\label{cocycle sobolev def}
				g_{\alpha\beta}(x)g_{\beta\gamma}(x) = g_{\alpha\gamma}(x) \qquad \text{ for a.e. } x \in U_{\alpha}\cap U_{\beta}\cap U_{\gamma}.
			\end{align}
		\end{enumerate}
		A $\mathsf{W}^{l}\mathrm{L}^{p,\lambda}$ principal $G$-bundle $P = \left( \left\lbrace U_{\alpha}\right\rbrace_{\alpha \in I}, \left\lbrace g_{\alpha\beta} \right\rbrace_{\alpha, \beta \in I} \right),$ is called a $\mathsf{W}^{l}\mathrm{VL}^{p,\lambda}$ principal $G$-bundle over $M^{n},$ if in addition, for every $ \alpha, \beta \in I $ with  $U_{\alpha}\cap U_{\beta} \neq \emptyset,$ we have 
		\begin{align*}
			g_{\alpha\beta} \in \mathsf{W}^{l}\mathrm{VL}^{p,\lambda}\left( U_{\alpha}\cap U_{\beta}; G\right). 
		\end{align*}
	\end{definition}
	We also need the notion of gauge equivalence of bundles, extended to the setting of Morrey-Sobolev bundles. 
	\begin{definition}[Gauge equivalence]
		Let $1 < p < \infty$ and $0 < \lambda < n$ be real numbers and let $l=1$ or $2.$ Let 
		\begin{align*}
			P = \left( \left\lbrace U_{\alpha}\right\rbrace_{\alpha \in I}, \left\lbrace g_{\alpha\beta} \right\rbrace_{\alpha, \beta \in I} \right) \qquad \text{ and } \qquad \tilde{P} = \left( \left\lbrace V_{\tilde{\alpha}}\right\rbrace_{\tilde{\alpha} \in \tilde{I}}, \left\lbrace h_{\tilde{\alpha}\tilde{\beta}} \right\rbrace_{\tilde{\alpha}, \tilde{\beta} \in \tilde{I}} \right)
		\end{align*} be two $\mathsf{W}^{l}\mathrm{L}^{p,\lambda}$ principal $G$-bundles over the same base space $M^{n}.$ $P$ and $\tilde{P}$ is called \emph{$\mathsf{W}^{l}\mathrm{L}^{p,\lambda}$-gauge equivalent} by a gauge transformation $\sigma$, denoted by 
		\begin{align*}
			P  \simeq_{\sigma}   \tilde{P},
		\end{align*}
		if there exists a common refinement $\left\lbrace W_{j}\right\rbrace_{j \in J}$ of the covers $\left\lbrace U_{\alpha}\right\rbrace_{\alpha \in I}$ and $\left\lbrace V_{\tilde{\alpha}}\right\rbrace_{\tilde{\alpha} \in \tilde{I}} $ and maps 
		$\sigma_{j} \in \mathsf{W}^{l}\mathrm{L}^{p,\lambda}\left( W_{j}; G\right)$ for each $j \in J$  such that 
		\begin{align}\label{equivalence gauge relation}
			h_{\tilde{\phi}(i)\tilde{\phi}(j)} = \sigma_{i}^{-1}g_{\phi(i)\phi(j)}\sigma_{j}\qquad \text{ a.e. in } W_{i}\cap W_{j},
		\end{align} 
		for each pair $i,j \in J$ with $W_{i}\cap W_{j} \neq \emptyset,$ where $\phi: J  \rightarrow I$ and $\tilde{\phi}: J \rightarrow \tilde{I}$ are the respective refinement maps.  If the local gauge transformations are also in the Vanishing Morrey-Sobolev space, i.e. $\sigma_{j} \in \mathsf{W}^{l}\mathrm{VL}^{p,\lambda}\left( W_{j}; G\right)$ for each $j \in J,$ them we say $P$ and $\tilde{P}$ is called \emph{$\mathsf{W}^{l}\mathrm{VL}^{p,\lambda}$-gauge equivalent}. 
	\end{definition} 
	Smooth or $C^{0}$ equivalence is defined in analogous manner by requiring the maps $\sigma_{j}$ to be smooth or $C^{0}$ respectively. It is easy to check that they are indeed equivalence relations in the corresponding category. If $P$ is a $C^{0}$-bundle, we denote its equivalence class under $C^{0}$-equivalence by $\left[P\right]_{C^{0}}.$   
	\paragraph{Morrey-Sobolev connections}
	Now we define the notion of a Morrey-Sobolev connection  on a bundle $P.$  
	\begin{definition}
		Let $1 < p < \infty$ and $0 < \lambda < n$ be real numbers Let $P$ be a $\mathsf{W}^{2}\mathrm{L}^{p,\lambda}$ be a principal $G$-bundle on $M^{n}.$ We say the connection $A=\left\lbrace A_{\alpha} \right\rbrace_{\alpha \in I}$ is a $\mathsf{W}^{1}\mathrm{L}^{p,\lambda}$-connection on $P$ if we have 
		\begin{align*}
			A_{\alpha} \in \mathsf{W}^{1}\mathrm{L}^{p,\lambda}\left( U_{\alpha}; \Lambda^{1}\mathbb{R}^{n}\otimes\mathfrak{g}\right)  \qquad \text{ for every } \alpha \in I. 
		\end{align*}
	\end{definition}

	\section{Preparatory results}\label{Prep estimates}	
	\subsection{Subcritical regime}\label{subcritical regime}
	\subsubsection{Approximation of $G$-valued maps in the critical case}
	We first begin with an approximation result for $G$-valued Morrey-Sobolev maps. The basic argument is well-known since \cite{SchoenUhlenbeck_boundaryregularityharmonicmaps} ( see also \cite{BrezisNirenberg_degree1}, \cite{Mironescu_sobolevmapsonmanifolds_survey},  \cite{VanSchaftingen_oxford}, \cite{mironescu:hal-04101540} ). We first record a standard fact about smooth compact submanifolds. 
	\begin{proposition}\label{tubular nbd lemma}
		For any compact real Lie group $G,$ assumed to be isometrically embedded in $\mathbb{R}^{N_{0}}$ for some $N_{0} \in \mathbb{N},$ there exists a constant $\delta_{G}>0$ and a map, called the `nearest point projection', $\Pi \in C^{\infty}\left( \overline{N_{\delta_{G}}}; G\right) $ such that 
		\begin{align*}
			\Pi \left(y \right) = y \qquad \text{ for any } y \in G,
		\end{align*}
		where $N_{\delta_{G}}$ is the tubular neighborhood of $G$ defined as 
		\begin{align*}
			N_{\delta_{G}}:= \left\lbrace x \in \mathbb{R}^{N_{0}}: \operatorname{dist} \left(x, G\right) < \delta_{G}\right\rbrace.  	
		\end{align*}
	\end{proposition}
	Now we establish an approximation result in the critical case. 
	\begin{theorem}\label{Approximation of G valued maps critical}
		Let $\tilde{\Omega} \subset \subset \Omega \subset \mathbb{R}^{n}$ be both open, bounded and smooth. Let $1 < m \leq n/2$ be a real number. 
		\begin{enumerate}[(i)]
			\item For any $u \in \mathsf{W}^{2}\mathrm{VL}^{m, n-2m}\left( \Omega; G\right),$ there exists a sequence $\left\lbrace u^{s} \right\rbrace_{s \in \mathbb{N}} \subset C^{\infty}\left( \overline{\tilde{\Omega}}; G\right)$ such that 
			\begin{align*}
				u^{s} \rightarrow u \qquad \text{ in } \mathsf{W}^{2}\mathrm{L}^{m, n-2m}\left( \tilde{\Omega}; G\right). 
			\end{align*}
			\item There exists a smallness parameter $\varepsilon_{\text{approx}} = \varepsilon_{\text{approx}}\left( \Omega, \tilde{\Omega}, n, m, G\right) \in (0,1)$ such that for any $u \in \mathsf{W}^{2}\mathrm{L}^{m, n-2m}\left( \Omega; G\right)$ satisfying 
			\begin{align*}
				\left\lVert du \right\rVert_{\mathrm{L}^{2m, n-2m}\left( \Omega; G\right)} \leq \varepsilon_{\text{approx}},
			\end{align*}
			there exists a sequence $\left\lbrace u^{s} \right\rbrace_{s \in \mathbb{N}} \subset C^{\infty}\left( \overline{\tilde{\Omega}}; G\right)$ such that 
			\begin{align*}
				u^{s} &\rightarrow u \qquad \text{ in } W^{2,m}\left( \tilde{\Omega}; G\right), \\
				u^{s} &\rightarrow u \qquad \text{ in } W^{1,2m}\left( \tilde{\Omega}; G\right), \\
				\limsup\limits_{s \in \mathbb{N}}	\left\lVert u^{s} \right\rVert_{\mathsf{W}^{2}\mathrm{L}^{m, n-2m}\left( \tilde{\Omega}; G\right)} &\leq C_{\text{approx}} 	\left\lVert u \right\rVert_{\mathsf{W}^{2}\mathrm{L}^{m, n-2m}\left( \Omega; G\right)},
			\end{align*}
			for some constant $C_{\text{approx}} = C_{\text{approx}} \left( \Omega, \tilde{\Omega}, n, m, G\right) >0.$
		\end{enumerate}
	\end{theorem} 
	\begin{proof}
		We only prove part (ii). The first part is similar but easier. Since $G$ embeds into $\mathbb{R}^{N_{0}}$ for some large natural number $N_{0},$ by standard mollification, we can construct a smooth sequence $v^{s}:= u \ast \eta^{s},$ where $\left\lbrace \eta_{s}\right\rbrace_{s \in \mathbb{N}}$ is a standard mollifying sequence, such that the Sobolev convergences hold on $\tilde{\Omega}.$ The estimate for the Morrey-Sobolev norms follows from a Jensen inequality argument outlined in the proof of Lemma \ref{approximabilityconditionremovallemma}. Hence the only issue here is that $v^{\varepsilon}$ does not, in general, takes values in $G.$ However, we have the key estimate 
		\begin{align*}
			\operatorname{dist}\left( v^{s}\left(x\right), G\right) \leq C \fint_{B_{1/s}\left(x\right)}\fint_{B_{1/s}\left(x\right)} \left\lvert u\left(y\right) - u\left(z\right) \right\rvert\mathrm{d}y \mathrm{d}z 
		\end{align*}
		for any $x \in \tilde{\Omega}$ and any $s \in \mathbb{N}.$ Now the integral on the right can be bounded above by the $\mathrm{BMO}$ seminorm on $u$ and consequently, using Proposition \ref{Adams embedding bounded domain}, also by the $\mathrm{L}^{2m, n-2m}$ norm of $du.$  Thus, we can choose $\varepsilon_{\text{approx}}$  small enough to obtain 
		\begin{align*}
			\sup\limits_{x \in \overline{\tilde{\Omega}}}\operatorname{dist}\left( v^{s}\left(x\right), G\right) < \delta_{G},
		\end{align*}
		where $\delta_{G}$ is the constant given by Proposition \ref{tubular nbd lemma}. Then we define 
		\begin{align*}
			u^{s} := \Pi \left( v^{s}\right).
		\end{align*}
		Now it is easy to verify that $\left\lbrace u^{s} \right\rbrace_{s \in \mathbb{N}}$ satisfies all the required properties. 
	\end{proof}
	\begin{remark}
		The crucial point is $\mathsf{W}^{2}\mathrm{L}^{m, n-2m}$ embeds into $\mathrm{BMO}$ and $\mathsf{W}^{2}\mathrm{VL}^{m, n-2m}$ embeds into $\mathrm{VMO}.$ This result is much easier for $\mathsf{W}^{2}\mathrm{L}^{p, \lambda}$  with $2p + \lambda > n,$ in view of the embedding of $\mathsf{W}^{2}\mathrm{L}^{p, \lambda}$ into $C^{0}.$ Notice that in part (ii), approximation in the Morrey norms are impossible, as smooth functions are not dense in Morrey norms, even for scalar valued functions. 
		
	\end{remark}
	\subsubsection{Smooth approximation in the subcritical regime}
	Now we investigate smooth approximation questions for $\mathsf{W}^{l}\mathrm{L}^{p, \lambda}$ Morrey-Sobolev bundles. By the approximation result in the last section, given any $\mathsf{W}^{2}\mathrm{VL}^{m, n-2m}$ cocycle  $\left\lbrace g_{\alpha\beta} \right\rbrace_{\alpha, \beta \in I},$ we can find sequence of maps $\left\lbrace \left\lbrace g^{s}_{\alpha\beta} \right\rbrace_{\alpha, \beta \in I} \right\rbrace_{s \in \mathbb{N}},$ such that for each $s \in \mathbb{N},$ 
	$g^{s}_{\alpha\beta} \in C^{\infty} \left( \overline{V_{\alpha}\cap V_{\beta}}; G\right) $ and 
	$$ g^{s}_{\alpha\beta} \rightarrow g_{\alpha\beta} \qquad \text{ in } \mathsf{W}^{2}\mathrm{L}^{m, n-2m}\left( V_{\alpha}\cap V_{\beta}; G\right),$$ where $\left\lbrace V_{\alpha}\right\rbrace_{\alpha \in I}$ is a slight shrinking of the original cover $\left\lbrace U_{\alpha}\right\rbrace_{\alpha \in I}.$ However, there is no particular reason why $\left\lbrace g^{s}_{\alpha\beta} \right\rbrace_{\alpha, \beta \in I}$ would define a cocycle, i.e. satisfy \eqref{cocycle condition def}. We can however modify our procedure to ensure this the subcritical regime, i.e. for $\mathsf{W}^{2}\mathrm{VL}^{p, \lambda}$ bundles with $2p + \lambda > n.$ Below we present a weaker result which however would be sufficient for our purposes. 
	\begin{theorem}[Smooth approximation in Morrey norm in the subcritical Sobolev  regime]\label{smoothingC0bundleMorrey} 
		Let $n \geq 3$ be an integer and let $0 < \lambda < n $ and $ 1< p < q <  \infty$ be real numbers satisfying $2p +\lambda > n.$ Given any $\mathsf{W}^{2}\mathrm{L}^{q,\lambda}$ principal $G$-bundle $P = \left( \left\lbrace V_{j}\right\rbrace_{j \in J}, \left\lbrace h_{ij} \right\rbrace_{i, j \in J}\right),$ there exists a refinement $\left\lbrace W_{i}\right\rbrace_{i \in K}$ of $\left\lbrace V_{j}\right\rbrace_{j \in J}$ with a refinement map $\psi: K \rightarrow J,$ a sequence of continuous gauges $\left\lbrace \left\lbrace \sigma_{i}^{s} \right\rbrace_{i \in K} \right\rbrace_{s \in \mathbb{N}}$ and a sequence of smooth bundles $\left\lbrace P^{s} \right\rbrace_{s \in \mathbb{N}},$ where 
		\begin{align*}
			P^{s} &= \left( \left\lbrace W_{i}\right\rbrace_{i \in K},\left\lbrace h^{s}_{ij}\right\rbrace_{i,j \in K}\right) \qquad \text{ for every } s \in \mathbb{N}, \\
			\sigma_{i}^{s} &\in \mathsf{W}^{2}\mathrm{VL}^{p,n-2m}\left( W_{i}; G\right) \qquad \text{ for every } i \in K  \text{ for every } s \in \mathbb{N},
		\end{align*} such that 
		\begin{align*}
			P^{s} &\simeq_{\sigma^{s}} P, \\
			h^{s}_{ij} &\rightarrow h_{\psi\left(i\right)\psi\left(j\right)} \qquad \text{ in } \mathsf{W}^{2}\mathrm{VL}^{p, n-2m}.
		\end{align*} 
	\end{theorem}
	\begin{proof}
		The crux of the matter are hinges on the following two points. The first is that since $2q+\lambda > 2p+\lambda > n,$ $\mathsf{W}^{2}\mathrm{VL}^{p,n-2m}$ embeds into $C^{0}.$ The second is that since $q >p,$ $\mathsf{W}^{2}\mathrm{L}^{q,n-2m}$ embeds into $\mathsf{W}^{2}\mathrm{VL}^{p,n-2m}.$ With these two facts in hand, the rest is a somewhat tedious but straight forward adaptation of the proof of Theorem $8$ in \cite{Sil_YangMillscriticaltoappear}. 
	\end{proof}	
	\subsubsection{Cocycle factorization in the subcritical regime}
	The cocycle factorization result in the subcritical regime has been essentially proved by Uhlenbeck. Here we just record a result, which is just Proposition 3.2 in \cite{UhlenbeckGaugefixing}.
	\begin{theorem}\label{C0 cocycle factorization}
		Let $\left\lbrace U_{\alpha}\right\rbrace_{\alpha \in I}$ be an finite open cover for $M^{n}$, where $I$ has $N$ elements. Let $h_{\alpha \beta}: U_\alpha \cap U_\beta \rightarrow G$ and $g_{\alpha \beta}: U_\alpha \cap U_\beta \rightarrow G$ be two sets of continuous cocycles. Then there exists $\delta_{0} = \delta_{0}\left( n, N, M^{n}, G \right)$ such that if
		$$
		m=\max _{\substack{\alpha, \beta \in I \\ x \in U_\alpha \cap U_\beta}}\left|\exp ^{-1} h_{\alpha, \beta}(x) g_{\beta, \alpha}(x)\right| \leq \delta_{0},
		$$
		then there exists a refinement $V_\alpha \subset U_\alpha, M^{n} \subset \bigcup_\alpha V_\alpha$ and continuous $\rho_\alpha: V_\alpha \rightarrow G$ such that $h_{\alpha, \beta}=\rho_\alpha g_{\alpha, \beta} \rho_\beta^{-1}$ on $V_\alpha \cap V_\beta$. Moreover, $\max \left|\exp ^{-1} \rho_\alpha\right| \leq cm$ for some constant $c = c \left( n, N, M^{n}, G \right)>0.$
	\end{theorem}
	Since $\mathsf{W}^{2}\mathrm{L}^{p, \lambda}$ embeds continuously into $C^{0, \frac{\lambda -n +2p}{2p}}$, for $1<p<n$ and $\lambda > n-2p,$ we can easily derive the following simple corollary by arguing along the same lines as in Corollary 3.3 in \cite{UhlenbeckGaugefixing}. 
	\begin{theorem}\label{subcritical cocycle factorization}
		Let $\left\lbrace U_{\alpha}\right\rbrace_{\alpha \in I}$ be an finite open cover for $M^{n}$, where $I$ has $N$ elements. Let $h_{\alpha \beta}: U_\alpha \cap U_\beta \rightarrow G$ and $g_{\alpha \beta}: U_\alpha \cap U_\beta \rightarrow G$ be two sets of  $\mathsf{W}^{2}\mathrm{L}^{p,\lambda}$ cocycles for some $1 < p < n$ and $\lambda > n-2p.$ Then there exists $\delta_{0} = \delta_{0}\left( n, p, \lambda, N, M^{n}, G \right)$ such that if
		$$
		\Lambda=\max _{\substack{\alpha, \beta \in I \\ x \in U_\alpha \cap U_\beta}}\left\lVert h_{\alpha, \beta} -  g_{\alpha\beta}\right\rVert_{\mathsf{W}^{2}\mathrm{L}^{p,\lambda}} \leq \delta_{0},
		$$
		then there exists a refinement $ \left\lbrace V_{j} \right\rbrace_{j \in J}$ of $\left\lbrace U_{\alpha}\right\rbrace_{\alpha \in I}$, and H\"{o}lder continuous gauge transformations $\left\lbrace \rho_{j} \right\rbrace_{j \in J}$ such that for each $j \in J,$ we have $\rho_{j} \in  \mathsf{W}^{2}\mathrm{L}^{p,\lambda} \left( V_j ; G \right)$ and for every $i, j \in J$ such that  $V_i \cap V_j \neq \emptyset, $ we have $$ h_{ij}=\rho_i g_{ij} \rho_{j}^{-1} \qquad \text{ on  } V_i \cap V_j.$$  Moreover, there exists a constant $c = c \left( n, N, p, \lambda, M^{n}, G \right)>0$ such that $$\max\limits_{j \in J} \left\lVert \rho_{j} - \mathbbm{1}_{G}  \right\rVert_{\mathsf{W}^{2}\mathrm{L}^{p,\lambda}} \leq c\Lambda.$$  
	\end{theorem}
	
	\subsection{Elliptic estimates in Morrey spaces}
	
	\subsubsection{Gaffney type inequalities}
	We begin with a Gaffney inequality for Morrey spaces. These are intimately related to second order Morrey estimates for the Hodge type boundary value problems for the Hodge Laplacian. Morrey's proof of $L^{p}$ and Schauder estimates for the Hodge Laplacian uses representation formulas and boundedness of Calderon-Zygmund operators in those spaces ( see Section 7.6 in \cite{Morrey1966} ) and the Calderon-Zygmund operators are known to be bounded in Morrey spaces ( cf. \cite{Adams_MorreySpaces} ), the results is probably well-known to experts. However, since we are unable to find a precise reference for this, here we deduce them from a more general result proved in \cite{Sengupta_Sil_Morrey_Gaffney}.    
	\begin{proposition}[Gaffney inequality in Morrey spaces]\label{Gaffney ineq in Morrey spaces}
		Let $\Omega \subset \mathbb{R}^{n}$ be open, bounded, smooth and contractible. Let $1 \leq k \leq n-1,$ $N \geq 1$ be integers. Let $X$ be any $N$-dimensional real vector space and let $ 0 <  \mu < n$ and $1 < p < \infty$ be real numbers. Let 
		\begin{gather*}
			f \in \mathrm{L}^{p,\mu}\left(\Omega; \Lambda^{k+1}\mathbb{R}^{n}\otimes X \right), 
			g \in  \mathrm{L}^{p,\mu}\left(\Omega; \Lambda^{k-1}\mathbb{R}^{n}\otimes X \right), \\\omega_{0} \in \mathsf{W}^{1}\mathrm{L}^{p,\mu}\left(\Omega; \Lambda^{k}\mathbb{R}^{n}\otimes X \right).
		\end{gather*}
		
		\noindent \textbf{(i)} Suppose $df = 0$, $\delta g = 0$ in $\Omega$ and 
		$ \iota^{\ast}_{\partial \Omega } d\omega_{0} = \iota^{\ast}_{\partial \Omega } f $ on $\partial\Omega.$
		Then there exists a unique solution $\omega \in \mathsf{W}^{1}\mathrm{L}^{p,\mu}\left(\Omega; \Lambda^{k}\mathbb{R}^{n}\otimes X \right)$  to the following boundary value problem, 
		\begin{equation} \label{problemddeltalinear}
			\left\lbrace \begin{aligned}
				d\omega = f  \quad &\text{and} \quad d^{\ast} \omega = g &&\text{ in } \Omega, \\
				\iota^{\ast}_{\partial \Omega }\omega &= 	\iota^{\ast}_{\partial \Omega }\omega_0 &&\text{  on } \partial\Omega,
			\end{aligned} 
			\right. \tag{$\mathcal{P}_{T}$}
		\end{equation}
		and there exists a constant $C_{\text{e}} = C_{\text{e}} \left(m,n,k,p,\mu, \Omega \right) >0$ such that 
		\begin{align*}
			\left\lVert \omega \right\rVert_{\mathsf{W}^{1}\mathrm{L}^{p,\mu}} \leq  C_{\text{e}} \left( \left\lVert f\right\rVert_{\mathrm{L}^{p,\mu}}  + \left\lVert g\right\rVert_{\mathrm{L}^{p,\mu}} + \left\lVert \omega_{0}\right\rVert_{\mathsf{W}^{1}\mathrm{L}^{p,\mu}}\right).  
		\end{align*}\smallskip
		
		\noindent\textbf{(ii)} Suppose $df = 0$, $\delta g = 0$ in $\Omega$ and $ 	\iota^{\ast}_{\partial \Omega } \left( \ast g\right)  = 	\iota^{\ast}_{\partial \Omega } \left( \ast d^{\ast}\omega_{0}\right) $ on $\partial\Omega.$
		Then there exists a unique solution $\omega \in \mathsf{W}^{1}\mathrm{L}^{p,\mu}\left(\Omega; \Lambda^{k}\mathbb{R}^{n}\otimes X \right)$  to the following boundary value problem, 
		\begin{equation} \label{problemddeltalinearnormal}
			\left\lbrace \begin{aligned}
				d\omega = f  \quad &\text{and} \quad  d^{\ast} \omega = g &&\text{ in } \Omega, \\
				\iota^{\ast}_{\partial \Omega } \left( \ast \omega\right)	 &= \iota^{\ast}_{\partial \Omega }  \left( \ast \omega_{0}\right)  &&\text{  on } \partial\Omega,
			\end{aligned} 
			\right. \tag{$\mathcal{P}_{N}$}
		\end{equation}
		and there exists a constant $C_{\text{e}} = C_{\text{e}} \left(m,n,k,p,\mu, \Omega \right) >0$ such that 
		\begin{align*}
			\left\lVert \omega \right\rVert_{\mathsf{W}^{1}\mathrm{L}^{p,\mu}} \leq  C_{\text{e}} \left( \left\lVert f\right\rVert_{\mathrm{L}^{p,\mu}}  + \left\lVert g\right\rVert_{\mathrm{L}^{p,\mu}} + \left\lVert \omega_{0}\right\rVert_{\mathsf{W}^{1}\mathrm{L}^{p,\mu}}\right).  
		\end{align*}
	\end{proposition}
	As a consequence of these estimates, we have the following. 
	\begin{proposition}\label{Morrey estimates for the Laplacian}
		Let $\Omega \subset \mathbb{R}^{n}$ be open, bounded, smooth and contractible. Let $N \geq 1$ be an integer and let $X$ be any $N$-dimensional real vector space. Let $ 0 <  \mu < n$ and $1 < p < \infty$ be real numbers. Let 
		\begin{align*}
			f \in \mathrm{L}^{p,\mu}\left(\Omega; X \right), u_{0} \in \mathsf{W}^{2}\mathrm{L}^{p,\mu}\left(\Omega; X \right).
		\end{align*}
		
		\noindent \textbf{(i)} There exists a unique solution $u \in \mathsf{W}^{2}\mathrm{L}^{p,\mu}\left(\Omega; X \right)$  to the following Dirichlet boundary value problem, 
		\begin{equation} \label{problemDirichletLaplacian}
			\left\lbrace \begin{aligned}
				\Delta u &= f  &&\text{ in } \Omega, \\
				u &= 	u_0 &&\text{  on } \partial\Omega,
			\end{aligned} 
			\right. \tag{$\mathcal{P}_{Dirichlet}$}
		\end{equation}
		and there exists a constant $C_{\text{e}} = C_{\text{e}} \left(m,n,p,\mu, \Omega \right) >0$ such that 
		\begin{align*}
			\left\lVert u \right\rVert_{\mathsf{W}^{2}\mathrm{L}^{p,\mu}} \leq  C_{\text{e}} \left( \left\lVert f\right\rVert_{\mathrm{L}^{p,\mu}} + \left\lVert u_{0}\right\rVert_{\mathsf{W}^{2}\mathrm{L}^{p,\mu}}\right).  
		\end{align*}\smallskip
		
		\noindent\textbf{(ii)} Suppose $f$ satisfy 
		\begin{align*}
			\int_{\Omega} f = \int_{\partial \Omega} \frac{\partial u_{0}}{\partial \nu}. 
		\end{align*}
		Then there exists a unique solution $\omega \in \mathsf{W}^{2}\mathrm{L}^{p,\mu}\left(\Omega; X \right)$  with 
		$\int_{\Omega} u = 0,$ to the following Neumann boundary value problem, 
		\begin{equation} \label{problemNeumannLaplacian}
			\left\lbrace \begin{aligned}
				\Delta u &= f  &&\text{ in } \Omega, \\
				\frac{\partial u}{\partial \nu}	 &= \frac{\partial u_{0}}{\partial \nu} &&\text{  on } \partial\Omega,
			\end{aligned} 
			\right. \tag{$\mathcal{P}_{Neumann}$}
		\end{equation}
		and there exists a constant $C_{\text{e}} = C_{\text{e}} \left(m,n,p,\mu, \Omega \right) >0$ such that 
		\begin{align*}
			\left\lVert \omega \right\rVert_{\mathsf{W}^{2}\mathrm{L}^{p,\mu}} \leq  C_{\text{e}} \left( \left\lVert f\right\rVert_{\mathrm{L}^{p,\mu}}  + \left\lVert u_{0}\right\rVert_{\mathsf{W}^{2}\mathrm{L}^{p,\mu}}\right).  
		\end{align*}	
	\end{proposition}
	\begin{proof}
		Existence of weak solution is standard. For the estimates, apply the estimates in Proposition \ref{Gaffney ineq in Morrey spaces} to $\omega = du.$
	\end{proof}
	\subsubsection{Estimates for the critical case}
	\begin{lemma}\label{zero boundary value estimate lemma}
		Let $n \geq 3, N \geq 1$ be integers and and let $X$ be any $N$-dimensional real vector space. Let $1 < m \leq \frac{n}{2}$ be a real number and $\Omega \subset \mathbb{R}^{n}$ be an open, bounded, smooth, contractible set. Let $A \in \mathrm{L}^{2m, n-2m}\left(\Omega;\mathbb{R}^{n}\otimes X\right).$ If $\alpha \in W_{0}^{1,2}\left( \Omega; X\right)\cap \mathsf{W}^{1}\mathrm{L}^{2,n-2m}\left( \Omega; X\right) $ be a weak solution to \begin{align}\label{critical elliptic with zero boundary}
			\left\lbrace \begin{aligned}
				\Delta \alpha &= A\cdot \nabla \alpha + F &&\text{ in } \Omega, \\
				\alpha &=0 &&\text{ on } \partial\Omega,
			\end{aligned} \right. 
		\end{align}
		with $F \in \mathrm{L}^{q, n-2m}\left( \Omega; X\right)$ for some $ \frac{2m}{m+1} \leq q < 2m, $ then there exists a small constant $\varepsilon_{1}= \varepsilon_{1}\left( n,m, N, q, \Omega \right) > 0$ such that if $$ \left\lVert A \right\rVert_{\mathrm{L}^{2m, n-2m}\left(\Omega;\mathbb{R}^{n}\otimes X\right)} \leq \varepsilon_{1},$$ then we have $\alpha \in \mathsf{W}^{2}\mathrm{L}^{q,n-2m}\left( \Omega; X\right).$ Moreover, there exists a constant $C_{1} = C_{1}\left(n,m, N, q, \Omega\right) > 0$ such that 
		\begin{align*}
			\left\lVert \alpha \right\rVert_{\mathsf{W}^{2}\mathrm{L}^{q,n-2m}\left( \Omega; X\right)} \leq C_{1} 	\left\lVert F \right\rVert_{\mathrm{L}^{q,n-2m}\left( \Omega; X\right)}.
		\end{align*}
	\end{lemma}
	\begin{proof}
		The proof of the claim is via a fixed point argument coupled with uniqueness. With $q$ as in the lemma, for any $v \in \mathsf{W}^{2}\mathrm{L}^{q,n-2m},$  let $T(v)$ be the solution of the equation $$\Delta \left( T(v) \right) = A.\nabla v + F$$ with zero boundary values. By H\"{o}lder inequality and Adam's embedding, we have 
		\begin{align*}
			\left\lVert A.\nabla v \right\rVert_{\mathrm{L}^{q, n-2m}} &\leq \left\lVert A \right\rVert_{\mathrm{L}^{2m, n-2m}} \left\lVert \nabla v \right\rVert_{\mathrm{L}^{\frac{2mq}{2m -q},n-2m}} \\
			&\leq C_{\text{S}}\left\lVert A \right\rVert_{\mathrm{L}^{2m, n-2m}} \left\lVert v \right\rVert_{\mathsf{W}^{2}\mathrm{L}^{q,n-2m}}. 
		\end{align*} Thus the term $A.\nabla v $ in the right hand side is $\mathrm{L}^{q,n-2m}$ and since $F \in \mathrm{L}^{q,n-2m}$ as well. By Proposition \ref{Morrey estimates for the Laplacian}, (i), we deduce  
		\begin{align*}
			\left\lVert T(v)\right\rVert_{\mathsf{W}^{2}\mathrm{L}^{q,n-2m}} &\leq C_{\text{e}}  \left\lVert \Delta \left( T(v) \right) \right\rVert_{\mathrm{L}^{q, n-2m}} \\&\leq C_{\text{e}}C_{\text{S}}\left\lVert A \right\rVert_{\mathrm{L}^{2m, n-2m}} \left\lVert v \right\rVert_{\mathsf{W}^{2}\mathrm{L}^{q,n-2m}} + C_{\text{e}}\left\lVert F \right\rVert_{\mathrm{L}^{q,n-2m}}. 
		\end{align*} Since $q \geq \frac{2m}{m+1},$ this implies $T$ maps $\mathsf{W}^{2}\mathrm{L}^{q,n-2m}\cap W_{0}^{1,2}$ to itself. Also, for any $v,w \in \mathsf{W}^{2}\mathrm{L}^{q,n-2m}\cap W^{1,2}_{0},$ we have the estimate $$ \left\lVert T(v)- T(w)\right\rVert_{\mathsf{W}^{2}\mathrm{L}^{q,n-2m}}\leq  C_{\text{e}}C_{\text{S}}\left\lVert A \right\rVert_{\mathrm{L}^{2m, n-2m}} \left\lVert  v - w \right\rVert_{\mathsf{W}^{2}\mathrm{L}^{q,n-2m}}. $$  Then we can choose $\left\lVert A \right\rVert_{\mathrm{L}^{2m, n-2m}} $ small enough such that $T$ is a contraction and conclude the existence of an unique fixed point $v_{0} \in \mathsf{W}^{2}\mathrm{L}^{q,n-2m}\cap W_{0}^{1,2}$ by Banach fixed point theorem. But then $v_{0}$ and $\alpha$ are both in  $W_{0}^{1,2}\left( \Omega; X\right)\cap \mathsf{W}^{1}\mathrm{L}^{2,n-2m}\left( \Omega; X\right) $ and are weak solutions of \eqref{critical elliptic with zero boundary}. Thus, we have the estimate  
		\begin{multline*}
			\left\lVert \nabla \alpha - \nabla v_{0}\right\rVert_{\mathrm{L}^{2, n-2m}} \leq C\left\lVert \Delta \left( \alpha - v_{0} \right) \right\rVert_{\mathrm{L}^{\frac{2m}{m+1}, n-2m}} \\\leq C \left\lVert A \right\rVert_{\mathrm{L}^{2m, n-2m}} \left\lVert \nabla \alpha - \nabla v_{0}  \right\rVert_{\mathrm{L}^{2, n-2m}}. 
		\end{multline*} Thus we can choose $\left\lVert A \right\rVert_{\mathrm{L}^{2m, n-2m}} $ small enough such that $C \left\lVert A \right\rVert_{\mathrm{L}^{2m, n-2m}}  < 1$ and hence the above estimate forces $\alpha=v_{0}.$ The result follows.
	\end{proof}
	\begin{lemma}[Elliptic estimate in Morrey-Sobolev setting]\label{ellipticCritical}
		Let $n \geq 3, N \geq 1$ be integers and and let $X$ be any $N$-dimensional real vector space. Let $1 < m \leq \frac{n}{2}$ and $\frac{2m}{m+1} \leq p < 2m$ be real numbers and $\Omega \subset \mathbb{R}^{n}$ be an open, bounded, smooth, contractible set.  Let $A \in \mathrm{L}^{2m, n-2m}\left(\Omega;\mathbb{R}^{n}\otimes X \right)$ and $f \in \mathrm{L}^{p, n-2m}\left(\Omega; X\right).$ There exists a constant $\varepsilon_{\Delta_{Cr}} = \varepsilon_{\Delta_{Cr}}\left( n,N,m, p, \Omega \right) > 0$ such that if $$ \left\lVert A \right\rVert_{\mathrm{L}^{2m, n-2m}\left(\Omega;\mathbb{R}^{n}\otimes X\right)} \leq \varepsilon_{\Delta_{Cr}},$$ then any weak solution $u \in W^{1,2}\left(\Omega; X\right)\cap \mathsf{W}^{1}\mathrm{L}^{2, n-2m}\left(\Omega; X\right)$ of 
		\begin{equation}\label{critical elliptic eqn}
			\Delta u = A\cdot \nabla u + f \qquad \text{ in } \Omega, 
		\end{equation} is in fact in $\mathsf{W}_{\text{loc}}^{2}\mathrm{L}^{p, n-2m}\left(\Omega; X\right).$ Furthermore, for any compact set $K \subset \subset \Omega, $  the exists a constant $C = C( \varepsilon_{\Delta_{Cr}}, n, N, m, p, \Omega, K ) \geq 1$ such that we have the estimate 
		\begin{equation}
			\left\lVert u \right\rVert_{\mathsf{W}^{2}\mathrm{L}^{p, n-2m}\left(K; \mathbb{R}^{n}\otimes X\right)} \leq C \left( \left\lVert u \right\rVert_{\mathsf{W}^{1}\mathrm{L}^{2, n-2m}\left(\Omega; X\right)} + \left\lVert f \right\rVert_{\mathrm{L}^{p, n-2m}\left(\Omega;X\right)} \right).  \end{equation}
		Moreover, the smallness parameter $\varepsilon_{\Delta_{Cr}}$ is scale invariant. More precisely, if $\Omega_{r}= \left\lbrace rx : x \in \Omega \right\rbrace $ is a rescaling of $\Omega,$ then $\Omega$ and $\Omega_{r}$ has the same $\varepsilon_{\Delta_{Cr}}.$ 
	\end{lemma}
	\begin{remark}
		(i) The scale invariance will be used crucially later, since we are going chose balls with small enough radius such that the smallness condition on the scale-invariant $\mathrm{L}^{2m, n-2m}$ norm of $A$ is satisfied, it is necessary that $\varepsilon_{\Delta_{Cr}}$ stays the same as we scale down the radius. Note however that neither the estimate nor the constant $C$ appearing there is scale-invariant. But this would not cause problems for us. \smallskip 
		
		\noindent (ii) The result does not extend to $p=2m.$    
	\end{remark}
	\begin{proof}
		With lemma \ref{zero boundary value estimate lemma} in hand, the proof of this lemma is just about localizing and bootstrapping. To this end, we chose the smallest integer $l \geq 1$ such that $l \geq \left(m+1\right) -\frac{2m}{p}.$ For any $K \subset \subset \Omega,$ we choose open sets $\left\lbrace \Omega_{k}\right\rbrace_{1 \leq k \leq l}$ such that 
		\begin{align*}
			K \subset \subset \Omega_{l} \subset \subset \ldots \Omega_{k+1} \subset \subset \Omega_{k} \subset \subset \ldots \Omega_{1}\subset \subset \Omega.  
		\end{align*}
		Now we shall show that we can choose $\left\lVert A \right\rVert_{\mathrm{L}^{2m, n-2m}}$ small enough such that for any solution  of 
		\eqref{critical elliptic eqn}, 
		\begin{align*}
			u \in \mathsf{W}^{1}\mathrm{L}^{2, n-2m}\left(\Omega; X\right) \Rightarrow u \in \mathsf{W}^{2}\mathrm{L}^{\frac{2m}{m-l+1}, n-2m}\left(\Omega_{l}; X \right).
		\end{align*}
		We prove this by induction over $k.$ Let $1 \leq k \leq l-1$ and choose a smooth cutoff function $\zeta \in C_{c}^{\infty}\left(\Omega_{k}\right)$ such that $\zeta \equiv 1$ in a neighborhood of $\Omega_{k+1}.$ Since $u$ is a solution to \eqref{critical elliptic eqn}, by the induction hypothesis, $\zeta u \in W_{0}^{1,2} \left(\Omega_{k}; X \right) \cap \mathsf{W}^{2}\mathrm{L}^{\frac{2m}{m-k+1}, n-2m}\left(\Omega_{k}; X \right)$ solves 
		\begin{align}\label{localized equation}
			\Delta \left(\zeta u \right) = A \cdot \nabla \left(\zeta u \right) + \zeta f + u \left( \Delta \zeta -  A \cdot \nabla \zeta \right) + 2 \nabla \zeta \cdot \nabla u  \qquad \text{ in } \Omega_{l}.  
		\end{align}
		We set $\alpha = \zeta u$ and $F = \zeta f + u \left( \Delta \zeta -  A \cdot \nabla \zeta \right) + 2 \nabla \zeta \cdot \nabla u $ and plan to use lemma \ref{zero boundary value estimate lemma}. Note that by our choice of $l,$ the Morrey-integrability of $F$ is determined by the least regular third term, which is in $\mathrm{L}^{\frac{2m}{m-k}, n-2m}\left(\Omega_{k}; X \right).$ Thus, using lemma \ref{zero boundary value estimate lemma} with $q = \frac{2m}{m-k},$ we obtain $\zeta u \in \mathsf{W}^{2}\mathrm{L}^{\frac{2m}{m-(k+1)+1}, n-2m}\left(\Omega_{k}; X \right).$ Since $\phi$ is identically $1$ in a neighborhood of $\Omega_{k+1},$ this proves the induction step.  The same argument shows that $u \in \mathsf{W}^{2}\mathrm{L}^{2, n-2m}\left(\Omega_{1}; X \right),$ as we have $u \in \mathsf{W}^{1}\mathrm{L}^{2, n-2m}\left(\Omega; X\right).$ Thus we can start the induction and conclude that $u \in \mathsf{W}^{2}\mathrm{L}^{\frac{2m}{m-l+1}, n-2m}\left(\Omega_{l}; X \right).$ Once again we choose a smooth cut-off function $\zeta \in C_{c}^{\infty}\left(\Omega_{l}\right)$ such that $\zeta \equiv 1$ in a neighborhood of $K.$ Once again, $\zeta u $ satisfies \eqref{localized equation} in $\Omega_{l}.$ But since $l \geq \left(m+1\right) -\frac{2m}{p}$, this time the Morrey-integrability of $F$ is determined by the first term $\zeta f,$ which is $\mathrm{L}^{p, n-2m}\left(\Omega_{l}; \mathbb{R}^{N} \right).$ Thus applying lemma \ref{zero boundary value estimate lemma} once again, we deduce that $u \in \mathsf{W}^{2}\mathrm{L}^{p, n-2m}\left(K; \mathbb{R}^{n}\otimes X\right).$ We finally choose the smallness parameter to be the minimum of the smallness parameters in the finitely many steps. Combining the estimates in each step yields the final estimate.

		Now we show the scale invariance. If $u, A, f$ are as in the hypothesis of the lemma and satisfies \eqref{critical elliptic eqn} in $\Omega_{r},$ then the rescaled maps \begin{align*}
			\tilde{u}(x):= u (rx), \tilde{A}(x):= r A(rx) \text{ and } \tilde{f}(x):= r^{2}f(rx)
		\end{align*}
		satisfies \eqref{critical elliptic eqn} in $\Omega.$ The scale invariance of the smallness parameter now follows from the fact that $$\left\lVert \tilde{A} \right\rVert_{\mathrm{L}^{2m, n-2m}\left(\Omega;\mathbb{R}^{n}\otimes X \right)} = \left\lVert A \right\rVert_{\mathrm{L}^{2m, n-2m}\left(\Omega_{r};\mathbb{R}^{n}\otimes X \right)}.$$
		This completes the proof. \end{proof}
	\subsection{Existence of Morrey-Sobolev connections}
	Here we prove our class of bundle-connection pairs is not empty. 
	\begin{theorem}\label{existence of connection}
		Every $\mathsf{W}^{2}\mathrm{L}^{m, n-2m},$ respectively $\mathsf{W}^{2}\mathrm{VL}^{m, n-2m},$ bundle admits a $\mathsf{W}^{1}\mathrm{L}^{m, n-2m},$ respectively $\mathsf{W}^{1}\mathrm{VL}^{m, n-2m},$ connection. Moreover, the $\mathsf{W}^{1}\mathrm{L}^{m, n-2m}$ norm of the connection ( also the Morrey modulii ) can be estimated in terms of the same for the transition functions. 
	\end{theorem}
	\begin{proof}
		Let $P = \left( \left\lbrace U_{\alpha}\right\rbrace_{\alpha \in I}, \left\lbrace g_{\alpha\beta}\right\rbrace_{\alpha, \beta \in I}\right)$ be a $\mathsf{W}^{2}\mathrm{L}^{m, n-2m}$ bundle. Pick a partition of unity $\left\lbrace \zeta_{\alpha} \right\rbrace_{\alpha}$ subordinate to the cover $\left\lbrace U_{\alpha}\right\rbrace_{\alpha \in I}$ such that 
		\begin{align*}
			\sup\limits_{\alpha \in I}\left\lvert \nabla \zeta_{\alpha} \right\rvert < C,
		\end{align*} for some constant $C>0,$ dependng on the cover, and define 
		\begin{equation}\label{definingconn}
			A_{\alpha} := \sum_{\substack{\beta \in I ,\beta \neq \alpha, \\ U_{\alpha}\cap U_{\beta} \neq \emptyset}} \zeta_{\beta} g_{\beta\alpha}^{-1} dg_{\beta\alpha} := \sum_{\substack{\beta \in I ,\beta \neq \alpha, \\ U_{\alpha}\cap U_{\beta} \neq \emptyset}} - \zeta_{\beta} dg_{\alpha\beta}g_{\alpha\beta}^{-1}    \qquad \text{ for each } \alpha \in I.
		\end{equation}
		By a straight forward computation, we can verify the gluing condition \begin{equation*}
			A_{\beta} = g_{\alpha\beta}^{-1} d g_{\alpha\beta} + g_{\alpha\beta}^{-1} A_{\alpha} g_{\alpha\beta} \qquad \text{ a.e. in } U_{\alpha} \cap U_{\beta}
		\end{equation*} holds for every $\alpha, \beta \in I$ whenever $U_{\alpha} \cap U_{\beta} \neq \emptyset.$  It is also easy to check $ A_{\alpha} \in \mathsf{W}^{1}\mathrm{L}^{m, n-2m}$ with corresponding estimates. This completes the proof.
	\end{proof}

	\subsection{Coulomb gauges and Coulomb bundles}
	\subsubsection{Local Coulomb gauges in Morrey spaces}
	We first prove the existence of local Coulomb gauges for smooth connections with curvature being small in the critical Morrey spaces. The scheme of the proof is well known and goes back to Uhlenbeck \cite{UhlenbeckGaugefixing} for the Lebesgue case $n=2m=4.$ Adaptation to the Morrey case when $m=2$ and $n \geq 5$ is obtained in Theorem 3.7 in \cite{MeyerRiviere_YangMills} and independently also in Theorem 4.6 in \cite{TaoTian_YangMills}. The main point of the proof in Meyer-Rivi\`{e}re \cite{MeyerRiviere_YangMills} is an interpolation inequality. This inequality is known to hold in $L^{p}$ setup too ( see \cite{Strzelecki_BMOGagliardoNirenberg} ) and can be used to prove our Theorem \ref{existence of Coulomb gauges in better spaces}. However, we provide a proof based on Gaffney inequality in Morrey spaces.  
	\begin{theorem}[Existence of Coulomb gauges for smooth connections]\label{existence of Coulomb gauges in better spaces}
		Let $n \geq 4$ be an integer and let $1 < m \leq \frac{n}{2} ,$ $r>0 $ be real numbers, $x_{0} \in \mathbb{R}^{n}$ and let $B_{r}(x_{0})\subset \mathbb{R}^{n}$ be the ball of radius $r$ around $x_{0}$. Then there exist constants $\varepsilon_{Coulomb} = \varepsilon_{Coulomb}\left( G, n \right)  > 0 $ and $C_{Coulomb}= C_{Coulomb}\left( G, n \right)$ such that for any $A \in C^{\infty}\left(\overline{B_{r}(x_{0})}; \Lambda^{1}\mathbb{R}^{n}\otimes \mathfrak{g}\right)$  and 
		\begin{align}
			\left\lVert F_{A}\right\rVert_{\mathrm{L}^{m,n-2m}\left( B_{r}(0);\Lambda^{2} T^{\ast}B_{r}(0)\otimes \mathfrak{g} \right)} \leq \varepsilon_{Coulomb},
		\end{align}
		there exists $\rho \in \mathsf{W}^{2}\mathrm{L}^{m,n-2m}\left(B_{r}(x_{0}) ; G \right)$ such that 
		\begin{align}\label{Coulomb gauge BVP}
			\left\lbrace \begin{aligned}
				&d^{\ast} A^{\rho} = 0 \quad \text{ in } B_{r}(x_{0}),  \\
				&\iota^{\ast}_{\partial B_{r}(x_{0}) } \left( \ast A^{\rho}\right)  = 0 \quad \text{ on } \partial B_{r}(x_{0})	\end{aligned}\right.
		\end{align}
		and we have the estimates 
		\begin{align}\label{Coulomb gauge coercivity}
			\left\lVert \nabla A^{\rho}\right\rVert_{\mathrm{L}^{m,n-2m}} + \left\lVert  A^{\rho}\right\rVert_{\mathrm{L}^{2m,n-2m}}  \leq C_{Coulomb} \left\lVert F_{A}\right\rVert_{\mathrm{L}^{m,n-2m}}, 
		\end{align}	
		\begin{align}\label{Coulomb gauge first derivative}
			\left\lVert  d\rho\right\rVert_{\mathrm{L}^{2m,n-2m}} \leq C_{G}\left( C_{Coulomb}\left\lVert F_{A}\right\rVert_{\mathrm{L}^{m,n-2m}}  + \left\lVert  A\right\rVert_{\mathrm{L}^{2m,n-2m}} \right)
		\end{align}
		and 
		\begin{align}
			\left\lVert  \nabla d\rho\right\rVert_{\mathrm{L}^{m,n-2m}} \leq C_{G}&\left( \left\lVert  \nabla A\right\rVert_{\mathrm{L}^{m,n-2m}}  + C_{Coulomb}\left\lVert F_{A}\right\rVert_{\mathrm{L}^{m,n-2m}}\right) 
			\notag \\ &+C_{G}\left( C_{Coulomb}\left\lVert F_{A}\right\rVert_{\mathrm{L}^{m,n-2m}}  + \left\lVert  A\right\rVert_{\mathrm{L}^{2m,n-2m}} \right)^{2},\label{Coulomb gauge second derivative}
		\end{align}
		where $C_{G} \geq 1 $ is an $L^{\infty}$ bound for $G.$	
	\end{theorem}
	\begin{proof}
		The translation invariance is of course obvious. So we can assume $x_{0}= 0 \in \mathbb{R}^{n}.$ Notice that if $A: B_{r} \rightarrow \Lambda^{1}\mathbb{R}^{n}\otimes \mathfrak{g}$ is a connection, the rescaled connection $\tilde{A}(x):= r A(rx)$ is a connection on $B_{1}$ with curvature $F_{\tilde{A}}(x) = r^{2} F_{A}(rx)$ and
		$ \nabla \tilde{A} \left(x \right) = r^{2}\nabla A \left(rx\right).$ From this we can easily check that for any $z_{0} \in B_{1}$ and any $0 < \rho < 1,$ we have the identities
		\begin{align*}
			\frac{1}{\rho^{n-2m}}\int_{B_{\rho}\left( z_{0}\right)\cap B_{1}}\left\lvert \tilde{A} \right\rvert^{2m} &= \frac{1}{\left( r\rho\right)^{n-2m}}\int_{B_{r\rho}\left( z_{0}\right)\cap B_{r}}\left\lvert A \right\rvert^{2m}, \\
			\frac{1}{\rho^{n-2m}}\int_{B_{\rho}\left( z_{0}\right)\cap B_{1}}\left\lvert \nabla \tilde{A} \right\rvert^{m} &= \frac{1}{\left( r\rho\right)^{n-2m}}\int_{B_{r\rho}\left( rz_{0}\right)\cap B_{r}}\left\lvert \nabla A \right\rvert^{m}, 
			\intertext{ and }  \frac{1}{\rho^{n-2m}}\int_{B_{\rho}\left( z_{0}\right)\cap B_{1}}\left\lvert F_{\tilde{A}} \right\rvert^{m} &= \frac{1}{\left( r\rho\right)^{n-2m}}\int_{B_{r\rho}\left( rz_{0}\right)\cap B_{r}}\left\lvert F_{A} \right\rvert^{m} .
		\end{align*} This implies the scale invariance of all the relevant Morrey norms. Similarly, for any map $g:B_{r} \rightarrow G$, we use the rescaled map $\tilde{g}: B_{1} \rightarrow G$ defined by $g\left(x\right):= g\left(rx\right).$ Hence without loss of generality, we can assume $r=1$ and $x_{0}=0 \in \mathbb{R}^{n}.$ \smallskip 
		
		Now fix $\kappa >0.$ For any $\varepsilon>0$ and $C > 0, $ we define the sets 
		\begin{align*}
			U_{\varepsilon}&:= \left\lbrace \begin{multlined}
				A \in \mathsf{W}^{1}\mathrm{L}^{m,n-2m+ \kappa}\left(B_{1}; \Lambda^{1}\mathbb{R}^{n}\otimes \mathfrak{g}\right) \\ \text{ such that } \left\lVert F_{A}\right\rVert_{\mathrm{L}^{m,n-2m}\left( B_{1};\Lambda^{2} \mathbb{R}^{n}\otimes \mathfrak{g} \right)} \leq \varepsilon.
			\end{multlined}\right\rbrace \intertext{ and }
			V_{\varepsilon}^{C_{0}}&:= \left\lbrace \begin{aligned}
				A &\in U_{\varepsilon} \text{ such that  there exists } \rho \in  \mathsf{W}^{2}\mathrm{L}^{m,n-2m+ \kappa}\left(B_{1} ; G \right)\\ &\text{ so that  } A^{\rho}   \text{ satisfies } \eqref{Coulomb gauge BVP} \text{ and the estimates } \\& 
				\left\lVert \nabla A^{\rho}\right\rVert_{\mathrm{L}^{m,n-2m}} + \left\lVert  A^{\rho}\right\rVert_{\mathrm{L}^{2m,n-2m}} \leq C_{0} \left\lVert F_{A}\right\rVert_{\mathrm{L}^{m,n-2m}}, \\
				&\left\lVert \nabla A^{\rho}\right\rVert_{\mathrm{L}^{m,n-2m+\kappa}} + \left\lVert  A^{\rho}\right\rVert_{\mathrm{L}^{2m,n-2m+2\kappa}} \leq C_{\kappa} \left\lVert F_{A}\right\rVert_{\mathrm{L}^{m,n-2m +\kappa}}. 
			\end{aligned}\right\rbrace ,  
		\end{align*}
		for some constant $C_{\kappa} >0,$ depending only on $m, n, G, \mathfrak{g}, \kappa. $
		Now we want to show that there exists $\varepsilon >0$ and $C_{0}>0$ such that 
		$$ U_{\varepsilon} = V_{\varepsilon}^{C_{0}}.$$
		
		\noindent\textbf{\textit{$U_{\varepsilon}$ is connected: }} Since $\kappa >0,$ it is easy to check that for any $\varepsilon >0,$ $U_{\varepsilon}$ is path-connected and hence connected. Indeed, for any $A \in U_{\varepsilon}$ the connection $A_{t}\left(x\right):= t A\left( tx\right) \in U_{\varepsilon}$ for any $t \in [0,1]$ and consequently $U_{\varepsilon}$ is star-shaped about the zero connection. Thus, the proof boils down to showing the existence of $\varepsilon>0$ and $C_{0}>0$ such that $V_{\varepsilon}^{C_{0}}$ is both open and closed in $U_{\varepsilon}.$ \bigskip

		\noindent\textbf{\textit{Openness of $V_{\varepsilon}^{C_{0}}$: }} Let $A \in V_{\varepsilon}^{C_{0}}.$ Without loss of generality, we can assume that $A$ satisfies 
		\begin{align*}
			\left\lbrace \begin{aligned}
				d^{\ast} A &= 0 &&\text{ in } B_{1},  \\
				\iota^{\ast}_{\partial B_{1} } \left( \ast A \right)  &= 0 &&\text{ on } \partial B_{1}.	\end{aligned}\right. \end{align*}
		and the estimate \begin{align*}
			\left\lVert \nabla A\right\rVert_{\mathrm{L}^{m,n-2m}} + \left\lVert  A\right\rVert_{\mathrm{L}^{2m,n-2m}} \leq C_{0} \left\lVert F_{A}\right\rVert_{\mathrm{L}^{m,n-2m}} \leq C_{0}\varepsilon.	
		\end{align*}
		We want to find $\delta >0$ sufficiently small, possibly depending on $A$ such that for any $\omega$ with $$\left\lVert \omega \right\rVert_{\mathsf{W}^{1}\mathrm{L}^{m,n-2m+ \kappa}\left(B_{1}; \Lambda^{1}\mathbb{R}^{n}\otimes \mathfrak{g}\right)} < \delta, $$
		there exists $g \in \mathsf{W}^{2}\mathrm{L}^{m,n-2m+ \kappa}\left(B_{1} ; G \right),$ such that $g$ is close to the constant map  $\mathbbm{1}_{G}$ in $\mathsf{W}^{2}\mathrm{L}^{m,n-2m+ \kappa}\left(B_{1} ; G \right)$ norm and satisfies 
		\begin{align*}
			\left\lbrace \begin{aligned}
				d^{\ast} \left[ g^{-1}dg + g^{-1}\left(A + \omega\right)g\right] &= 0 &&\text{ in } B_{1},  \\
				\iota^{\ast}_{\partial B_{1} } \left( \ast \left[ A +\omega\right]^{g}\right)  &= 0 &&\text{ on } \partial B_{1}.	\end{aligned}\right.
		\end{align*}
		Now let us set 
		\begin{align*}
			\mathscr{X}_{1} &:= \mathsf{W}^{1}\mathrm{L}^{m,n-2m+ \kappa}\left(B_{1}; \Lambda^{1}\mathbb{R}^{n}\otimes \mathfrak{g}\right), \\
			\mathscr{X}_{2}&:= \left\lbrace U \in \mathsf{W}^{2}\mathrm{L}^{m,n-2m+ \kappa}\left(B_{1}; \mathfrak{g}\right): \fint_{B_{1}} U = 0 \right\rbrace, \\ 
			\mathscr{Y}&:=  	\left[ \iota^{\ast}_{\partial B_{1} } \circ  \ast \right] \left( X_{1}\right), \\
			\mathscr{Z}&:= \left\lbrace \left( f, g\right)\in \mathrm{L}^{m,n-2m+ \kappa}\left(B_{1}; \mathfrak{g}\right)\times Y: \int_{B_{1}}f + \int_{\partial B_{1}}g = 0 \right\rbrace, 
		\end{align*}
		where $\left[ \iota^{\ast}_{\partial B_{1} } \circ  \ast \right]$ denotes the normal trace map and thus $Y$ is the range of the normal trace map. It is easy to see all these spaces are Banach spaces. We consider the nonlinear map 
		$	\mathcal{N}^{A}:\mathscr{X}_{1}\times \mathscr{X}_{2}\rightarrow \mathscr{Z} ,$ defined by 
		\begin{align*}
			\mathcal{N}^{A}\left(\omega, U\right) = \left( d^{\ast} \left[ g_{U}^{-1}dg_{U} + g_{U}^{-1}\left(A + \omega\right)g_{U}\right], 	\iota^{\ast}_{\partial B_{1} }  \ast \left[ A +\omega\right]^{g_{U}} \right)
		\end{align*}
		Since  $\mathsf{W}^{2}\mathrm{L}^{m,n-2m + \kappa}$ embeds into $C^{0},$  the map $\mathcal{N}^{A}$ is $C^{1}.$ The partial derivative along the $U$ direction at $(0,0)$ is 
		\begin{align*}
			\partial_{U}	\mathcal{N}^{A}\left(0, 0\right)\left[ V \right] &= \left(d^{\ast}d V + \ast \left[ dV \wedge \left(  \ast A \right)\right], \partial_{r}V\right) \\
			&= \left(\Delta  V + \ast \left[ dV \wedge \left(  \ast A \right)\right], \partial_{r}V\right)\qquad \text{ for all } V \in \mathscr{X}_{2}. 
		\end{align*} Now Proposition \ref{Morrey estimates for the Laplacian}, (ii) implies that the linear map 
		$T: \mathscr{X}_{2} \rightarrow \mathscr{Z} ,$  defined by 
		\begin{align*}
			T\left[V\right] = \left(  \Delta V , \partial_{r}V\right)
		\end{align*} is invertible. Now we can easily estimate 
		\begin{multline*}
			\left\lVert 	T\left[V\right] - \partial_{U}	\mathcal{N}^{A}\left(0, 0\right)\left[ V \right] \right\rVert_{\mathscr{Z}} \\\begin{aligned}
				&= \left\lVert \ast \left[ dV \wedge \left(  \ast A \right)\right] \right\rVert_{\mathrm{L}^{m,n-2m+ \kappa}\left(B_{1}; \mathfrak{g}\right)} \\&\leq C\left\lVert  A \right\rVert_{\mathrm{L}^{2m,n-2m}\left(B_{1}; \Lambda^{1}\mathbb{R}^{n}\otimes \mathfrak{g}\right)}\left\lVert dV  \right\rVert_{\mathrm{L}^{2m,n-2m+ 2\kappa}\left(B_{1}; \Lambda^{1}\mathbb{R}^{n}\otimes \mathfrak{g}\right)}\\
				&\leq CC_{0}\varepsilon \left\lVert dV  \right\rVert_{\mathrm{L}^{2m,n-2m+ 2\kappa}\left(B_{1}; \Lambda^{1}\mathbb{R}^{n}\otimes \mathfrak{g}\right)}\\&\stackrel{\textbf{emb}}{\leq} CC_{0}\varepsilon\left\lVert V  \right\rVert_{\mathsf{W}^{2}\mathrm{L}^{m,n-2m+ \kappa}\left(B_{1}; \mathfrak{g}\right)}. 
			\end{aligned}
		\end{multline*}
		Thus, we have 
		\begin{align*}
			\left\lVert T - \partial_{U} \mathcal{N}^{A}\left(0, 0\right)\right\rVert_{\mathcal{B}\left(\mathscr{X}_{2}; \mathscr{Z} \right)}
			&:= \sup\limits_{\left\lVert V\right\rVert_{\mathscr{X}_{2}}\leq 1}  \left\lVert 	\left[ T - \partial_{U}	\mathcal{N}^{A}\left(0, 0\right)\right] \left[ V \right] \right\rVert_{\mathscr{Z}} \\
			&\leq CC_{0}\varepsilon. 
		\end{align*}
		Thus, by choosing $\varepsilon >0$ small enough, for any fixed $C_{0},$ we can make the operator norm of $T - \partial_{U} \mathcal{N}^{A}\left(0, 0\right)$ less than $1,$ which would imply, by standard functional analysis arguments via the Neumann series that $\mathcal{N}^{A}\left(0, 0\right)$ is invertible. Now inverse function theorem implies the existence of $\delta >0$ and an open neighborhood $0 \in \mathscr{O} \subset  	\mathscr{X}_{2}$ such that for all $\omega \in \mathscr{X}_{1}$ with 
		\begin{align*}
			\left\lVert \omega \right\rVert_{\mathsf{W}^{1}\mathrm{L}^{m,n-2m+ \kappa}\left(B_{1}; \Lambda^{1}\mathbb{R}^{n}\otimes \mathfrak{g}\right)} < \delta, 
		\end{align*}
		there exists unique $V_{\omega} \in \mathscr{O}$ such that 
		\begin{align*}
			\mathcal{N}^{A}\left(\omega, V_{\omega}\right) = 0. 
		\end{align*}
		Shrinking $\delta$ to shrink $\mathscr{O}$, if necessary, we can ensure $g_{\omega}:=\operatorname{exp}\left(V_{\omega}\right)$ is well defined, $\left( A + \omega \right)^{g_{\omega}} \in U_{\varepsilon}$ and $\left( A + \omega \right)^{g_{\omega}}$  satisfies \eqref{Coulomb gauge BVP}. 
		Now we have the estimate 
		\begin{align*}
			\left\lVert  \left( A + \omega \right)^{g_{\omega}}\right\rVert_{\mathrm{L}^{2m,n-2m}} &\leq C_{G}\left( 	\left\lVert A \right\rVert_{\mathrm{L}^{2m,n-2m}} + 	\left\lVert dg_{\omega}\right\rVert_{\mathrm{L}^{2m,n-2m}} +	\left\lVert  \omega \right\rVert_{\mathrm{L}^{2m,n-2m}}\right) \\
			&\leq C_{G}\left( 	\left\lVert F_{A} \right\rVert_{\mathrm{L}^{m,n-2m}} + 	\left\lVert V_{\omega}\right\rVert_{\mathsf{W}^{2}\mathrm{L}^{m,n-2m}} +	\left\lVert  \omega \right\rVert_{\mathsf{W}^{1}\mathrm{L}^{m,n-2m}}\right) \\
			&\leq C_{G}\left( 	\varepsilon + 	\left\lVert V_{\omega}\right\rVert_{\mathsf{W}^{2}\mathrm{L}^{m,n-2m}} +	\left\lVert  \omega \right\rVert_{\mathsf{W}^{1}\mathrm{L}^{m,n-2m}}\right).
		\end{align*}
		Thus, the last two terms can be made arbitrarily small by choosing $\delta >0$ small enough. Thus, we have 
		\begin{multline*}
			\left\lVert  d\left( A + \omega \right)^{g_{\omega}}\right\rVert_{\mathrm{L}^{m,n-2m}} \\
			\begin{aligned}
				&\leq \left\lVert F_{A+\omega}\right\rVert_{\mathrm{L}^{m,n-2m}} + 	\left\lVert  \left( A + \omega \right)^{g_{\omega}}\wedge \left( A + \omega \right)^{g_{\omega}} \right\rVert_{\mathrm{L}^{m,n-2m}} \\ 
				&\leq \left\lVert F_{A+\omega}\right\rVert_{\mathrm{L}^{m,n-2m}} + 	\left\lVert  \left( A + \omega \right)^{g_{\omega}}\right\rVert_{\mathrm{L}^{2m,n-2m}}^{2} \\
				&\leq \left\lVert F_{A+\omega}\right\rVert_{\mathrm{L}^{m,n-2m}} + 	\left\lVert  \left( A + \omega \right)^{g_{\omega}}\right\rVert_{\mathrm{L}^{2m,n-2m}}\left\lVert  d\left( A + \omega \right)^{g_{\omega}}\right\rVert_{\mathrm{L}^{m,n-2m}} \\&\leq \left\lVert F_{A+\omega}\right\rVert_{\mathrm{L}^{m,n-2m}} + 	\frac{1}{2}\left\lVert  d\left( A + \omega \right)^{g_{\omega}}\right\rVert_{\mathrm{L}^{m,n-2m}}.
			\end{aligned}
		\end{multline*}
		This implies the estimate 
		\begin{align*}
			\left\lVert  d\left( A + \omega \right)^{g_{\omega}}\right\rVert_{\mathrm{L}^{m,n-2m}} \leq 2 \left\lVert F_{A+\omega}\right\rVert_{\mathrm{L}^{m,n-2m}}. 
		\end{align*}
		Now Proposition \ref{Gaffney ineq in Morrey spaces}, (ii) concludes the argument. \bigskip 
		
		\noindent \textbf{\textit{Closedness of $V_{\varepsilon}^{C_{0}}$: }} Let $\left\lbrace A_{s} \right\rbrace_{s \in \mathbb{N}} \subset V_{\varepsilon}^{C_{0}}$ be a sequence such that 
		\begin{align}\label{strong convergence Morrey Sobolev non borderline}
			A_{s} \rightarrow A \qquad \text{ strongly in }  \mathsf{W}^{1}\mathrm{L}^{m,n-2m+ \kappa}\left(B_{1}; \Lambda^{1}\mathbb{R}^{n}\otimes \mathfrak{g}\right). 
		\end{align}
		We need to show $A \in V_{\varepsilon}^{C_{0}}$ for some choice of $\varepsilon>0$ small enough and $C_{0} \geq 1$ large enough. 
		Now note that since $A_{s}  \in V_{\varepsilon}^{C_{0}}$ for every $s \in \mathbb{N},$ there exists a sequence $\left\lbrace \rho_{s} \right\rbrace_{s \in \mathbb{N}} \subset \mathsf{W}^{2}\mathrm{L}^{m,n-2m+ \kappa}\left(B_{1} ; G \right)$ such that we have 
		\begin{align}
			\left\lVert \nabla A_{s}^{\rho_{s}}\right\rVert_{\mathrm{L}^{m,n-2m +\kappa}} + \left\lVert  A_{s}^{\rho_{s}}\right\rVert_{\mathrm{L}^{2m,n-2m +2\kappa}} &\leq C_{\kappa} \left\lVert F_{A_{s}}\right\rVert_{\mathrm{L}^{m,n-2m +\kappa}}, \label{Coulomb estimate for closedness sequence kappa} \\
			\left\lVert \nabla A_{s}^{\rho_{s}}\right\rVert_{\mathrm{L}^{m,n-2m}} + \left\lVert  A_{s}^{\rho_{s}}\right\rVert_{\mathrm{L}^{2m,n-2m}} &\leq C_{0} \left\lVert F_{A_{s}}\right\rVert_{\mathrm{L}^{m,n-2m}} \label{Coulomb estimate for closedness sequence}. 
		\end{align}
		Now note that the strong convergence \eqref{strong convergence Morrey Sobolev non borderline} implies that 
		\begin{align}\label{strong convergence Morrey Sobolev non borderline for curvatures}
			F_{A_{s}} \rightarrow F_{A} \qquad \text{ strongly in }  \mathrm{L}^{m,n-2m+ \kappa}\left(B_{1}; \Lambda^{2}\mathbb{R}^{n}\otimes \mathfrak{g}\right). 
		\end{align}
		Thus, RHS of \eqref{Coulomb estimate for closedness sequence} is uniformly bounded. Thus, we deduce 
		\begin{align*}
			\left\lVert d\rho_{s}\right\rVert_{\mathrm{L}^{2m,n-2m + \kappa}} &\leq C_{G} \left\lVert \rho_{s}^{-1}d\rho_{s}\right\rVert_{\mathrm{L}^{2m,n-2m + \kappa}} \\
			&= C_{G} \left\lVert A_{s}^{\rho_{s}} - \rho_{s}^{-1}A_{s}\rho_{s}\right\rVert_{\mathrm{L}^{2m,n-2m + \kappa}} \\
			&\leq C_{G} \left\lVert A_{s}^{\rho_{s}} \right\rVert_{\mathrm{L}^{2m,n-2m + \kappa}} + C_{G}^{3}\left\lVert A_{s} \right\rVert_{\mathrm{L}^{2m,n-2m + \kappa}} \\
		\end{align*}
		Thus, $\left\lbrace d\rho_{s} \right\rbrace_{s \in \mathbb{N}}$ is uniformly bounded in $\mathrm{L}^{2m,n-2m + \kappa}$ and thus, up to the extraction of a subsequence, which we do not relabel, 
		\begin{align*}
			\rho_{s} \rightarrow \rho \qquad \text{ strongly in } C^{0}. 
		\end{align*}
		Thus, $\rho$ is $G$-valued. Going back to \eqref{Coulomb estimate for closedness sequence kappa}, it is also easy to see that this implies $\rho \in \mathsf{W}^{2}\mathrm{L}^{m,n-2m+ \kappa}\left(B_{1} ; G \right),$ $A^{\rho} \in  \mathsf{W}^{1}\mathrm{L}^{m,n-2m+ \kappa}\left(B_{1}; \Lambda^{1}\mathbb{R}^{n}\otimes \mathfrak{g}\right)$  and we have 
		\begin{align*}
			\nabla A_{s}^{\rho_{s}} &\rightharpoonup \nabla A^{\rho} \qquad \text{ weakly in } L^{m}\left(B_{1}; \mathbb{R}^{n}\otimes\Lambda^{1}\mathbb{R}^{n}\otimes \mathfrak{g}\right), \\ \intertext{ and }
			A_{s}^{\rho_{s}} &\rightharpoonup  A^{\rho} \qquad \text{ weakly in } L^{2m}\left(B_{1}; \Lambda^{1}\mathbb{R}^{n}\otimes \mathfrak{g}\right). 
		\end{align*}
		This immediately implies $A^{\rho}$ satisfies \eqref{Coulomb gauge BVP}. Thus, we only need to show that we can choose $C_{0}>0$ large enough such that the we have the estimates 
		\begin{align*}
			\left\lVert \nabla A^{\rho}\right\rVert_{\mathrm{L}^{m,n-2m}} + \left\lVert  A^{\rho}\right\rVert_{\mathrm{L}^{2m,n-2m}} &\leq C_{0} \left\lVert F_{A}\right\rVert_{\mathrm{L}^{m,n-2m}}, \\
			\left\lVert \nabla A^{\rho}\right\rVert_{\mathrm{L}^{m,n-2m+\kappa}} + \left\lVert  A^{\rho}\right\rVert_{\mathrm{L}^{2m,n-2m+2\kappa}} &\leq C_{\kappa} \left\lVert F_{A}\right\rVert_{\mathrm{L}^{m,n-2m +\kappa}}.
		\end{align*}
		We only show the first estimate, the second one is much easier. Note that since $	A_{s}^{\rho_{s}} \rightharpoonup  A^{\rho}$ weakly in $ L^{2m},$ we have 
		\begin{align*}
			\left\lVert  A^{\rho}\right\rVert_{\mathrm{L}^{2m,n-2m}} &\leq \liminf\limits_{s\rightarrow \infty} \left\lVert  A_{s}^{\rho_{s}}\right\rVert_{\mathrm{L}^{2m,n-2m}} \\&\stackrel{\eqref{Coulomb estimate for closedness sequence}}{\leq} C_{0} \liminf\limits_{s\rightarrow \infty}\left\lVert F_{A_{s}}\right\rVert_{\mathrm{L}^{m,n-2m}} \leq  C_{0} \left\lVert F_{A}\right\rVert_{\mathrm{L}^{m,n-2m}} \leq C_{0}\varepsilon.
		\end{align*}
		But this implies 
		\begin{align*}
			\left\lVert  dA^{\rho}\right\rVert_{\mathrm{L}^{m,n-2m}} &\leq \left\lVert F_{A}\right\rVert_{\mathrm{L}^{m,n-2m}} + \left\lVert A^{\rho}\wedge A^{\rho} \right\rVert_{\mathrm{L}^{m,n-2m}} \\
			&\leq \left\lVert F_{A}\right\rVert_{\mathrm{L}^{m,n-2m}} + \left\lVert A^{\rho}\right\rVert^{2}_{\mathrm{L}^{2m,n-2m}} \\
			&\leq\left\lVert F_{A}\right\rVert_{\mathrm{L}^{m,n-2m}} +  C_{0}\varepsilon \left\lVert A^{\rho}\right\rVert_{\mathrm{L}^{2m,n-2m}} \\
			&\leq\left\lVert F_{A}\right\rVert_{\mathrm{L}^{m,n-2m}} +  C_{0}C_{S}\varepsilon \left\lVert A^{\rho}\right\rVert_{\mathsf{W}^{1}\mathrm{L}^{m,n-2m}}. 
		\end{align*}
		Finally, by Proposition \ref{Gaffney ineq in Morrey spaces}, (ii), we deduce  
		\begin{align*}
			\left\lVert A^{\rho}\right\rVert_{\mathsf{W}^{1}\mathrm{L}^{m,n-2m}} \leq 	\left\lVert  dA^{\rho}\right\rVert_{\mathrm{L}^{m,n-2m}} \leq  \left\lVert F_{A}\right\rVert_{\mathrm{L}^{m,n-2m}} + C_{0}C_{S}\varepsilon\left\lVert A^{\rho}\right\rVert_{\mathsf{W}^{1}\mathrm{L}^{m,n-2m}}.
		\end{align*}
		Choosing $\varepsilon$ sufficiently small implies the estimate. This establishes the existence of the Coulomb gauge and the estimate \eqref{Coulomb gauge coercivity}. The estimates \eqref{Coulomb gauge first derivative} and \eqref{Coulomb gauge second derivative} follow from combining \eqref{Coulomb gauge coercivity} with the obvious estimates using the identity
		\begin{align*}
			d\rho &= \rho A^{\rho} - A \rho &&\text{ in } B_{1}. 
		\end{align*}
		This completes the proof. 
	\end{proof}
	By approximation, we can immediately extend this result to vanishing Morrey-Sobolev connections. We state the results and skip the easy details. 
	\begin{theorem}[Existence of Coulomb gauges for vanishing Morrey-Sobolev connections]\label{existence of Coulomb gauges in critical Morrey-Sobolev case} Let $n \geq 4$ be an integer and let $1 < m \leq \frac{n}{2} ,$ $r>0 $ be real numbers, $x_{0} \in \mathbb{R}^{n}$ and let $B_{r}(x_{0})\subset \mathbb{R}^{n}$ be the ball of radius $r$ around $x_{0}$. Then there exist constants $\varepsilon_{Coulomb} = \varepsilon_{Coulomb}\left( G, n \right)  > 0 $ and $C_{Coulomb}= C_{Coulomb}\left( G, n \right)$ such that for any $A \in \mathsf{W}^{1}\mathrm{VL}^{m,n-2m}\left(B_{r}(x_{0}); \Lambda^{1}\mathbb{R}^{n}\otimes \mathfrak{g}\right)$ with
		\begin{align}
			\left\lVert F_{A}\right\rVert_{\mathrm{L}^{m,n-2m}\left( B_{r}(0);\Lambda^{2}\mathbb{R}^{n}\otimes \mathfrak{g} \right)} \leq \varepsilon_{Coulomb},
		\end{align}
		there exists $\rho \in \mathsf{W}^{2}\mathrm{VL}^{m,n-2m}\left(B_{r}(x_{0}) ; G \right)$ such that 
		\begin{align}\label{Coulomb gauge BVP VMS}
			\left\lbrace \begin{aligned}
				&d^{\ast} A^{\rho} = 0 \quad \text{ in } B_{r}(x_{0}),  \\
				&\iota^{\ast}_{\partial B_{r}(x_{0}) } \left( \ast A^{\rho}\right)  = 0 \quad \text{ on } \partial B_{r}(x_{0})	\end{aligned}\right.
		\end{align}
		and we have the estimates 
		\begin{align}\label{Coulomb gauge coercivity VMS}
			\left\lVert \nabla A^{\rho}\right\rVert_{\mathrm{L}^{m,n-2m}} + \left\lVert  A^{\rho}\right\rVert_{\mathrm{L}^{2m,n-2m}}  \leq C_{Coulomb} \left\lVert F_{A}\right\rVert_{\mathrm{L}^{m,n-2m}}, 
		\end{align}	
		\begin{align}\label{Coulomb gauge first derivative VMS}
			\left\lVert  d\rho\right\rVert_{\mathrm{L}^{2m,n-2m}} \leq C_{G}\left( C_{Coulomb}\left\lVert F_{A}\right\rVert_{\mathrm{L}^{m,n-2m}}  + \left\lVert  A\right\rVert_{\mathrm{L}^{2m,n-2m}} \right)
		\end{align}
		and 
		\begin{align}
			\left\lVert  \nabla d\rho\right\rVert_{\mathrm{L}^{m,n-2m}} \leq C_{G}&\left( \left\lVert  \nabla A\right\rVert_{\mathrm{L}^{m,n-2m}}  + C_{Coulomb}\left\lVert F_{A}\right\rVert_{\mathrm{L}^{m,n-2m}}\right) 
			\notag \\ &+C_{G}\left( C_{Coulomb}\left\lVert F_{A}\right\rVert_{\mathrm{L}^{m,n-2m}}  + \left\lVert  A\right\rVert_{\mathrm{L}^{2m,n-2m}} \right)^{2},\label{Coulomb gauge second derivative VMS}
		\end{align}
		where $C_{G} \geq 1 $ is an $L^{\infty}$ bound for $G.$	
	\end{theorem}
	
	Now we slightly sharpen the result in Proposition \ref{approximation in Morrey-Sobolev spaces} (i). 
	\begin{lemma}[Smooth approximation in Lebesgue norms for Morrey-Sobolev connections]\label{approximabilityconditionremovallemma}
		Let $n \geq 4$ be an integer and let $1 < m \leq \frac{n}{2} $ and $r>0$ be real numbers. Let $A: B_{r}(0) \rightarrow \Lambda^{1}\mathbb{R}^{n}\otimes \mathfrak{g}$ be a Lie-algebra valued $1$-form for some compact finite dimensional Lie group $G$ such that $A \in \mathsf{W}^{1}\mathrm{L}^{m, n-2m}\left( B_{r}(0),\Lambda^{1}\mathbb{R}^{n}\otimes \mathfrak{g} \right).$ Then there exists a sequence $\left\lbrace A_{s} \right\rbrace_{s \in \mathbb{N}} \subset C^{\infty}\left( \overline{B_{r}(0)},\Lambda^{1}\mathbb{R}^{n}\otimes \mathfrak{g} \right)$ and a constant $C=C(n, m, r ) > 0$ such that 
		\begin{align}
			A_{s} &\rightarrow A &&\text{ in } L^{2m}\left( B_{r}(0),\Lambda^{1}\mathbb{R}^{n}\otimes \mathfrak{g} \right), \label{L4 convgce} \\
			dA_{s} &\rightarrow dA &&\text{ in } L^{m}\left( B_{r}(0),\Lambda^{2}\mathbb{R}^{n}\otimes \mathfrak{g} \right) \label{L2 convgce} 
		\end{align} 
		and we have the bounds 
		\begin{align}\label{uniformmorreybound}
			\limsup\limits_{s\rightarrow \infty}\left\lVert A_{s}\right\rVert_{\mathrm{L}^{2m, n-2m}} &\leq C \left\lVert A \right\rVert_{\mathsf{W}^{1}\mathrm{L}^{m, n-2m}}, \intertext{and}  	\limsup\limits_{s\rightarrow \infty}\left\lVert dA_{\nu} \right\rVert_{\mathrm{L}^{m, n-2m}} &\leq C  \left\lVert A \right\rVert_{\mathsf{W}^{1}\mathrm{L}^{m, n-2m}}.\label{uniformmorreyboundderivatives} 
		\end{align}
	\end{lemma}
	\begin{remark}
		Note that the estimates \eqref{uniformmorreybound} and \eqref{uniformmorreyboundderivatives} are not scale-homogeneous. 
	\end{remark}
	\begin{proof}
		We first use Proposition \ref{approximation in Morrey-Sobolev spaces} to extend $A \in \mathsf{W}^{1}\mathrm{L}^{m, n-2m}\left( B_{r}(0),\Lambda^{1}\mathbb{R}^{n}\otimes \mathfrak{g} \right)$ to all of $\mathbb{R}^{n}.$ Let $\tilde{A} \in \mathsf{W}^{1}\mathrm{L}^{m, n-2m}\left( \mathbb{R}^{n},\Lambda^{1}\mathbb{R}^{n}\otimes \mathfrak{g} \right)$ be the extension. Thus, $\tilde{A} = A$ in $B_{r}(0)$ and we have the bounds
		\begin{align}\label{control of W1m norm connection}
			\left\lVert \tilde{A}\right\rVert_{\mathsf{W}^{1}\mathrm{L}^{m, n-2m}\left( \mathbb{R}^{n},\Lambda^{1}\mathbb{R}^{n}\otimes \mathfrak{g} \right)} \leq C \left\lVert A \right\rVert_{\mathsf{W}^{1}\mathrm{L}^{m, n-2m}\left( B_{r}(0),\Lambda^{1}\mathbb{R}^{n}\otimes \mathfrak{g} \right)}.
		\end{align} 
		By Proposition \ref{Adams embedding}, we have $\tilde{A} \in \mathrm{L}^{2m, n-2m}\left( \mathbb{R}^{n},\Lambda^{1}\mathbb{R}^{n}\otimes \mathfrak{g} \right)$ as well and we have the estimate 
		\begin{align}
			\left\lVert \tilde{A}\right\rVert_{\mathrm{L}^{2m, n-2m}\left( \mathbb{R}^{n},\Lambda^{1}\mathbb{R}^{n}\otimes \mathfrak{g} \right)} &\leq C	\left\lVert \tilde{A}\right\rVert_{\mathsf{W}^{1}\mathrm{L}^{m, n-2m}\left( \mathbb{R}^{n},\Lambda^{1}\mathbb{R}^{n}\otimes \mathfrak{g} \right)} \notag \\&\leq C \left\lVert A \right\rVert_{\mathsf{W}^{1}\mathrm{L}^{m, n-2m}\left( B_{r}(0),\Lambda^{1}\mathbb{R}^{n}\otimes \mathfrak{g} \right)} \label{control of L2m norm connection}.
		\end{align}
		Now we let $\eta \in C_{c}^{\infty}(B_{1}(0))$ be a smooth mollifier with $ 0 \leq \eta \leq 1$ and $\int_{\mathbb{R}^{n}} \eta = 1.$ Define $\eta_{s}(x) := s^{n}\eta(s x)$ and set $A_{s} := \tilde{A} \ast \eta_{s}.$ Then the convergences in \eqref{L4 convgce} and \eqref{L2 convgce} are standard facts. So we only need to show \eqref{uniformmorreybound} and \ref{uniformmorreyboundderivatives}. 
		
		To show this, for a fixed $x \in \mathbb{R}^{n},$ set $g_{x}(y):= \tilde{A}(x-y)$ for all $y \in \mathbb{R}^{n}.$ Since $\int_{\mathbb{R}^{n}} \eta_{s} = 1$ for all $s \in \mathbb{N},$ the measures $\eta_{s}(x)dx$ are probability measures for all $s.$ Thus, applying Jensen's inequality with respect to these measures, we obtain 
		$$ \left\lvert  \int_{\mathbb{R}^{n}} g_{x}(y)\eta_{s}(y)dy \right\rvert^{2m} \leq  \int_{\mathbb{R}^{n}} \left\lvert   g_{x}(y) \right\rvert^{2m}\eta_{s}(y)dy \qquad \text{ for every } x \in \mathbb{R}^{n} \text{ and  every } s.  $$
		
		Thus, integrating with respect to $x \in B_{\rho}(x_{0})$ over any ball $B_{\rho}(x_{0}) ,$    we obtain 
		\begin{align*}
			\frac{1}{\rho^{n-2m}}\int_{B_{\rho}(x_{0})} \left\lvert A_{s}(x) \right\rvert^{2m} dx &= \frac{1}{\rho^{n-2m}}\int_{B_{\rho}(x_{0})} \left\lvert   \int_{\mathbb{R}^{n}} g_{x}(y)\eta_{s}(y)dy \right\rvert^{2m}  dx  \\
			&\leq \frac{1}{\rho^{n-2m}}\int_{B_{\rho}(x_{0})} \left( \int_{\mathbb{R}^{n}} \left\lvert   g_{x}(y) \right\rvert^{2m}\eta_{s}(y)dy \right)  dx \\
			&= \frac{1}{\rho^{n-2m}}\int_{B_{\rho}(x_{0})} \left( \int_{\mathbb{R}^{n}} \left\lvert   \tilde{A}( x- y) \right\rvert^{2m}\eta_{s}(y)dy \right)  dx \\
			&\stackrel{\text{Fubini}}{\leq } \left\lVert \tilde{A} \right\rVert_{\mathrm{L}^{2m, n-2m}\left( \mathbb{R}^{n},\Lambda^{1}\mathbb{R}^{n}\otimes \mathfrak{g} \right)}\int_{\mathbb{R}^{n}} \eta_{s}(y)dy   \\
			&= \left\lVert \tilde{A} \right\rVert^{2m}_{\mathrm{L}^{2m, n-2m}\left( \mathbb{R}^{n},\Lambda^{1}\mathbb{R}^{n}\otimes \mathfrak{g} \right)} \\
			&\stackrel{\eqref{control of L2m norm connection}}{\leq} C \left\lVert A \right\rVert^{2m}_{\mathsf{W}^{1}\mathrm{L}^{m, n-2m}\left( B_{r}(0),\Lambda^{1}\mathbb{R}^{n}\otimes \mathfrak{g} \right)}. 
		\end{align*} This implies \eqref{uniformmorreybound}. Using the standard fact $dA_{s} = d\tilde{A} \ast \eta_{s},$ we can similarly estimate $\mathrm{L}^{m, n -2m}$ norm of $dA_{s}.$ This completes the proof. 
	\end{proof}	
	Now we are ready to show the following. 	
	\begin{theorem}[Existence of Coulomb gauges under small norm of connection]\label{existence_Coulombgauges_small_connection}
		Let $n \geq 4$ be an integer and let $1 < m \leq \frac{n}{2} ,$ $r>0$ be real numbers, $x_{0} \in \mathbb{R}^{n}$ and let $B_{r}(x_{0})\subset \mathbb{R}^{n}$ be the ball of radius $r$ around $x_{0}$. Then there exist constants $\varepsilon_{C} = \varepsilon_{C}\left( G, n, m, r \right)  \in \left(0,1\right)$ and $C_{C}= C_{C}\left( G, n, m, r \right) \geq 1 $ such that for any $A \in \mathsf{W}^{1}\mathrm{L}^{m,n-2m}\left(B_{r}(x_{0}); \Lambda^{1}\mathbb{R}^{n}\otimes \mathfrak{g}\right)$ with 
		\begin{align}
			\left\lVert A\right\rVert_{\mathsf{W}^{1}\mathrm{L}^{m,n-2m}\left(B_{r}(x_{0}); \Lambda^{1}\mathbb{R}^{n}\otimes \mathfrak{g}\right)} \leq \varepsilon_{C},
		\end{align}
		there exists $\rho \in \mathsf{W}^{2}\mathrm{L}^{m,n-2m}\left(B_{r}(x_{0}) ; G \right)$ such that 
		\begin{align*}
			\left\lbrace \begin{aligned}
				&d^{\ast} A^{\rho} = 0 \quad \text{ in } B_{r}(x_{0}),  \\
				&\iota^{\ast}_{\partial B_{r}(x_{0}) } \left( \ast A^{\rho}\right)  = 0 \quad \text{ on } \partial B_{r}(x_{0}),	\end{aligned}\right.
		\end{align*}
		and we have the estimates 
		\begin{align}\label{Coulomb with small conn}
			\left\lVert \nabla A^{\rho}\right\rVert_{\mathrm{L}^{m,n-2m}} + \left\lVert  A^{\rho}\right\rVert_{\mathrm{L}^{2m,n-2m}}  \leq C_{C}\left( \left\lVert  F_{A}\right\rVert_{\mathrm{L}^{m,n-2m}} + \left\lVert  A\right\rVert_{\mathsf{W}^{1}\mathrm{L}^{m,n-2m}}\right)  
		\end{align}	
		and 
		\begin{align*}
			\left\lVert  \nabla d\rho\right\rVert_{\mathrm{L}^{m,n-2m}} + \left\lVert   d\rho\right\rVert_{\mathrm{L}^{2m,n-2m}} \leq C_{1}\left\lVert  A\right\rVert_{\mathsf{W}^{1}\mathrm{L}^{m,n-2m}},
		\end{align*}	
		where $C_{1} \geq 1 $ is a constant depending only on $C_{C},$ $r,m $ and the $L^{\infty}$ bound for $G.$	
	\end{theorem}
	\begin{remark}
		Note that the theorem is markedly different from usual Coulomb gauge existence results. Firstly, the smallness paramater ( and also the constants ) are not scale-invariant and depend on the radius $r.$ Also, we need an additional smallness condition on the norm of the connection. Moreover, the norm of the connection appears on the right hand side of the estimate \eqref{Coulomb with small conn}. So norm of the curvature alone does not control the Morrey-Sobolev norm of the connection in the Coulomb gauge. In a sense, all of these somewhat defeats the purpose of trying to obtain the existence of a Coulomb gauge. But as we shall see, this would still be sufficient and useful in many purposes. The question of the existence of a Coulomb gauge in super critical dimensions under only a scale-invariant smallness condition on the norm of the curvature is widely open till date.    
	\end{remark}
	\begin{proof}
		The translation invariance is obvious. So we prove for $B_{r}\left(0\right).$ Let $\varepsilon \in \left(0,1\right)$ be a parameter, which we are going to chose later, such that 
		\begin{align*}
			\left\lVert A\right\rVert_{\mathsf{W}^{1}\mathrm{L}^{m,n-2m}} \leq \varepsilon.
		\end{align*}
		Since $A \in \mathsf{W}^{1}\mathrm{L}^{m,n-2m},$ Lemma \ref{approximabilityconditionremovallemma}, we can find a sequence of smooth connection $\left\lbrace A_{s} \right\rbrace_{s \in \mathbb{N}} \subset C^{\infty}\left( \overline{B_{r}(0)},\Lambda^{1}\mathbb{R}^{n}\otimes \mathfrak{g} \right)$ such that \begin{align*}
			A_{s} &\rightarrow A &&\text{ in } L^{2m}\left( B_{r},\Lambda^{1}\mathbb{R}^{n}\otimes \mathfrak{g} \right), \\
			dA_{s} &\rightarrow dA &&\text{ in } L^{m}\left( B_{r},\Lambda^{2}\mathbb{R}^{n}\otimes \mathfrak{g} \right) 
		\end{align*} 
		and for all $ s\in \mathbb{N},$ we have the uniform bounds 
		\begin{align}\label{morrey bound for approx conn}
			\left\lVert A_{s}\right\rVert_{\mathrm{L}^{2m, n-2m}} , \quad \left\lVert dA_{s} \right\rVert_{\mathrm{L}^{m, n-2m}} \leq C  \left\lVert A \right\rVert_{\mathsf{W}^{1}\mathrm{L}^{m, n-2m}} . 
		\end{align} Thus, for all $s \in \mathbb{N},$ we deduce the bound  
		\begin{align}
			\left\lVert F_{A_{s}}\right\rVert_{\mathrm{L}^{2m, n-2m}} &\leq  \left\lVert dA_{s} \right\rVert_{\mathrm{L}^{m, n-2m}} + \left\lVert A_{s}\right\rVert^{2}_{\mathrm{L}^{2m, n-2m}}  \notag \\&\leq C  \left\lVert A \right\rVert_{\mathsf{W}^{1}\mathrm{L}^{m, n-2m}}  + C^{2}  \left\lVert A \right\rVert_{\mathsf{W}^{1}\mathrm{L}^{m, n-2m}}^{2}  \notag \\&\leq  C^{2}  \left\lVert A \right\rVert_{\mathsf{W}^{1}\mathrm{L}^{m, n-2m}} \leq C^{2}\varepsilon. 
		\end{align}Thus, choosing $\varepsilon < \frac{1}{C^{2}}\varepsilon_{Coulomb},$ where here $\varepsilon_{Coulomb}$ stands for the smallness parameter given by Theorem \ref{existence of Coulomb gauges in better spaces}, we deduce as a consequence of Theorem \ref{existence of Coulomb gauges in better spaces} the existence of Coulomb gauges $\rho_{s} \in \mathsf{W}^{2}\mathrm{L}^{m,n-2m}\left(B_{r} ; G \right)$ such that 
		\begin{align*}
			\left\lbrace \begin{aligned}
				&d^{\ast} A_{s}^{\rho_{s}} = 0 \quad \text{ in } B_{r},  \\
				&\iota^{\ast}_{\partial B_{r} } \left( \ast A_{s}^{\rho_{s}}\right)  = 0 \quad \text{ on } \partial B_{r}.	\end{aligned}\right.
		\end{align*}
		
		and we have the estimates 
		\begin{align}
			\left\lVert \nabla A_{s}^{\rho_{s}}\right\rVert_{\mathrm{L}^{m,n-2m}} + \left\lVert  A_{s}^{\rho_{s}}\right\rVert_{\mathrm{L}^{2m,n-2m}}  &\leq C_{Coulomb} \left\lVert F_{A_{s}}\right\rVert_{\mathrm{L}^{m,n-2m}} \notag \\&\leq  C_{Coulomb}C^{2}  \left\lVert A \right\rVert_{\mathsf{W}^{1}\mathrm{L}^{m, n-2m}}.\label{estimate for connections}
		\end{align}	
		and  \begin{align}\label{estimate for gauges}
			\left\lVert  \nabla d\rho_{s}\right\rVert_{\mathrm{L}^{m,n-2m}} + \left\lVert   d\rho_{s}\right\rVert_{\mathrm{L}^{2m,n-2m}} \leq C_{1}\left\lVert  A\right\rVert_{\mathsf{W}^{1}\mathrm{L}^{m,n-2m}} .
		\end{align}
		
		Thus we have, in particular, that $d\rho_{s}$ is uniformly bounded in $W^{1,m}.$ Since $G$ is compact, $\rho_{s}$ is uniformly bounded in $L^{\infty}.$ Thus, we have, 
		\begin{align}
			\rho_{s} \rightharpoonup \rho  \quad \text{ in } W^{2,m} \qquad \text{ and } \qquad \rho_{s} \stackrel{\ast}{\rightharpoonup} \rho  \quad \text{ in } L^{\infty}
		\end{align} for some $\rho \in W^{2,m}\cap L^{\infty}. $ By compactness of the Sobolev embedding $W^{2,m}\hookrightarrow L^{m},$ for any $m \leq p < \infty,$ we deduce 
		\begin{align*}
			\left\lVert \rho_{s} - \rho \right\rVert_{L^{p}} \leq  \left( 2C_{G}\right)^{\frac{p-m}{p}}\left\lVert \rho_{\nu} - \rho \right\rVert^{\frac{m}{p}}_{L^{m}} \rightarrow 0. 
		\end{align*}
		Thus, \begin{align}\label{strong convergence to rho}
			\rho_{s} \rightarrow \rho  \quad \text{ in } L^{p} \quad \text{ and thus also } \quad \rho_{s}(x) \rightarrow \rho(x) \text{ for a.e. } x,
		\end{align} for any $1 \leq p < \infty. $ So $\rho(x) \in G$ for a.e. $x.$ Now, \eqref{estimate for connections} implies, in particular, 
		\begin{align}\label{weak convergence to B}
			A_{s}^{\rho_{s}} \rightharpoonup B  \quad \text{ in } W^{1,m}
		\end{align} for some $B \in W^{1,m}\left( B_{r};\Lambda^{1}\mathbb{R}^{n}\otimes \mathfrak{g} \right).$ By uniqueness of weak limits, we also have 
		\begin{align*}
			A_{s}^{\rho_{s}} \rightharpoonup B  \quad \text{ in } L^{2m}.
		\end{align*} Now we claim that $$B = A^{\rho}:= \rho^{-1}d\rho + \rho^{-1}A\rho$$ for a.e. $x \in B_{r}.$ Indeed, since we have 
		$$A_{s}^{\rho_{s}}=  \rho_{s}^{-1}d\rho_{s} + \rho_{s}^{-1}A_{s}\rho_{s}$$ for every $s,$ and $A_{s} \rightarrow A$ in $L^{2m},$ in view of \eqref{strong convergence to rho}, the right hand side converges weakly to 
		$$\rho_{\nu}^{-1}d\rho_{\nu} + \rho_{\nu}^{-1}A_{\nu}\rho_{\nu} \rightharpoonup \rho^{-1}d\rho + \rho^{-1}A\rho \qquad \text{ in } L^{q}, $$
		for any $q < m.$ Thus the claim follows by uniqueness of weak limits. \smallskip 
		
		Thus we have established 
		\begin{align*}
			A_{s}^{\rho_{s}} \rightharpoonup A^{\rho}  \quad \text{ in } W^{1,m} \text{ and in }  L^{2m}.
		\end{align*}
		This implies that $A^{\rho}$ satisfies 
		\begin{align*}
			\left\lbrace \begin{aligned}
				&d^{\ast} A^{\rho} = 0 \quad \text{ in } B_{r},  \\
				&\iota^{\ast}_{\partial B_{r} } \left( \ast A^{\rho}\right)  = 0 \quad \text{ on } \partial B_{r},	\end{aligned}\right.
		\end{align*}
		Thus,  using Proposition \ref{Adams embedding bounded domain} and Proposition \ref{Gaffney ineq in Morrey spaces} and the gauge invariance of the curvature, we have,
		\begin{align*}
			\left\lVert \nabla A^{\rho}\right\rVert_{\mathrm{L}^{m,n-2m}} &+ \left\lVert  A^{\rho}\right\rVert_{\mathrm{L}^{2m,n-2m}} 
			\\&\leq C\left\lVert  A^{\rho}\right\rVert_{\mathsf{W}^{1}\mathrm{L}^{m,n-2m}} \\
			&\leq C \left\lVert  dA^{\rho}\right\rVert_{\mathrm{L}^{m,n-2m}} \\
			&\leq C \left( \left\lVert  F_{A^{\rho}}\right\rVert_{\mathrm{L}^{m,n-2m}} + \left\lVert  A^{\rho}\wedge A^{\rho}\right\rVert_{\mathrm{L}^{m,n-2m}}\right) \\
			&\leq C\left( \left\lVert  F_{A}\right\rVert_{\mathrm{L}^{m,n-2m}} + \left\lVert  A^{\rho}\right\rVert^{2}_{\mathrm{L}^{2m,n-2m}}\right).
		\end{align*}
		Since $	A_{s}^{\rho_{s}} \rightharpoonup A^{\rho}$ in $L^{2m},$ using $\varepsilon <1,$ we have 
		\begin{align*}
			\left\lVert  A^{\rho}\right\rVert^{2}_{\mathrm{L}^{2m,n-2m}} \leq \liminf\limits_{s\rightarrow \infty}\left\lVert  A_{s}^{\rho_{s}}\right\rVert^{2}_{\mathrm{L}^{2m,n-2m}} \leq C \left\lVert  A\right\rVert^{2}_{\mathsf{W}^{1}\mathrm{L}^{m,n-2m}} \leq C \left\lVert  A\right\rVert_{\mathsf{W}^{1}\mathrm{L}^{m,n-2m}}.
		\end{align*}
		The estimate for $\left\lVert  d\rho\right\rVert_{\mathsf{W}^{1}\mathrm{L}^{m,n-2m}}$ follows easily from this. This completes the proof. 		
	\end{proof}
	\subsubsection{Regularity of Coulomb bundles}
	Now we establish an epsilon regularity result for Coulomb bundles, which is perhaps the most crucial tool for all our results.  
	\begin{theorem}\label{regularity Coulomb}
		Let $P = \left( \left\lbrace U_{\alpha}\right\rbrace_{\alpha \in I}, \left\lbrace g_{\alpha\beta} \right\rbrace_{\alpha, \beta \in I}\right)$ be a $\mathsf{W}^{2}\mathrm{L}^{m,n-2m}$ principal $G$-bundle over $M^{n}$ and $A = \left\lbrace A_{\alpha}\right\rbrace_{\alpha \in I}$ be a $ \mathsf{W}^{1}\mathrm{L}^{m,n-2m}$ connection on $P$ which is Coulomb.	Then there exists a smallness parameter $\varepsilon_{\text{reg}} = \varepsilon_{\text{reg}} \left( n, m, G, M^{n} \right) \in (0,1)$  such that if 
		\begin{align*}
			\sup\limits_{\alpha \in I}  \left\lVert A \right\rVert_{\mathrm{L}^{2m, n-2m}\left(U_{\alpha};\Lambda^{1}\mathbb{R}^{n}\otimes \mathfrak{g}\right)} \leq \varepsilon_{\text{reg}},
		\end{align*} then $P$ is a $\mathsf{W}^{2}\mathrm{L}^{q,n-2m}\cap C^{0,\gamma}$-bundle for any exponent $2m/(m+1) \leq q < 2m$ and any $0 \leq \gamma < 1.$ More precisely, there exists a good refinement $\left\lbrace V_{j}\right\rbrace_{j \in J}$ of $\left\lbrace U_{\alpha}\right\rbrace_{\alpha \in I}$ with $V_{j} \subset \subset U_{\phi (j)}$ for every $j \in J,$ where $\phi: J \rightarrow I$ is the refinement map,  such that we have  $g_{\phi(i)\phi(j)} \in  \mathsf{W}^{2}\mathrm{L}^{q,n-2m}\left(V_{i}\cap V_{j}; G\right)$ for any $ m < q < 2m $ and thus also  $C^{0,\gamma}\left(\overline{V_{i}\cap V_{j}}; G\right)$ for any $0 \leq \gamma < 1,$ for all $i,j \in J,$ whenever $V_{i}\cap V_{j} \neq \emptyset$.
	\end{theorem} 
	\begin{proof}
		Let $C\left(\mathfrak{g}\right) \geq 1$ be a fixed number depending only on the dimension of the Lie algebra $\mathfrak{g}.$ We shall specify the choice of $C\left(\mathfrak{g}\right)$ later, but the important point is that this choice would not depend on any subsequent choices.  We choose a good refinement $\left\lbrace V_{j}\right\rbrace_{j \in J}$ of $\left\lbrace U_{\alpha}\right\rbrace_{\alpha \in I}$ by geodesic balls in such a way that there is an enlarged cover $\left\lbrace V^{'}_{j}\right\rbrace_{j \in J}$ which is also a refinement of $\left\lbrace U_{\alpha}\right\rbrace_{\alpha \in I}$ with the same refinement map $\phi: J \rightarrow I.$ Choosing $\varepsilon_{\text{reg}} < \frac{\varepsilon_{\Delta_{Cr}}}{4C\left(\mathfrak{g}\right)},$ we have 
		\begin{align}
			\left\lVert A_{j}\right\rVert_{\mathrm{L}^{2m,n-2m}\left( V^{'}_{j};\Lambda^{1} \mathbb{R}^{n}\otimes \mathfrak{g} \right)} &<\frac{\varepsilon_{\Delta_{Cr}}}{4C\left(\mathfrak{g}\right)}  \label{LnsmallnessCoulomb}
		\end{align} for every $ j \in J,$ where $A_{j}:= A_{\phi(j)}|_{V^{'}_{j}},$  and 
		we have $$ V_{j} \subset\subset V^{'}_{j} \subset U_{\phi(j)} \quad \text{ for every } j \in J, \qquad \bigcup\limits_{j \in J} V_{j} = \bigcup\limits_{\alpha \in I} U_{\alpha} = M^{n} .$$ 
		Now, setting $h_{ij}= g_{\phi(i)\phi(j)}$ for every $i,j \in J$ such that $V^{'}_{i}\cap V^{'}_{j} \neq \emptyset,$ we have the gluing relations,  
		\begin{equation}\label{gluing identityCoulomb}
			A_{j} = h_{ij}^{-1}dh_{ij} + h_{ij}^{-1}A_{i}h_{ij} \qquad \text{ for a.e } x \text{ in } V^{'}_{i}\cap V^{'}_{j} \text{ whenever }V^{'}_{i}\cap V^{'}_{j} \neq \emptyset.
		\end{equation}
		Rewriting \eqref{gluing identityCoulomb}, we have, 
		$$ dh_{ij} = h_{ij}A_{j} -  A_{i} h_{ij} \qquad \text{ for a.e } x \text{ in } V^{'}_{ij} \text{ whenever }V^{'}_{ij} \neq \emptyset .$$
		Also, since $A$ is in Coulomb gauge, we have  $$d^{\ast}A_{i} = 0 = d^{\ast}A_{j}  \qquad \text{ in } V^{'}_{ij} .$$
		This implies, 
		\begin{equation}
			- \Delta h_{ij} = \ast \left[ dh_{ij} \wedge \left( \ast A_{j}\right)\right] + \ast \left[  \left( \ast A_{i}\right) \wedge dh_{ij} \right]   \qquad \text{ in } V^{'}_{ij} .
		\end{equation}
		This is of the same form as \eqref{critical elliptic eqn} with $f=0.$
		Thus, by \eqref{LnsmallnessCoulomb}, we can choose $C\left(\mathfrak{g}\right)$ such that we can apply lemma \ref{ellipticCritical} with $f =0$ and deduce that  $h_{ij} \in \mathsf{W}^{2} \mathrm{L}^{q,n-2m}\left( V_{ij}; G \right),$ for \emph{any} $2m/(m+1) \leq  q < 2m.$ Now we can apply the last conclusion of Proposition \ref{Adams embedding bounded domain} for $m < q < 2m$ to prove the H\"{o}lder continuity of $h_{ij}$ in $V_{ij}.$ This completes the proof.
	\end{proof}
	\subsubsection{Uniqueness of Coulomb bundles}
	\begin{proposition}\label{gaugerelatedcoulomb}
		For $i=1,2,$ let $P^{i}$ be a $\mathsf{W}^{2}\mathrm{VL}^{m, n-2m}$ principal $G$-bundle over $M^{n}$ and $A^{i}$ be a $\mathsf{W}^{1}\mathrm{VL}^{m, n-2m}$ connection on $P^{i}$ such that $A^{i}$ is Coulomb on $P^{i}.$ If there exists a $\mathsf{W}^{2}\mathrm{L}^{m, n-2m}$ gauge transformation such that  $$\left( P^{1}, A^{1}\right)\simeq_{\sigma} \left( P^{2}, A^{2}\right),$$ Then $P_{1}$ and $P_{2}$ are $C^{0}$-equivalent. 	
	\end{proposition}
	\begin{proof}
		The proof is very similar to how we proved the continuity of Coulomb bundles, so we provide only a brief sketch. By passing to a common refinement and slightly shrinking the domains, using the regularity results, we can assume, without loss of generality, that $P^{1}$ and $P^{2}$ are both $C^{0}$ bundles. Since the connections are also gauge related, we have
		$$ d\sigma_{i} = \sigma A_{i}^{2} - A_{i}^{1}\sigma \qquad \text{ in } U_{i},$$ 
		for every $i\in I. $ Since $A^{1}, A^{2}$ are both Coulomb, we have, 
		$$ - \Delta\sigma_{i} = \ast \left[ d\sigma_{i} \wedge \left( \ast A^{2}_{i}\right)\right] + \ast \left[  \left( \ast A^{1}_{i}\right) \wedge d\sigma_{i} \right]  \qquad \text{ in } U_{i},$$ for every $i\in I. $ But once again this is exactly of the form of Equation  \eqref{critical elliptic eqn} with $f=0.$ Thus, by passing to a refinement of the cover  such that $\mathrm{L}^{2m, n-2m}$ norms of $A_{i}^{1}$ and $A_{i}^{2}$ are suitably small, using lemma \ref{ellipticCritical} we deduce the continuity of $\sigma_{i}$ in the interior. The proof is concluded by slightly shrinking the domains.    
	\end{proof}
	This immediately implies the uniqueness of the $C^{0}$ isomorphism class of Coulomb bundles for a bundle-connection pair in the vanishing Morrey case. 
	\begin{proposition}[Uniqueness of Coulomb bundles]\label{uniquenessvanishingMorreyCoulomb}
		Given a pair $\left( P, A \right),$ where $P$ is a $\mathsf{W}^{2}\mathrm{VL}^{2m,n-2m}$ principal $G$-bundle over $M^{n}$ and $A$ be is a $\mathsf{W}^{1}\mathrm{VL}^{m, n-2m}$ connection on $P,$ there exists a  $C^{0}$-bundle $P_{C}$, \emph{unique up to $C^{0}$-equivalence}, such that $P_{C} \simeq_{\sigma} P$ and $\left( \sigma^{-1}\right)^{\ast}A$ is a Coulomb connection on $P_{C}$ for some $\mathsf{W}^{2}\mathrm{VL}^{m, n-2m}$ gauge transformations $\sigma.$
	\end{proposition}
	\section{Topology of bundles}\label{topology}
	
	\subsection{Topology and the role of the connection}
	Using Theorem \ref{Vanishing Morrey case approx and topo}, we can associate a topological isomorphism class to a bundle-connection pair as follows. 
	\begin{definition}
		Let $n \geq 3$ be an integer and let $\frac{\sqrt{n}}{2} < m \leq \frac{n}{2}$ be a real number. Let $P$ be a principal $G$-bundle over a $n$-dimensional closed manifold $M^{n},$ such that the bundle transition functions are in the vanishing Morrey-Sobolev class $\mathsf{W}^{2}\mathrm{VL}^{2m,n-2m}$ and let $A$ be a connection on $P$ such that the local connection forms are in vanishing Morrey-Sobolev class $\mathsf{W}^{1}\mathrm{VL}^{m, n-2m}.$ Then we define 
		\begin{align*}
			\left[ \left( P, A\right)\right]_{\mathrm{VL}^{m, n-2m}}:= \left[ P_{C}\right]_{C^{0}},
		\end{align*}
		where $P_{C}$ is any Coulomb bundle which is gauge equivalent to $P$ via $\mathsf{W}^{2}\mathrm{VL}^{m, n-2m}$ gauge transformations.  
	\end{definition}
	Clearly, Theorem \ref{Vanishing Morrey case approx and topo} asserts such a Coulomb bundle exists and its $C^{0}$ isomorphism class is uniqueley determined. However, note that at present, we are unable to assert if the $C^{0}$ isomorphism class for $\left( P, A\right)$ and $\left( P, B\right),$ where $A$ and $B$ are two different $\mathsf{W}^{1}\mathrm{VL}^{m, n-2m}$ connections are the same or not, when $A$ and $B$ are not $\mathsf{W}^{2}\mathrm{VL}^{m, n-2m}$-gauge related. Similarly, if $P$ is itself a $C^{0}$ bundle, then we do not know if
	\begin{align*}
		\left[ \left( P, A\right)\right]_{\mathrm{VL}^{m, n-2m}} = \left[ P\right]_{C^{0}}. 
	\end{align*} 
	This somewhat strange situation is explained by the fact that a Morrey-Sobolev connection can indeed `live' on two different $C^{0}$ bundles ( see Subsection \ref{Instanton example} below ). However, as soon as the connections are more regular, our definition reduces to the usual scenario where the topology of the bundle is independent of the connection.   
	\begin{theorem}\label{topology notions coincide}
		For any $\kappa >0,$ if $P$ is a $\mathsf{W}^{2}\mathrm{L}^{m, n-2m +\kappa} $ bundle and $A$ is a $\mathsf{W}^{1}\mathrm{L}^{m, n-2m +\kappa} $ connection on $P,$ then 
		\begin{align*}
			\left[ \left( P, A\right)\right]_{\mathrm{VL}^{m, n-2m}} = \left[ P\right]_{C^{0}}. 
	\end{align*} \end{theorem}
	\begin{proof}
		In this non-borderline case, using Theorem \ref{existence of Coulomb gauges in better spaces} and the fact that $$\mathrm{L}^{m, n-2m + \kappa} \subset \mathrm{VL}^{m, n-2m + \delta} \qquad \text{for any }0 \leq \delta < \kappa,$$ we see immediately that the Coulomb gauges themselves are continuous. 
	\end{proof}
	
	\subsection{Example of Instanton on $\mathbb{S}^{4}$}\label{Instanton example}
	Let $P= \mathbb{S}^{4}\times \mathrm{SU}\left(2\right)$ be the trivial principal $\mathrm{SU}\left(2\right)$-bundle over the four sphere. Being a trivial bundle, this admits the trivial flat connection, often called the `zero' connection. Since this connection is Coulomb, the topological isomorphism class for the pair $\left(P, A\right)$ is just $\left[ P \right]_{C^{0}},$ when $A$ is the trivial connection. On the other hand, the well-known `basic instanton' on $\mathbb{S}^{4}$ can also be viewed as a Morrey-Sobolev type connection with $L^{2}$ curvature on $P.$ The fact that the curvature of this connection is $L^{2} \left( \mathbb{S}^{4}\right)$ is well-known ( see Section 6 of Freed-Uhlenbeck \cite{FreedUhlenbeck_instantonandfourmanifold} or Chapter II, Section 2 of Atiyah \cite{Atiyah_geometryofYM} ).  From the local expressions, one can also verify that the connection forms are in the Marcinkiewicz space weak-$L^{4}$ or  $L^{4}_{w}.$ Since $L^{4}_{w}$ embeds into $\mathrm{L}^{2m, n-2m}$ for any $1 < m < 2$ ( cf. Theorem 24, page 27, \cite{Sawano_Morrey_spacesVolI} for a proof ),  it is easy to check that this defines a Morrey-Sobolev connection on the trivial bundle.  However, the basic instanton is a smooth connection on the Hopf bundle ( which is $\mathbb{S}^{7}$ viewed as nontrivial principal $\mathrm{SU}\left(2\right)$-bundle over $\mathbb{S}^{4}$). So the natural isomorphism class that we should be associating to this connection form is the isomorphism class of the Hopf bundle, which is not topologically isomorphic to the trivial bundle. Note that since we are viewing the connection as $\mathsf{W}^{1}\mathrm{L}^{m, n-2m}$ connection, apriori the curvature is only $\mathrm{L}^{m, n-2m}$ and since $m <2,$ the computation of second Chern class is apriori not justified using the connection. However, the curvature being gauge-invariant is in fact $L^{2}$ and the second Chern class computed via the basic instanton is indeed that of the Hopf bundle.

	\section{Proof of the main results}\label{Proof Main results}
	\subsection{Proof of Theorem \ref{Vanishing Morrey case approx and topo}, \ref{small in Morrey}, \ref{approximation of cocycles}}
	\begin{proof}[\textbf{Proof of Theorem \ref{Vanishing Morrey case approx and topo}}]
		Let $\left( P, A \right) = \left( \left\lbrace U_{\alpha}\right\rbrace_{\alpha \in I}, \left\lbrace g_{\alpha\beta} \right\rbrace_{\alpha, \beta \in I}, \left\lbrace A_{\alpha} \right\rbrace_{\alpha \in I} \right).$ We divide the proof into several steps for the sake of clarity.\smallskip   
		
		\noindent \textbf{Step 1: Existence of local Coulomb gauges} Since we have assumed $A_{\alpha} \in \mathsf{W}^{1}\mathrm{VL}^{m, n-2m}\left( U_{\alpha}; \Lambda^{1} \mathbb{R}^{n}\otimes \mathfrak{g}\right),$ we can easily deduce that  we have 
		\begin{align*}
			A_{\alpha} \in \mathrm{VL}^{2m. n-2m}\left( U_{\alpha}; \Lambda^{1} \mathbb{R}^{n}\otimes \mathfrak{g}\right) \text{ and } F_{A_{\alpha}} \in \mathrm{VL}^{m. n-2m}\left( U_{\alpha}; \Lambda^{2} \mathbb{R}^{n}\otimes \mathfrak{g}\right).
		\end{align*}
		Thus, we can choose a refinement $\left\lbrace V_{j}\right\rbrace_{j \in J}$ by geodesic balls of sufficiently small radii with the refinement map $\phi$ such that for every $ j \in J,$ we have 
		\begin{align}
			\left\lVert F_{A_{j}}\right\rVert_{\mathrm{L}^{m,n-2m}\left( V_{j};\Lambda^{2}\mathbb{R}^{n}\otimes \mathfrak{g} \right)} &< \min\left\lbrace \frac{\varepsilon_{Coulomb}}{16}, \frac{\varepsilon_{\text{reg}}}{16C_{Coulomb}}   \right\rbrace \label{coulombsmallness}, \end{align}  where $A_{j}:= A_{\phi(j)}|_{V_{j}}.$ 
		By \eqref{coulombsmallness}, we can apply Theorem \eqref{existence of Coulomb gauges in critical Morrey-Sobolev case}. Thus, we deduce, for each $j \in J, $ there exist maps $\rho_{j} \in \mathsf{W}^{2}\mathrm{VL}^{m, n-2m}\left(  V_{j}; G \right) $ such that 
		\begin{align*}
			\left\lbrace \begin{aligned}
				d^{\ast} A_{j}^{\rho_{j}} &= 0  &&\text{  in } V_{j}, \\
				\iota^{\ast}_{\partial V_{j}} \left( \ast A_{j}^{\rho_{j}}\right)  &= 0 &&\text{  on } \partial V_{j},
			\end{aligned}\right. 
		\end{align*}
		and we have the estimates, 
		\begin{align}\label{Coulomb gauge coercivity VMS topo proof}
			\left\lVert \nabla A_{j}^{\rho_{j}}\right\rVert_{\mathrm{L}^{m,n-2m}} + \left\lVert  A_{j}^{\rho_{j}}\right\rVert_{\mathrm{L}^{2m,n-2m}}  \leq C_{Coulomb} \left\lVert F_{A_{j}}\right\rVert_{\mathrm{L}^{m,n-2m}}, 
		\end{align}	
		\begin{align}\label{Coulomb gauge first derivative VMS topo proof}
			\left\lVert  d\rho_{j}\right\rVert_{\mathrm{L}^{2m,n-2m}} \leq C_{G}\left( C_{Coulomb}\left\lVert F_{A_{j}}\right\rVert_{\mathrm{L}^{m,n-2m}}  + \left\lVert  A_{j}\right\rVert_{\mathrm{L}^{2m,n-2m}} \right)
		\end{align}
		and 
		\begin{align}
			\left\lVert  \nabla d\rho_{j}\right\rVert_{\mathrm{L}^{m,n-2m}} \leq C_{G}&\left( \left\lVert  \nabla A_{j}\right\rVert_{\mathrm{L}^{m,n-2m}}  + C_{Coulomb}\left\lVert F_{A_{j}}\right\rVert_{\mathrm{L}^{m,n-2m}}\right) 
			\notag \\ &+C_{G}\left( C_{Coulomb}\left\lVert F_{A_{j}}\right\rVert_{\mathrm{L}^{m,n-2m}}  + \left\lVert  A_{j}\right\rVert_{\mathrm{L}^{2m,n-2m}} \right)^{2},\label{Coulomb gauge second derivative VMS topo proof}
		\end{align}
		where  all the norms above are on $V_{j}$ and $C_{G} \geq 1 $ is an $L^{\infty}$ bound for $G.$	\smallskip

		\noindent \textbf{Step 2: Gluing local Coulomb gauges} Now we wish to show that the data $\left( \left\lbrace V_{j}\right\rbrace_{j \in J}, \left\lbrace h_{ij} \right\rbrace_{i, j \in J} \right):= P_{C} $ defines a $\mathsf{W}^{2}\mathrm{VL}^{q,n-2m}$ bundle for every $m< q < 2m,$ on which the pullback of the original connection $A$ defines a $\mathsf{W}^{1}\mathrm{VL}^{m, n-2m}$ connection in the Coulomb gauge, where \begin{align}\label{transferring to coulomb gauge}
			h_{ij} : = \rho_{i}^{-1}g_{\phi(i)\phi(j)}\rho_{j}:V_{ij} \rightarrow G \quad \text{ for every }i,j \in J \text{ with } V_{ij} \neq \emptyset.\end{align} It is easy to check that $\left\lbrace h_{ij} \right\rbrace_{i, j \in J}$ define $\mathsf{W}^{2}\mathrm{VL}^{q,n-2m}$ cocycles. Also, since $\left\lbrace A_{\alpha} \right\rbrace_{\alpha \in I}$ satisfies the gluing relations 
		\begin{equation*}
			A_{\beta} = g_{\alpha\beta}^{-1} d g_{\alpha\beta} + g_{\alpha\beta}^{-1} A_{\alpha} g_{\alpha\beta} \quad \text{ a.e. in } U_{\alpha} \cap U_{\beta},
		\end{equation*} 
		whenever $U_{\alpha} \cap U_{\beta} \neq \emptyset,$ we can easily check that $A_{j}^{\rho_{j}}$ satisfies the gluing relations 
		\begin{align*}
			A_{j}^{\rho_{j}} =  h_{ij}^{-1} dh_{ij} + h_{ij}^{-1} A_{i}^{\rho_{i}} h_{ij} \quad \text{ a.e. in } V_{i} \cap V_{j} \text{ for every }i,j \in J \text{ with } V_{ij} \neq \emptyset.
		\end{align*}
		Clearly,  $A_{C}:= A^{\rho}$ is a $\mathsf{W}^{1}\mathrm{VL}^{m, n-2m}$ connection on $P_{C}$ which is Coulomb.  By \eqref{coulombsmallness} and \eqref{Coulomb gauge coercivity VMS topo proof}, we have 
		\begin{align*}
			\left\lVert  A_{j}^{\rho_{j}}\right\rVert_{\mathrm{L}^{2m,n-2m}\left(V_{j}; \Lambda^{1}\mathbb{R}^{n}\otimes \mathfrak{g}\right)}  \leq \varepsilon_{\text{reg}} \qquad \text{ for all } j \in J.
		\end{align*} Thus the claimed regularity of $P_{C}$ follows by Theorem \ref{regularity Coulomb}. The uniqueness of the $C^{0}$ equivalence class of $P_{C}$ follows from Proposition \ref{uniquenessvanishingMorreyCoulomb}.\smallskip 
		
		\noindent \textbf{Step 3: Reduction to Approximation of $\left(P_{C}, A_{C}\right)$} Now we want to show that it is enough to approximate $\left(P_{C}, A_{C}\right)$ by sequences $\left\lbrace \left( P^{s}, A^{s}\right)\right\rbrace_{s \in \mathbb{N}}.$ Suppose there exist sequences of $\mathsf{W}^{2}\mathrm{VL}^{m, n-2m}$ gauges $\left\lbrace \left\lbrace \sigma^{s}_{i} \right\rbrace_{i \in J}\right\rbrace_{s \in \mathbb{N}}$, smooth cocycles $\left\lbrace \left\lbrace h^{s}_{ij} \right\rbrace_{i, j \in J}\right\rbrace_{s \in \mathbb{N}}$ and smooth local connection forms $\left\lbrace \left\lbrace A^{s}_{i} \right\rbrace_{i \in J}\right\rbrace_{s \in \mathbb{N}}$ such that 
		\begin{align*}
			P_{C} &\simeq_{\sigma_{s}}  P^{s}, \\
			h^{s}_{ij} &\rightarrow h_{ij} \quad \text{ in } \mathsf{W}^{2}\mathrm{VL}^{m, n-2m}, \\
			\left( A^{s}_{i} \right)^{\left( \sigma^{s}_{i}\right)^{-1}} &\rightarrow A_{i}^{\rho_{i}} \quad \text{ in } \mathsf{W}^{1}\mathrm{VL}^{m, n-2m}. 
		\end{align*}
		Now since $\rho_{i} \in \mathsf{W}^{2}\mathrm{VL}^{m, n-2m}\left(  V_{i}; G \right) $ for every $i \in J,$ by Theorem \ref{Approximation of G valued maps critical} (i), upto slightly shrinking the $V_{i}$s, which we do not relabel, for every $i \in I,$ there exists a sequence  $\left\lbrace \rho^{s}_{i} \right\rbrace_{s \in \mathbb{N}} \subset C^{\infty}\left( \overline{V_{i}}; G\right)$ such that 
		\begin{align*}
			\rho^{s}_{i} \rightarrow \rho_{i} \qquad \text{ in } \mathsf{W}^{2}\mathrm{VL}^{m, n-2m}\left(  V_{i}; G \right). 
		\end{align*}
		Note that by compactness of $G,$ a straight forward computation shows that this also implies 
		\begin{align*}
			\left( \rho^{s}_{i} \right)^{-1}\rightarrow \rho_{i}^{-1} \qquad \text{ in } \mathsf{W}^{2}\mathrm{VL}^{m, n-2m}\left(  V_{i}; G \right). 
		\end{align*}
		Now we define the sequence of bundle-connection pairs
		$$\left\lbrace \left( Q^{s}, B^{s}\right)\right\rbrace_{s \in \mathbb{N}} = \left\lbrace \left( \left\lbrace V_{i} \right\rbrace_{i \in J}, \left\lbrace  g^{s}_{ij} \right\rbrace_{i,j \in J}, \left\lbrace  B^{s}_{i} \right\rbrace_{i \in J} \right)\right\rbrace_{s \in \mathbb{N}},$$ where we set \begin{align*}
			g^{s}_{ij} := 	\rho^{s}_{i} h^{s}_{ij} \left(	\rho^{s}_{j}\right)^{-1} \qquad \text{ and } \qquad B^{s}_{i}:= \left( A^{s}_{i} \right)^{\left( \rho_{i}^{s}\right)^{-1}}. 
		\end{align*}
		Thus, setting $\theta_{i}^{s} := \rho_{i}\sigma^{s}_{i}\left( \rho_{i}^{s}\right)^{-1},$ we see that we have 
		\begin{align*}
			g^{s}_{ij} = \left( \theta_{i}^{s} \right)^{-1}g_{ij} \theta_{j}^{s} \qquad  \text{ and } \qquad  \left( B^{s}_{i} \right)^{\left( \theta_{i}^{s}\right)^{-1}}= 	\left( A^{s}_{i} \right)^{\left( \rho_{i}\sigma^{s}_{i}\right)^{-1}}. 
		\end{align*}
		Note that this mean we have 
		\begin{align*}
			g^{s}_{ij} - g_{ij} &= 	\rho^{s}_{i} h^{s}_{ij} \left(	\rho^{s}_{j}\right)^{-1} - g_{ij} 
			\\&=\rho^{s}_{i} h^{s}_{ij} \left(	\rho^{s}_{j}\right)^{-1} - \rho_{i} h_{ij} \rho_{j}^{-1}  \\&= \left( \rho^{s}_{i} - \rho_{i}\right)h^{s}_{ij} \left(	\rho^{s}_{j}\right)^{-1} + \rho_{i} \left( h^{s}_{ij} - h_{ij}\right) \left(	\rho^{s}_{j}\right)^{-1}  + \rho_{i}h_{ij}\left( \left(	\rho^{s}_{j}\right)^{-1} - \left(	\rho_{j}\right)^{-1}\right). 
		\end{align*}	
		Now straight forward computations, by computing the derivatives and using H\"{o}lder inequality in Morrey spaces shows that  	
		\begin{align*}
			g^{s}_{ij} \rightarrow 	g_{ij} \qquad \text{ in } \mathsf{W}^{2}\mathrm{VL}^{m, n-2m}\left(  V_{i}; G \right).
		\end{align*}
		Similarly, we have 
		\begin{align*}
			\left( B^{s}_{i} \right)^{\left( \theta_{i}^{s}\right)^{-1}} - A_{i}&= \left( A^{s}_{i} \right)^{\left( \rho_{i}^{s}\sigma^{s}_{i}\right)^{-1}} - A_{i} \\
			&=\left( A^{s}_{i} \right)^{\left( \sigma^{s}_{i}\right)^{-1}\left( \rho_{i}^{s}\right)^{-1}} - A_{i} \\
			&=\left[ \left( A^{s}_{i} \right)^{\left( \sigma^{s}_{i}\right)^{-1}\left( \rho_{i}^{s}\right)^{-1}} - \left( A^{s}_{i} \right)^{\left( \sigma^{s}_{i}\right)^{-1} \rho_{i}^{-1}}\right] + \left[ \left( A^{s}_{i} \right)^{\left( \sigma^{s}_{i}\right)^{-1}\left( \rho_{i}\right)^{-1}}- A_{i}\right].  
		\end{align*}	
		From the convergence 	$\left( A^{s}_{i} \right)^{\left( \sigma^{s}_{i}\right)^{-1}} \rightarrow A_{i}^{\rho_{i}}$ in $\mathsf{W}^{1}\mathrm{VL}^{m, n-2m},$ it is easy to see that the term inside the last square bracket converges to zero in $\mathsf{W}^{1}\mathrm{VL}^{m, n-2m}.$  For the second, let us write $D_{i}^{s}:= \left( A^{s}_{i} \right)^{\left( \sigma^{s}_{i}\right)^{-1}}$	and we compute 
		\begin{align*}
			\left( D_{i}^{s}\right)^{\left( \rho_{i}^{s}\right)^{-1}} - \left( D_{i}^{s}\right)^{\rho_{i}^{-1}} &=  \left[ \rho_{i}^{s}d\left( \rho_{i}^{s}\right)^{-1} - \rho_{i}d\rho_{i}^{-1}\right]   + \left[ \rho_{i}^{s} D_{i}^{s}\left( \rho_{i}^{s}\right)^{-1}  -\rho_{i} D_{i}^{s}\rho_{i}^{-1} \right].
		\end{align*}
		Now for the first square bracket, we have 
		\begin{align*}
			\rho_{i}^{s}d\left( \rho_{i}^{s}\right)^{-1} - \rho_{i}d\rho_{i}^{-1} &=\left( 	\rho_{i}^{s} - 	\rho_{i}\right)d\left( \rho_{i}^{s}\right)^{-1} + \rho_{i}\left( d\left( \rho_{i}^{s}\right)^{-1} - d\rho_{i}^{-1}\right).
		\end{align*}
		It is easy to conclude that this term converges to zero in $\mathsf{W}^{1}\mathrm{VL}^{m, n-2m}.$ 
		For the second 	square bracket, we write 
		\begin{align*}
			\rho_{i}^{s} D_{i}^{s}\left( \rho_{i}^{s}\right)^{-1}  -\rho_{i} D_{i}^{s}\rho_{i}^{-1} = \rho_{i}^{s} D_{i}^{s}\left( \left( \rho_{i}^{s}\right)^{-1} -  \rho_{i}^{-1} \right) + \left( \rho_{i}^{s}  -\rho_{i}\right) D_{i}^{s}\rho_{i}^{-1}. 
		\end{align*}
		Once again, easy computation shows this term converges to zero in  $\mathsf{W}^{1}\mathrm{VL}^{m, n-2m}.$ Thus we conclude see that we have 
		\begin{align*}
			\left(P, A\right) &\simeq_{\theta_{s}} \left( Q^{s}, B^{s}\right), \\
			g^{s}_{ij} &\rightarrow g_{ij} \quad \text{ in } \mathsf{W}^{2}\mathrm{VL}^{m, n-2m}, \\
			\left( B^{s}_{i} \right)^{\left( \theta^{s}_{i}\right)^{-1}} &\rightarrow A_{i} \quad \text{ in } \mathsf{W}^{1}\mathrm{VL}^{m, n-2m}. 
		\end{align*}\smallskip 
		
		\noindent \textbf{Step 4: Approximation of $\left(P_{C}, A_{C}\right)$}  
		Now we want to approximate $$\left(P_{C}, A_{C}\right) = \left( \left\lbrace V_{j}\right\rbrace_{j \in J}, \left\lbrace h_{ij} \right\rbrace_{i, j \in J}, \left\lbrace \left( A_{C}\right)_{i}\right\rbrace_{i \in J}\right).$$
		We pick an exponent $m < q < 2m.$ Since $P_{C}$ is a $\mathsf{W}^{2}\mathrm{VL}^{q,n-2m}$ bundle, by Theorem \ref{smoothingC0bundleMorrey}, there exists a refinement $\left\lbrace W_{i}\right\rbrace_{i \in K}$ of $\left\lbrace V_{j}\right\rbrace_{j \in J}$ with a refinement map $\psi: K \rightarrow J,$ a sequence of continuous gauges $\left\lbrace \left\lbrace \sigma_{i}^{s} \right\rbrace_{i \in K} \right\rbrace_{s \in \mathbb{N}}$ and a sequence of smooth bundles $\left\lbrace P^{s} \right\rbrace_{s \in \mathbb{N}},$ where 
		\begin{align*}
			P^{s} &= \left( \left\lbrace W_{i}\right\rbrace_{i \in K},\left\lbrace h^{s}_{ij}\right\rbrace_{i,j \in K}\right) \qquad \text{ for every } s \in \mathbb{N}, \\
			\sigma_{i}^{s} &\in \mathsf{W}^{2}\mathrm{VL}^{q,n-2m}\left( W_{i}; G\right) \qquad \text{ for every } i \in K  \text{ for every } s \in \mathbb{N},
		\end{align*} such that 
		\begin{align*}
			P^{s} &\simeq_{\sigma^{s}} P_{C}, \\
			h^{s}_{ij} &\rightarrow h_{\psi\left(i\right)\psi\left(j\right)} \qquad \text{ in } \mathsf{W}^{2}\mathrm{VL}^{q, n-2m}.
		\end{align*}
		Now, observe that the pullback of $A_{C}$ on $P^{s}$ is a $\mathsf{W}^{1}\mathrm{VL}^{m, n-2m}$ connection on $P^{s}$ for each $s \in \mathbb{N}$. Indeed, denoting the local representatives of the pullback by  
		\begin{align}\label{definitionintermediateconn}
			\tilde{A}^{s}_{i}: = \left( A_{C}\right)_{i}^{ \left( \sigma^{s}_{i}\right)^{-1}} =  \sigma^{s}_{i}d \left( \sigma^{s}_{i}\right)^{-1} + \sigma^{s}_{i}\left( A_{C}\right)_{\psi(i)} \left( \sigma^{s}_{i}\right)^{-1} \quad \text{ in } W_{i} \text{ for each }i \in K,
		\end{align}
		we can infer that $\tilde{A}^{s}_{i}$ is $\mathsf{W}^{1}\mathrm{VL}^{m, n-2m},$ since $\left( A_{C}\right)_{\psi(i)}$ is $\mathsf{W}^{1}\mathrm{VL}^{m, n-2m}$ and $\sigma^{s}_{i}$ is $\mathsf{W}^{2}\mathrm{VL}^{q, n-2m}$ for some $q > m.$
		Next we plan to show that for each fixed $s \in \mathbb{N},$ there is a sequence of smooth connection form $\left\lbrace B^{\mu,s}\right\rbrace_{\mu \in \mathbb{N}}= \left\lbrace \left\lbrace B^{\mu,s}_{i}\right\rbrace_{i \in K}\right\rbrace_{\mu \in \mathbb{N}}$ on $P^{s}$ such that 
		\begin{align}\label{approximant conn form fixed s}
			B^{\mu,s}_{i}  \rightarrow \tilde{A}^{s}_{i} \qquad \text{ in } \mathsf{W}^{1}\mathrm{VL}^{m, n-2m} \text{ for each } i \in K,  \text{ as } \mu \rightarrow \infty.
		\end{align} Note that approximating $\tilde{A}^{s}_{i}$ by smooth forms is easy, but the real point is to ensure that the approximating forms $\left\lbrace B^{\mu,s}_{i}\right\rbrace_{i \in K}$ satisfy the gluing relations 
		\begin{align}\label{verificationgluingconn}
			B^{\mu,s}_{j}= \left( h_{ij}^{s}\right)^{-1}dh^{s}_{ij} + \left(h_{ij}^{s}\right)^{-1}B^{\mu,s}_{i}h^{s}_{ij} \qquad \text{ in } W_{ij} \text{ whenever } W_{ij} \neq \emptyset.
		\end{align}
		To do this, the main idea is to approximate all the local connection forms and glue them together appropriately by using a partition of unity. We choose a partition of unity $\left\lbrace \zeta_{i}\right\rbrace_{i \in K}$ subordinate to the cover  $\left\lbrace W_{i}\right\rbrace_{k \in K}$ such that  we have the bounds 
		\begin{align}\label{partionofunitychoice}
			\left\lVert d\zeta_{i}\right\rVert_{L^{\infty}}, \left\lVert \nabla d\zeta_{i}\right\rVert_{L^{\infty}} \leq C_{part}\qquad \text{ for every } i \in K,  
		\end{align} where  $C_{part} \geq 1$ is a constant, depending on the cover $\left\lbrace W_{i}\right\rbrace_{k \in K}$. Now fix $s \in \mathbb{N}.$ By density, for each $i \in K,$ we can find a sequence $\left\lbrace \tilde{A}^{\mu, s}_{i} \right\rbrace_{\mu \in \mathbb{N}} \subset  C^{\infty}\left( \overline{W_{i}}; \Lambda^{1}\mathbb{R}^{n}\otimes \mathfrak{g}\right)$  such that 
		\begin{align}\label{smoothingconnapproxchoice}
			\left. \begin{aligned}
				\left\lVert 	\nabla \tilde{A}^{\mu, s}_{i}  - \nabla \tilde{A}^{s}_{i} \right\rVert_{\mathrm{L}^{m, n-2m}\left( W_{i}; \Lambda^{1}\mathbb{R}^{n}\otimes \mathfrak{g}\right)} &\rightarrow 0, \\
				\left\lVert \tilde{A}^{\mu, s}_{i}  -  \tilde{A}^{s}_{i} \right\rVert_{\mathrm{L}^{2m, n-2m}\left( W_{i}; \Lambda^{1}\mathbb{R}^{n}\otimes \mathfrak{g}\right)} &\rightarrow 0,	
			\end{aligned}\right\rbrace  \qquad \text{ as } \mu \rightarrow \infty.
		\end{align}
		Now, for any $\mu \in \mathbb{N}$ and for any $j \in K,$ we define 
		\begin{align}\label{definitionapproxconn}
			B^{\mu, s}_{j}: = \sum_{\substack{k \in K \\ W_{j}\cap W_{k} \neq \emptyset}} \zeta_{k}\left[ \left( h_{kj}^{s}\right)^{-1}dh_{kj}^{s} + \left(h_{kj}^{s}\right)^{-1}\tilde{A}^{\mu, s}_{k}h_{kj}^{s}\right].
		\end{align}
		Note that the possibility $j=k$ is not excluded. Clearly, $B^{\mu, s}_{j}$ is smooth for each $j \in K.$ By a straight forward computation using the identity $ dh_{kj}^{s} -  dh_{ki}^{s}h_{ij}^{s} = h_{ki}^{s}dh_{ij}^{s},$ obtained by differentiating the cocycle condition $h_{kj}^{s}=h_{ki}^{s}h_{ij}^{s},$ we arrive at \begin{align*}
			h_{ij}^{s}	B^{\mu, s}_{j} - 	B^{\mu, s}_{i}h_{ij}^{s} =\left(\zeta_{i} + \zeta_{j}\right)dh_{ij}^{s} + \sum_{\substack{k \in K, k\neq i,j,\\W_{i}\cap W_{j}\cap W_{k} \neq \emptyset }} \zeta_{k} dh_{ij}^{s} \qquad \text{ in } W_{i}\cap W_{j}, 
		\end{align*}
		for all $i \neq j \in K$ with $  W_{i}\cap W_{j} \neq \emptyset.$ Since $\left\lbrace \zeta_{k} \right\rbrace_{k \in K}$ is a partition of unity, this proves prove \eqref{verificationgluingconn}. \smallskip 
		
		Now it only remains to show \eqref{approximant conn form fixed s}. Now since $\tilde{A}^{s}_{i}$ is a connection on $P^{s}$ and $\left\lbrace \zeta_{k} \right\rbrace_{k \in K}$ is a partition of unity, we can write 
		\begin{align}
			\tilde{A}^{s}_{j} &=  \sum_{\substack{k \in K \\ W_{j}\cap W_{k} \neq \emptyset}} \zeta_{k} \left( \tilde{A}^{s}_{j}\big|_{W_{j}\cap W_{k}} \right) \notag \\&=	\sum_{\substack{k \in K \\ W_{j}\cap W_{k} \neq \emptyset}} \zeta_{k}\left[ \left( h_{kj}^{s}\right)^{-1}dh_{kj}^{s} + \left(h_{kj}^{s}\right)^{-1}\tilde{A}^{s}_{k}h_{kj}^{s}\right], \label{tilde A with part unity}
		\end{align}
		for any $j \in K.$ Thus, by \eqref{tilde A with part unity} and \eqref{definitionapproxconn}, we deduce 
		\begin{align*}
			B^{\mu, s}_{j} - \tilde{A}^{s}_{j} = \sum_{\substack{k \in K \\ W_{j}\cap W_{k} \neq \emptyset}} \zeta_{k}\left(h_{kj}^{s}\right)^{-1}\left( \tilde{A}^{\mu, s}_{k} - \tilde{A}^{s}_{k}\right)h_{kj}^{s} \quad \text{ for any } j \in K. 
		\end{align*}
		Now it is easy to see that this, combined with \eqref{smoothingconnapproxchoice} implies \eqref{approximant conn form fixed s}. \smallskip

		\noindent \textbf{Step 5: Concluding the proof} Note that \eqref{approximant conn form fixed s} clearly implies 
		\begin{align*}
			\left( B^{\mu,s}_{i} \right)^{\sigma^{s}_{i}} \rightarrow 	\left(  	\tilde{A}^{s}_{i}\right)^{\sigma^{s}_{i}} = \left( A_{C}\right)_{i}  \qquad \text{ in } \mathsf{W}^{1}\mathrm{VL}^{m, n-2m} \text{ for each } i \in K,  \text{ as } \mu \rightarrow \infty.
		\end{align*} 
		Now a diagonal argument concludes the proof. \end{proof}
	\begin{remark}
		From the proof, it might appear that we have not used the fact that $g_{\alpha\beta} \in \mathsf{W}^{2}\mathrm{VL}^{m, n-2m}$ at all and thus Theorem \ref{Vanishing Morrey case approx and topo} holds for $\mathsf{W}^{2}\mathrm{L}^{m, n-2m}$ bundles as well as long as the connection is $\mathsf{W}^{1}\mathrm{VL}^{m, n-2m}.$ This, however, is an illusion. The gluing identites imply 
		\begin{align*}
			dg_{\alpha\beta} = g_{\alpha\beta}A_{\beta} - A_{\alpha}g_{\alpha\beta} \qquad \text{ in } U_{\alpha}\cap U_{\beta}.
		\end{align*}
		This immediately implies $g_{\alpha\beta} \in  \mathsf{W}^{2}\mathrm{VL}^{m, n-2m}$ as soon as $A_{\alpha} \in \mathsf{W}^{1}\mathrm{VL}^{m, n-2m}$ for all $\alpha \in I.$ Thus, it is impossible for a $\mathsf{W}^{2}\mathrm{L}^{m, n-2m}$ bundle to `support' a $\mathsf{W}^{1}\mathrm{VL}^{m, n-2m}$ connection unless the bundle itself is a $\mathsf{W}^{2}\mathrm{VL}^{m, n-2m}$ bundle. 
	\end{remark}
	
	Now we are ready to prove Theorem \ref{small in Morrey}. The proof follows the same lines, except that this time we have to use Theorem \ref{existence_Coulombgauges_small_connection} to construct the Coulomb bundles. 
	\begin{proof}[\textbf{Proof of Theorem \ref{small in Morrey}}]
		Let $\left( P, A \right) = \left( \left\lbrace U_{\alpha}\right\rbrace_{\alpha \in I}, \left\lbrace g_{\alpha\beta} \right\rbrace_{\alpha, \beta \in I}, \left\lbrace A_{\alpha} \right\rbrace_{\alpha \in I} \right).$ Without loss of generality, we can assume that there exists a refined cover $\left\lbrace V_{\alpha}\right\rbrace_{\alpha \in I},$ with the refinement map being identity, such that $V_{\alpha}$ is obtained by slightly shrinking $U_{\alpha}$ for each $\alpha \in I.$ We first construct the Coulomb bundles. Using Theorem \ref{existence_Coulombgauges_small_connection}, for each $\alpha \in I,$ there exists a smallness parameter $ 0 < \varepsilon_{C, \alpha} < 1$ and constants $C_{C, \alpha}, C_{1, \alpha} \geq 1$ such that if 	\begin{align*}
			\left\lVert A_{\alpha}\right\rVert_{\mathsf{W}^{1}\mathrm{L}^{m,n-2m}\left( U_{\alpha}; \Lambda^{1}\mathbb{R}^{n}\otimes \mathfrak{g}\right)} \leq \varepsilon_{C, \alpha},
		\end{align*}
		there exists $\rho_{\alpha} \in \mathsf{W}^{2}\mathrm{L}^{m,n-2m}\left(U_{\alpha} ; G \right)$ such that 
		\begin{align*}
			\left\lbrace \begin{aligned}
				&d^{\ast} A_{\alpha}^{\rho_{\alpha}} = 0 \quad \text{ in } U_{\alpha},  \\
				&\iota^{\ast}_{\partial U_{\alpha} } \left( \ast  A_{\alpha}^{\rho_{\alpha}}\right)  = 0 \quad \text{ on } \partial U_{\alpha},	\end{aligned}\right.
		\end{align*}
		and we have the estimates 
		\begin{align*}
			\left\lVert \nabla  A_{\alpha}^{\rho_{\alpha}}\right\rVert_{\mathrm{L}^{m,n-2m}} + \left\lVert   A_{\alpha}^{\rho_{\alpha}}\right\rVert_{\mathrm{L}^{2m,n-2m}}  \leq C_{C, \alpha}\left( \left\lVert  F_{ A_{\alpha}}\right\rVert_{\mathrm{L}^{m,n-2m}} + \left\lVert   A_{\alpha}\right\rVert_{\mathsf{W}^{1}\mathrm{L}^{m,n-2m}}\right)  
		\end{align*}	
		and 
		\begin{align*}
			\left\lVert  \nabla d\rho_{\alpha}\right\rVert_{\mathrm{L}^{m,n-2m}} + \left\lVert   d\rho_{\alpha}\right\rVert_{\mathrm{L}^{2m,n-2m}} \leq C_{1, \alpha}\left\lVert  A_{\alpha}\right\rVert_{\mathsf{W}^{1}\mathrm{L}^{m,n-2m}}.
		\end{align*}
		Also, by Theorem \ref{Approximation of G valued maps critical} (ii), for each $\alpha \in I,$ there exists a smallness parameter $\varepsilon_{\text{approx}, \alpha}.$  Thus, we can choose  $0 < \delta < 1 $ small enough such that 
		\begin{align*}
			\sup\limits_{\alpha \in I} \left\lVert A_{\alpha} \right\rVert_{\mathsf{W}^{1}\mathrm{L}^{m, n-2m}} \leq \delta, 
		\end{align*}
		implies that for each $\alpha \in I,$ we have 
		\begin{align}\label{definition of delta existence}
			\left\lVert A_{\alpha} \right\rVert_{\mathsf{W}^{1}\mathrm{L}^{m, n-2m}\left( U_{\alpha}; \Lambda^{1}\mathbb{R}^{n}\otimes \mathfrak{g}\right)} \leq \varepsilon_{C, \alpha}, \\
			\left\lVert A_{\alpha}^{\rho_{\alpha}} \right\rVert_{\mathrm{L}^{2m, n-2m}\left( U_{\alpha}; \Lambda^{1}\mathbb{R}^{n}\otimes \mathfrak{g}\right)} \leq \varepsilon_{\text{reg}}, \label{definition of delta regularity}\\
			\left\lVert   d\rho_{\alpha}\right\rVert_{\mathrm{L}^{2m,n-2m}} \leq \varepsilon_{\text{approx}, \alpha}\label{definition of delta approximation}.
		\end{align}
		Now it is easy to check that 
		\begin{align*}
			h_{\alpha\beta} : = \rho_{\alpha}g_{\alpha\beta}\rho_{\beta}^{-1} \quad \text{ in } U_{\alpha}\cap U_{\beta},
		\end{align*}whenever $U_{\alpha}\cap U_{\beta} \neq \emptyset,$ defines a cocycle and thus $\left( \left\lbrace U_{\alpha}\right\rbrace_{\alpha \in I}, \left\lbrace h_{\alpha\beta} \right\rbrace_{\alpha, \beta \in I} \right)$ defines a $\mathsf{W}^{2}\mathrm{L}^{m,n-2m}$ bundle and $\left\lbrace A_{\alpha}^{\rho_{\alpha}}\right\rbrace_{\alpha \in I}$ is a $\mathsf{W}^{1}\mathrm{L}^{m,n-2m}$ connection on this bundle which is Coulomb. By \eqref{definition of delta regularity} and Theorem \ref{regularity Coulomb}, we see that 
		$P_{C}:=  \left( \left\lbrace V_{\alpha}\right\rbrace_{\alpha \in I}, \left\lbrace h_{\alpha\beta} \right\rbrace_{\alpha, \beta \in I} \right)$ is a $\mathsf{W}^{2}\mathrm{L}^{q,n-2m}\cap C^{0,\gamma}$-bundle for any exponent $2m/(m+1) \leq q < 2m$ and any $0 \leq \gamma < 1,$ and $A_{C}:= \left\lbrace A_{\alpha}^{\rho_{\alpha}}\right\rbrace_{\alpha \in I}$ is a $\mathsf{W}^{1}\mathrm{L}^{m,n-2m}$ connection on $P_{C}$ which is Coulomb. The rest proceeds exactly as in the proof of Theorem \ref{Vanishing Morrey case approx and topo}. Due to \eqref{definition of delta approximation}, we can approximate the Coulomb gauges $\rho_{\alpha}$ by smooth $G$-valued maps in the Sobolev norms using Theorem \ref{Approximation of G valued maps critical} (ii). 
	\end{proof}
	\begin{proof}[\textbf{Proof of Theorem \ref{approximation of cocycles}}]
		This is now an immediate corollary of Theorem \ref{Vanishing Morrey case approx and topo} and Theorem \ref{small in Morrey}, using Theorem \ref{existence of connection}. 
	\end{proof}

	\subsection{Proof of Theorem \ref{equivanishing Morrey}, \ref{flatness in vanishing}, \ref{small distance to flat}, \ref{cocycle factorization MorreySobolev}}
	\begin{proof}[\textbf{Proof of Theorem \ref{equivanishing Morrey}}]
		Since $\Theta \left(r\right) \rightarrow 0$ as $r \rightarrow 0,$ we can choose $r >0$ small enough such that 
		\begin{align*}
			\sup\limits_{\alpha \in I} \sup\limits_{\substack{0 < \rho < r, \\ B_{\rho}(x) \subset \subset U_{\alpha}, }} 
			\frac{1}{\rho^{n-2m}} \int_{B_{\rho}(x)} \lvert F_{A^{s}_{\alpha}} \rvert^{m} \leq 	\Theta\left( r \right) < \min\left\lbrace \frac{\varepsilon_{Coulomb}}{16}, \frac{\varepsilon_{\text{reg}}}{16C_{Coulomb}}  \right\rbrace. 
		\end{align*}
		Thus, by covering with geodesic balls of radius less than $r,$ we can find a cover $\left\lbrace V^{\infty}_{i} \right\rbrace_{i \in J}$ of $M^{n}$, which is refinement for the cover $\left\lbrace U_{\alpha}\right\rbrace_{\alpha \in I}$ and we have, for all $s \in \mathbb{N},$ 
		\begin{align*}
			\left\lVert F_{A^{s}_{i}} \right\rVert_{\mathrm{L}^{m, n-2m}\left( V^{\infty}_{i}; \Lambda^{2}\mathbb{R}^{n}\otimes\mathfrak{g}\right)} < \min\left\lbrace \frac{\varepsilon_{Coulomb}}{16}, \frac{\varepsilon_{\text{reg}}}{16C_{Coulomb}}  \right\rbrace \quad \text{ for all } i \in J. 
		\end{align*}
		Thus, as before, for every $s \in \mathbb{N},$ we can construct Coulomb bundles $P^{s}_{C}$, which are all trivialized over the common fixed cover $\left\lbrace V^{\infty}_{i} \right\rbrace_{i \in I}.$  Denoting the transition functions of $P^{s}_{C}$ by $g^{s}_{ij},$ we have the Coulomb condition and the gluing relations for the Coulomb bundles, 
		\begin{align}
			&dg^{s}_{ij} = g^{s}_{ij}\left( A^{s}_{C}\right)_{j}  - \left( A^{s}_{C}\right)_{i} g^{s}_{ij} &&\text{ in } V^{\infty}_{i}\cap V^{\infty}_{j} \label{gluing relation nu} \\
			&d^{\ast}\left( A^{s}_{C}\right)_{i} = 0  \qquad &&\text{ in } V^{\infty}_{i} \label{Coulomb nu}
		\end{align} 
		for every $ i,j \in J$ with $V^{\infty}_{i}\cap V^{\infty}_{j} \neq \emptyset$ and for every $i \in J$ respectively. Also, we have the estimate 
		\begin{align}\label{common coulomb estimate}
			\left\lVert  \nabla \left( A^{s}_{C}\right)_{i}\right\rVert_{\mathrm{L}^{m,n-2m}} + \left\lVert  \left( A^{s}_{C}\right)_{i} \right\rVert_{\mathrm{L}^{2m,n-2m}}  \leq C_{Coulomb} \left\lVert F_{A^{s}_{i}}\right\rVert_{\mathrm{L}^{m,n-2m}}
		\end{align} 
		every $i \in J.$ Combining \eqref{gluing relation nu} and \eqref{common coulomb estimate} and recalling that $G$ is compact, we deduce, 
		\begin{align*}
			\left\lVert dg^{s}_{ij} \right\rVert_{\mathrm{L}^{2m, n-2m}\left( V^{\infty}_{i}\cap V^{\infty}_{j}; \Lambda^{1}\mathbb{R}^{n}\otimes\mathfrak{g} \right)}  \leq C \left( \left\lVert F_{A^{s}_{i}}\right\rVert_{\mathrm{L}^{m, n-2m}}  + \left\lVert F_{A^{s}_{j}}\right\rVert_{\mathrm{L}^{m, n-2m}} \right)
		\end{align*}
		Combining this with \eqref{common coulomb estimate} and differentiating \eqref{gluing relation nu}, we also deduce 
		\begin{align*}
			\left\lVert \nabla dg^{s}_{ij} \right\rVert_{\mathrm{L}^{m, n-2m}} \leq C \left( \left\lVert F_{A^{s}_{i}}\right\rVert_{\mathrm{L}^{m, n-2m}}  + \left\lVert F_{A^{s}_{j}}\right\rVert_{\mathrm{L}^{m, n-2m}} \right)^{2}. 
		\end{align*}
		This implies that $\left\lbrace dg^{s}_{ij} \right\rbrace_{s \in \mathbb{N}}$ is uniformly bounded in $\mathsf{W}^{1}\mathrm{L}^{m, n-2m}$ and $\mathrm{L}^{2m, n-2m}.$ Since $G$ is compact, $\left\lbrace g^{s}_{ij} \right\rbrace_{s \in \mathbb{N}}$ is uniformly bounded in $L^{\infty}$ and consequently also in $\mathsf{W}^{2}\mathrm{L}^{m, n-2m}$ and $\mathsf{W}^{1}\mathrm{L}^{2m, n-2m}.$ Similarly, \eqref{common coulomb estimate} implies $\left\lbrace \left( A^{s}_{C}\right)_{i}\right\rbrace_{s \in \mathbb{N}}$ is uniformly bounded in $\mathsf{W}^{1}\mathrm{L}^{m, n-2m}$ and $\mathrm{L}^{2m, n-2m}$. Thus, using Proposition \eqref{Morrey uniformly bounded}, up to the extraction of a subsequence which we do not relabel, we have 
		\begin{align}
			g^{s}_{ij} &\stackrel{\ast}{\rightharpoonup} g^{\infty}_{ij} &&\text{ weakly $\ast$ in } L^{\infty}, \label{weakstargLinf}\\
			g^{s}_{ij} &\rightharpoonup g^{\infty}_{ij} &&\text{ weakly in } W^{2,m} \text{ and } W^{1,2m}, \label{weakgijW1n}\\
			\left( A^{s}_{C}\right)_{i} &\rightharpoonup A_{i}^{\infty} &&\text{ weakly in } W^{1,m}\text{ and } L^{2m}, \label{weakconnectionW1n2}
		\end{align}
		for some $g_{ij}^{\infty}$ and $A_{i}^{\infty}$ satisfying 
		\begin{align*}
			\left\lVert dg^{\infty}_{ij} \right\rVert_{\mathrm{L}^{2m, n-2m}} &\leq 	\liminf\limits_{s \rightarrow \infty}\left\lVert dg^{s}_{ij}\right\rVert_{\mathrm{L}^{2m, n-2m}} , \\
			\left\lVert \nabla dg^{\infty}_{ij}\right\rVert_{\mathrm{L}^{m, n-2m}} &\leq 	\liminf\limits_{s \rightarrow \infty}\left\lVert \nabla dg^{s}_{ij}\right\rVert_{\mathrm{L}^{m, n-2m}} , \\
			\left\lVert A_{i}^{\infty}\right\rVert_{\mathrm{L}^{2m, n-2m}} &\leq \liminf\limits_{s \rightarrow \infty}	\left\lVert \left( A^{s}_{C}\right)_{i} \right\rVert_{\mathrm{L}^{2m, n-2m}}, \\
			\left\lVert \nabla A_{i}^{\infty}\right\rVert_{\mathrm{L}^{m, n-2m}} &\leq 	\liminf\limits_{s \rightarrow \infty}\left\lVert \nabla \left( A^{s}_{C}\right)_{i} \right\rVert_{\mathrm{L}^{m, n-2m}}, 
		\end{align*} for every $i,j \in J.$   By compactness of the Sobolev embedding, up to the extraction of a further subsequence which we do not relabel, \eqref{weakgijW1n} and \eqref{weakstargLinf} implies 
		\begin{align}\label{strong convergence of gauges}
			g^{s}_{ij} \rightarrow g^{\infty}_{ij}\quad \text{ strongly in } L^{p} \text{ for every } p < \infty 
		\end{align} and since the maps $g^{s}_{ij}$ takes values in $G$ a.e. and satisfy the cocycle conditions, passing to the limit we deduce that the maps $g^{\infty}_{ij}$ also takes values in $G$ a.e. and satisfy the cocycle conditions as well. Thus they define a $\mathsf{W}^{2}\mathrm{VL}^{m, n-2m}$ bundle $P^{\infty}$. Using compactness of the Sobolev embedding again, up to the extraction of a further subsequence which we do not relabel, \eqref{weakconnectionW1n2} implies 
		\begin{align}\label{strong convergence of connections}
			\left( A^{s}_{C}\right)_{i} \rightarrow A_{i}^{\infty}\quad \text{ strongly in } L^{q} \text{ for every } 1 \leq q < 2m 
		\end{align} for every $i$ and thus, we have 
		\begin{align*}
			g^{s}_{ij}\left( A^{s}_{C}\right)_{j}  - \left( A^{s}_{C}\right)_{i} g^{s}_{ij} \rightarrow  g^{\infty}_{ij}A_{i}^{\infty}  - A_{j}^{\infty} g^{\infty}_{ij} \quad \text{ in } L^{\theta}
		\end{align*}
		for every $1 \leq \theta < 2m.$ Combining with \eqref{gluing relation nu} and \eqref{weakgijW1n}, this implies that the gluing relations 
		\begin{align}\label{gluing infty}
			dg^{\infty}_{ij} = g^{\infty}_{ij}A^{\infty}_{j}  - A^{\infty}_{i} g^{\infty}_{ij} \qquad \text{ in } V^{\infty}_{i}\cap V^{\infty}_{j}
		\end{align} 
		holds in the sense of distributions and pointwise a.e. for every $i,j$ with $V^{\infty}_{i}\cap V^{\infty}_{j} \neq \emptyset.$ Thus the local representatives $\left\lbrace A^{\infty}_{i}\right\rbrace_{i \in J}$ patch together to yield a global connection form $A^{\infty}$ on $P^{\infty}. $ Note that by \eqref{strong convergence of connections}, we also have 
		$$\left( A^{s}_{C}\right)_{i}\wedge \left( A^{s}_{C}\right)_{i} \rightarrow A_{i}^{\infty}\wedge A_{i}^{\infty}\quad \text{ strongly in } L^{\frac{p}{2}} \text{ for every } 2 \leq p < 2m  $$ for every $i \in J.$ Since \eqref{weakconnectionW1n2} implies that $\left( A^{s}_{C}\right)_{i}\wedge \left( A^{s}_{C}\right)_{i}$ is uniformly bounded in $L^{m},$ by uniqueness of weak limits, we deduce 
		$$ \left( A^{s}_{C}\right)_{i}\wedge \left( A^{s}_{C}\right)_{i} \rightharpoonup A_{i}^{\infty}\wedge A_{i}^{\infty}\quad \text{ weakly in } L^{m}.$$ Combining this with \eqref{weakconnectionW1n2}, we obtain 
		$$F_{A^{s}_{i}} \rightharpoonup F_{A^{\infty}_{i}} \qquad \text{ weakly in } L^{m}. $$ This implies
		\begin{align*}
			\left\lVert F_{A^{\infty}_{i}}\right\rVert_{\mathrm{L}^{m, n-2m}} \leq  \liminf\limits_{s \rightarrow \infty }\left\lVert F_{A^{s}_{i}}\right\rVert_{\mathrm{L}^{m, n-2m}}. 
		\end{align*}By \eqref{Coulomb nu} and \eqref{weakconnectionW1n2}, we deduce that $A^{\infty}$ is Coulomb and thus, up to shrinking the domains which we do not rename, by Theorem \ref{regularity Coulomb}, we can assume that  $g^{\infty}_{ij}$ are $\mathsf{W}^{2}\mathrm{L}^{q, n-2m}$ in $V^{\infty}_{i}\cap V^{\infty}_{j}$ for every $m< q < 2m$ and every $i,j$ with $V^{\infty}_{i}\cap V^{\infty}_{j} \neq \emptyset.$ Now from \eqref{gluing infty} and \eqref{gluing relation nu}, we deduce that the equation  
		\begin{multline}\label{d of diff}
			d\left( g^{s}_{ij} - g^{\infty}_{ij}\right) = \left( g^{s}_{ij} - g^{\infty}_{ij}\right)\left( A^{s}_{C}\right)_{j} - \left( A^{s}_{C}\right)_{i}\left( g^{s}_{ij} - g^{\infty}_{ij}\right) 
			\\ + g^{\infty}_{ij}\left[ \left( A^{s}_{C}\right)_{j} - A^{\infty}_{j}\right] -  \left[ \left( A^{s}_{C}\right)_{i} - A^{\infty}_{i}\right]g^{\infty}_{ij}
		\end{multline}
		holds in $V^{\infty}_{i}\cap V^{\infty}_{j}$ whenever the intersection is non-empty. 
		Now note that for any $p > \frac{n}{m-1} \geq \frac{2m}{m-1},$ we have 
		\begin{align*}
			\frac{1}{\rho^{n-2m}} \int_{B_{\rho}(x)}\left\lvert  g^{s}_{ij} - g^{\infty}_{ij} \right\rvert^{\frac{2m}{m-1}} &\leq \rho^{n\left(1 - \frac{2m}{\left(m-1\right)p}\right) - \left(n-2m\right)} \left( \int_{B_{\rho}(x)}\left\lvert  g^{s}_{ij} - g^{\infty}_{ij} \right\rvert^{p} \right)^{\frac{2m}{\left(m-1\right)p}} \\
			&= \left( \rho^{\left(m -1 \right)p -n} \int_{B_{\rho}(x)}\left\lvert  g^{s}_{ij} - g^{\infty}_{ij} \right\rvert^{p} \right)^{\frac{2m}{\left(m-1\right)p}}. 
		\end{align*}
		As $\left(m-1\right)p -n >0$ by our choice of $p$ and in view of \eqref{strong convergence of gauges}, we see that 
		\begin{align}\label{strong convergence of gauges in Morrey norms}
			\left\lVert 	g^{s}_{ij}- g^{\infty}_{ij}\right\rVert_{\mathrm{L}^{\frac{2m}{m-1},n-2m}} \rightarrow 0. 
		\end{align}
		Returning back to \eqref{d of diff}, we estimate 
		\begin{align}
			\left\lVert 	dg^{s}_{ij}- dg^{\infty}_{ij}\right\rVert_{\mathrm{L}^{2,n-2m}} &\leq 
			\left\lVert g^{s}_{ij}- g^{\infty}_{ij}\right\rVert_{\mathrm{L}^{\frac{2m}{m-1},n-2m}} 	\left\lVert 	\left( A^{s}_{C}\right)_{j}\right\rVert_{\mathrm{L}^{2m,n-2m}} \notag \\ &\quad + 	\left\lVert 	g^{s}_{ij}- g^{\infty}_{ij}\right\rVert_{\mathrm{L}^{\frac{2m}{m-1},n-2m}}\left\lVert \left( A^{s}_{C}\right)_{i}\right\rVert_{\mathrm{L}^{2m,n-2m}}  \notag \\ &\qquad + \left\lVert g^{\infty}_{ij}\right\rVert_{L^{\infty}}	\left\lVert \left( A^{s}_{C}\right)_{j} - A^{\infty}_{j}\right\rVert_{\mathrm{L}^{2,n-2m}} \notag \\&\quad\qquad + \left\lVert g^{\infty}_{ij}\right\rVert_{L^{\infty}}\left\lVert \left( A^{s}_{C}\right)_{i} - A^{\infty}_{i}\right\rVert_{\mathrm{L}^{2,n-2m}}.
			\label{estimate for d of diff in 2morrey}
		\end{align}
		The first two terms on the right converge to zero by \eqref{strong convergence of gauges in Morrey norms}, as the $\mathrm{L}^{2m, n-2m}$ norms of the connections are uniformly bounded by \eqref{common coulomb estimate} and the uniform bound on the Morrey norms of the curvatures. For the last two terms, it is enough to show 
		\begin{align*}
			\left\lVert \left( A^{s}_{C}\right)_{i} - A^{\infty}_{i}\right\rVert_{\mathrm{L}^{2,n-2m}} \rightarrow 0 \qquad \text{ for all } i \in J. 
		\end{align*}
		We establish something stronger. Note that for any $1 < q < 2m,$ we have 
		\begin{align*}
			&\frac{1}{\rho^{n-2m}} \int_{B_{\rho}(x)}\left\lvert   \left( A^{s}_{C}\right)_{i} - A^{\infty}_{i}\right\rvert^{q} \\&= 	\frac{1}{\rho^{n-2m}} \int_{B_{\rho}(x)}\left\lvert   \left( A^{s}_{C}\right)_{i} - A^{\infty}_{i}\right\rvert^{\frac{2m-q}{2m-1}}\left\lvert \left( A^{s}_{C}\right)_{i} - A^{\infty}_{i}\right\rvert^{\frac{2m\left(q-1\right)}{2m-1}} \\
			&\leq \frac{1}{\rho^{n-2m}} \left( \int_{B_{\rho}(x)}\left\lvert   \left( A^{s}_{C}\right)_{i} - A^{\infty}_{i}\right\rvert\right)^{\frac{2m-q}{2m-1}} \left( \int_{B_{\rho}(x)} \left\lvert \left( A^{s}_{C}\right)_{i} - A^{\infty}_{i}\right\rvert^{2m}\right)^{\frac{q-1}{2m-1}} \\
			&= \left( \frac{1}{\rho^{n-2m}}\int_{B_{\rho}(x)}\left\lvert   \left( A^{s}_{C}\right)_{i} - A^{\infty}_{i}\right\rvert\right)^{\frac{2m-q}{2m-1}} \left\lVert \left( A^{s}_{C}\right)_{i} - A^{\infty}_{i}\right\rVert_{\mathrm{L}^{2m, n-2m}}^{\frac{2m\left(q-1\right)}{2m-1}}
			\\&\leq \left( \int_{B_{\rho}(x)}\left\lvert   \left( A^{s}_{C}\right)_{i} - A^{\infty}_{i}\right\rvert^{\frac{n}{2m}}\right)^{\frac{2m\left(2m-q\right)}{n\left( 2m-1\right)}}  \left\lVert \left( A^{s}_{C}\right)_{i} - A^{\infty}_{i}\right\rVert_{\mathrm{L}^{2m, n-2m}}^{\frac{2m\left(q-1\right)}{2m-1}}  \\
			&\leq \left\lVert \left( A^{s}_{C}\right)_{i} - A^{\infty}_{i}\right\rVert_{L^{\frac{n}{2m}}}^{\frac{2m-q}{2m-1}}\left\lVert \left( A^{s}_{C}\right)_{i} - A^{\infty}_{i}\right\rVert_{\mathrm{L}^{2m, n-2m}}^{\frac{2m}{2m-1}}.
		\end{align*}
		Since $m > \sqrt{n}/2$ implies $n/2m < 2m,$ in view of \eqref{strong convergence of connections}, this implies  
		\begin{align}\label{strong convergence of connections in Morrey norms}
			\left\lVert \left( A^{s}_{C}\right)_{i} - A^{\infty}_{i}\right\rVert_{\mathrm{L}^{q,n-2m}} \rightarrow 0 \qquad \text{ for all } i \in J, 
		\end{align}
		for any $1 < q < 2m.$ Hence \eqref{estimate for d of diff in 2morrey} implies 
		\begin{align}\label{convergence of derivative of trabsi in 2Morrey}
			\left\lVert 	dg^{s}_{ij}- dg^{\infty}_{ij}\right\rVert_{\mathrm{L}^{2,n-2m}} \rightarrow 0. 
		\end{align}
		Now we return to \eqref{d of diff} again and noting $A^{s}_{C}$ and $A^{\infty}$ are both Coulomb, we deduce the equation
		\begin{multline}\label{laplacian of the diff}
			-\Delta u^{s}_{ij} = \ast \left[ d u^{s}_{ij} \wedge  \ast \left( A^{s}_{C}\right)_{j}\right] + \ast \left[  \ast \left( A^{s}_{C}\right)_{i} \wedge du^{s}_{ij}\right] \\
			+\ast \left[ d g^{\infty}_{ij} \wedge  \ast \left[ \left( A^{s}_{C}\right)_{j} - A^{\infty}_{j}\right]\right]  + \ast \left[  \ast \left[ \left( A^{s}_{C}\right)_{i} - A^{\infty}_{i}\right] \wedge dg^{\infty}_{ij}\right] 
		\end{multline} 
		in $V^{\infty}_{i}\cap V^{\infty}_{j}$ whenever the intersection is non-empty, where $u^{s}_{ij}=g^{s}_{ij} - g^{\infty}_{ij}.$
		Now choose exponents $m < p < 2m$ and $ 1 < q < 2m$ such that $\frac{1}{2m} < \frac{1}{q} + \frac{2m-p}{2mp} < \frac{1}{m}.$ Note that since $p > m$ implies $\frac{2mp}{2m-p}> 2m,$ such a choice of $q$ is possible. Also, since $g^{\infty}_{ij}$ is $\mathsf{W}^{2}\mathrm{L}^{p, n-2m}$, using Lemma \ref{ellipticCritical}, we deduce the estimate 
		\begin{align}
			\left\lVert u^{s}_{ij}\right\rVert_{\mathsf{W}^{2}\mathrm{L}^{\theta, n-2m}} &\leq 
			C \left\lVert u^{s}_{ij}\right\rVert_{\mathsf{W}^{1}\mathrm{L}^{2, n-2m}} \notag \\ &\quad + C \left\lVert dg^{\infty}_{ij}\right\rVert_{\mathrm{L}^{\frac{2mp}{2m-p}, n-2m}} \left\lVert \left( A^{s}_{C}\right)_{i} - A^{\infty}_{i} \right\rVert_{\mathrm{L}^{q, n-2m}} \notag \\ &\qquad + C \left\lVert dg^{\infty}_{ij}\right\rVert_{\mathrm{L}^{\frac{2mp}{2m-p}, n-2m}}\left\lVert \left( A^{s}_{C}\right)_{j} - A^{\infty}_{j} \right\rVert_{\mathrm{L}^{q, n-2m}}, \label{estimate for coulomb transition diff}
		\end{align}
		where the exponent $m < \theta < 2m $ is defined by $\frac{1}{2m} < \frac{1}{\theta}:= \frac{1}{q} + \frac{2m-p}{2mp} < \frac{1}{m}.$
		Now \eqref{strong convergence of connections in Morrey norms} implies that the last two terms on the right hand side of the estimate above converges to zero as $s \rightarrow \infty.$ The first term also converges to zero by  \eqref{convergence of derivative of trabsi in 2Morrey} and \eqref{strong convergence of gauges in Morrey norms}. Thus, using Proposition \ref{Adams embedding bounded domain}, we obtain 
		\begin{align*}
			\left\lVert g^{s}_{ij} - g^{\infty}_{ij}\right\rVert_{C^{0}} \leq 	C\left\lVert g^{s}_{ij} - g^{\infty}_{ij} \right\rVert_{\mathsf{W}^{2}\mathrm{L}^{\theta, n-2m}} \rightarrow 0 \quad \text{ as } s \rightarrow \infty
		\end{align*} for each $i,j \in J.$  By Theorem \ref{subcritical cocycle factorization}, we deduce the existence of
		a smaller cover $\left\lbrace U^{\infty}_{i}\right\rbrace_{i \in K}$ and gauge changes $\sigma_{i} \in \mathsf{W}^{2}\mathrm{L}^{\theta, n-2m}\left( U^{\infty}_{i}; G \right)$ satisfying $$ g^{\infty}_{ij}= \sigma_{i}^{-1} g^{s_{0}}_{ij} \sigma_{j} \qquad \text{ in } U^{\infty}_{i}\cap U^{\infty}_{j},$$
		whenever the intersection is non-empty, for some integer $s_{0}$ large enough. The result follows.  
	\end{proof}
	\begin{proof}[\textbf{Proof of Theorem \ref{flatness in vanishing}}]
		If the result is false, then there exists a sequence of bundle-connection pairs $\left\lbrace \left( P^{s}, A^{s} \right) \right\rbrace_{s \in \mathbb{N}}$  such that for each $s \in \mathbb{N},$ we have $ P^{s} $ is a  $\mathsf{W}^{2}\mathrm{VL}^{m, n-2m}$ bundle, trivialized over $\left\lbrace U_{\alpha}\right\rbrace_{\alpha \in I}$ and $ A^{s}$ is a $\mathsf{W}^{1}\mathrm{VL}^{m, n-2m}$ connection on $P^{s}$ such that the corresponding Coulomb bundles are not gauge equivalent to the flat bundle and 
		\begin{align*}
			\sup\limits_{\alpha \in I} \left\lVert F_{A^{s}_{\alpha}}\right\rVert_{\mathrm{L}^{m, n-2m}\left( U_{\alpha}; \Lambda^{2}\mathbb{R}^{n}\otimes \mathfrak{g} \right)} &\rightarrow 0.   
		\end{align*}
		Since this implies that the curvatures are equibounded and equivanishing in $\mathrm{L}^{m, n-2m},$ we can proceed as in the proof of Theorem \ref{equivanishing Morrey} and conclude that for the transition functions for the limitiong Coulomb bundle $P^{\infty},$ we have 
		\begin{align*}
			dg^{\infty}_{ij} = 0 \qquad \text{ for all } i, j.   
		\end{align*}
		Thus, $\left\lbrace g^{\infty}_{ij} \right\rbrace_{i, j}$ is a cocycle of constant maps and thus $P^{\infty}$ is a flat bundle. Since the Coulomb bundles associated to $\left( P^{s}, A^{s} \right)$ is gauge equivalent to $P^{\infty}$ for large enough $s,$ this contradiction proves the result.   
	\end{proof}
	\begin{proof}[\textbf{Proof of Theorem \ref{small distance to flat}}]
		This proof is similar to Theorem \ref{flatness in vanishing}. If the result is false, then there exists a sequence of bundle-connection pairs $\left\lbrace \left( P^{s}, A^{s} \right) \right\rbrace_{s \in \mathbb{N}}$  such that for each $s \in \mathbb{N},$ we have $ P^{s} $ is a  $\mathsf{W}^{2}\mathrm{L}^{m, n-2m}$ bundle, trivialized over $\left\lbrace U_{\alpha}\right\rbrace_{\alpha \in I},$ which is not gauge equivalent to a flat bundle via $\mathsf{W}^{2}\mathrm{L}^{m, n-2m}$ gauges and $ A^{s}$ is a $\mathsf{W}^{1}\mathrm{L}^{m, n-2m}$ connection on $P^{s}$ such that 
		\begin{align*}
			\sup\limits_{\alpha \in I} \left\lVert A^{s}_{\alpha}\right\rVert_{\mathsf{W}^{1}\mathrm{L}^{m, n-2m}\left( U_{\alpha}; \Lambda^{2}\mathbb{R}^{n}\otimes \mathfrak{g} \right)} &\rightarrow 0.   
		\end{align*}
		Thus, using Theorem \ref{existence_Coulombgauges_small_connection}, there exists $s_{0} \in \mathbb{N}$ such that we have 
		\begin{align*}
			\left\lVert A^{s}_{\alpha}\right\rVert_{\mathsf{W}^{1}\mathrm{L}^{m, n-2m}\left( U_{\alpha}; \Lambda^{2}\mathbb{R}^{n}\otimes \mathfrak{g} \right)} &< \varepsilon_{C, \alpha}   \quad \text{ for all } \alpha \in I \text{ for all } s\geq s_{0}. 
		\end{align*}
		Thus, the corresponding Coulomb bundles exists. Now once again, arguing exactly as in the proof of Theorem \ref{equivanishing Morrey}, we can show that the limitimg Coulomb bundle $P^{\infty}$ exists and we have 
		\begin{align*}
			dg^{\infty}_{ij} = 0 \qquad \text{ for all } i, j, 
		\end{align*} 
		where  $\left\lbrace g^{\infty}_{ij} \right\rbrace_{i, j}$ are the transition functions for $P^{\infty}.$ As before, $P^{\infty}$ must be a flat bundle which is gauge equivalent to $P^{s}$ via $\mathsf{W}^{2}\mathrm{L}^{m, n-2m}$ gauges for $s$ large enough. This contradiction proves the result. 
	\end{proof}
	\begin{proof}[\textbf{Proof of Theorem \ref{cocycle factorization MorreySobolev}}]
		Using Theorem \ref{existence of connection}, we can deduce from Theorem \ref{flatness in vanishing} and Theorem \ref{small distance to flat} that the cocycle $\left\lbrace g_{\alpha\beta}\right\rbrace_{\alpha, \beta \in I}$ is $\mathsf{W}^{2}\mathrm{L}^{m, n-2m}$ gauge related to a cocycles of constant maps and the gauges are $\mathsf{W}^{2}\mathrm{VL}^{m, n-2m}$ if $g_{\alpha\beta}$s are. But since $M^{n}$ is simply connected, any cocycle of constant maps can be factorized by constant maps over simply connected manifolds. This is exactly how one proves that any flat bundle is isomorphic to the trivial bundle over simply connected manifolds ( see  \cite{Steenrod_fibrebundles} ). This implies the result. 
	\end{proof}
	
	\section*{Acknowledgments}
	The author warmly thanks Tristan Rivi\`{e}re for introducing him to this subject and numerous discussions, suggestions and encouragement.

\end{document}